\documentclass[11pt]{amsart}
\usepackage{amssymb,amsfonts,amsthm,amscd,amsmath,amsgen,upref}

\usepackage[margin=1in]{geometry}

\theoremstyle{plain}
\newtheorem{cor}{Corollary}

\newtheorem{prop}[cor]{Proposition}
\newtheorem{con}[cor]{Control}
\newtheorem{thm}[cor]{Theorem}
\theoremstyle{definition}

\numberwithin{cor}{section}
\numberwithin{equation}{section}

\DeclareMathOperator{\tr}{tr}

\DeclareMathOperator{\C}{C}

\DeclareMathOperator{\BUC}{BUC}
\DeclareMathOperator{\Lip}{Lip}

\DeclareMathOperator{\Supp}{Supp}

\renewcommand{\d}{d} 

\newcommand{\abs}[1]{\lvert#1\rvert}
\newcommand{\norm}[1]{\lVert#1\rVert}

\def\XXint#1#2#3{{\setbox0=\hbox{$#1{#2#3}{\int}$ }
\vcenter{\hbox{$#2#3$ }}\kern-.6\wd0}}

\title{On the Exit Time and Stochastic Homogenization of Isotropic Diffusions in Large Domains}

\author{Benjamin J. Fehrman$^\dagger$}
\thanks{$^\dagger$ This material is based upon work supported by the National Science Foundation Mathematical Sciences Postdoctoral Research Fellowship under Grant Number 1502731.}

\address{Max Planck Institute for Mathematics in the Sciences \\ Inselstra{$\ss$}e 22\\ 04103 Leipzig, Germany}
\email{fehrman@mis.mpg.de}

\date{27/2/2016}

\begin{document}

\maketitle

\begin{abstract}  Stochastic homogenization is achieved for a class of elliptic and parabolic equations describing the lifetime, in large domains, of stationary diffusion processes in random environment which are small, statistically isotropic perturbations of Brownian motion in dimension at least three.  Furthermore, the homogenization is shown to occur with an algebraic rate.  Such processes were first considered in the continuous setting by Sznitman and Zeitouni \cite{SZ}, upon whose results the present work relies strongly, and more recently their smoothed exit distributions from large domains were shown to converge to those of a Brownian motion by the author \cite{F3}.  This work shares in philosophy with \cite{F3}, but requires substantially new methods in order to control the expectation of exit times which are generically unbounded in the microscopic scale due to the emergence of a singular drift in the asymptotic limit.\end{abstract}

\section{Introduction}\label{introduction}

The purpose of this paper is to characterize, in dimensions greater than two, the lifetime of diffusion processes in large domains which are associated to generators of the form \begin{equation}\label{intro_gen}\frac{1}{2}\sum_{i,j=1}^da_{ij}(x,\omega)\frac{\partial^2}{\partial x_i\partial x_j}+\sum_{i=1}^db_i(x,\omega)\frac{\partial}{\partial x_i},\end{equation} where the uniformly elliptic diffusion matrix $A=(a_{ij})$ and drift $b=(b_i)$ are bounded, Lipschitz and describe a stationary, strongly mixing random environment, as indexed by an underlying probability space $(\Omega,\mathcal{F},\mathbb{P})$, which corresponds to a small, statistically isotropic perturbation of Brownian motion.

Precisely, the stationarity is quantified by a measure preserving transformation group $\left\{\tau_x\right\}$ of the probability space which satisfies, for each $x,y\in\mathbb{R}^d$ and $\omega\in\Omega$, \begin{equation}\label{intro_stationary} A(x+y,\omega)=A(x,\tau_y\omega)\;\;\textrm{and}\;\;b(x+y,\omega)=b(x,\tau_y\omega).\end{equation}  The coefficients are statistically isotropic in the sense that, for every orthogonal transformation $r$ of $\mathbb{R}^d$ which preserves the coordinate axes, for each $x\in\mathbb{R}^d$, the random variables \begin{equation}\label{intro_isotropy}(A(rx,\omega), b(rx,\omega))\;\;\textrm{and}\;\;(rA(x,\omega)r^t, rb(x,\omega))\;\;\textrm{have the same law.}\end{equation}  The environment is strongly mixing in the way of a finite range dependence.  Whenever subsets $A,B$ of $\mathbb{R}^d$ are sufficiently separated in space, the sigma algebras \begin{equation}\label{intro_independent}\sigma(A(x,\cdot), b(x,\cdot)\;|\;x\in A)\;\;\textrm{and}\;\;\sigma(A(x,\cdot), b(x,\cdot)\;|\;x\in B)\;\;\textrm{are independent.}\end{equation}  And, finally, there exists a constant $\eta>0$ to be chosen small such that, for every $x\in\mathbb{R}^d$ and $\omega\in\Omega$, \begin{equation}\label{intro_perturbation}\abs{A(x,\omega)-I}<\eta\;\;\textrm{and}\;\;\abs{b(x,\omega)}<\eta,\end{equation} which implies that the stochastic process determined by (\ref{intro_gen}) is a small perturbation of Brownian motion.  Such environments were first considered in the continuous setting by Sznitman and Zeitouni \cite{SZ}, and correspond to the analogue of the discrete framework studied by Bricmont and Kupiainen \cite{BK}.

The lifetime of these processes, for bounded domains $U$ of $\mathbb{R}^d$ satisfying an exterior ball condition, will be understood in terms of solutions to the associated elliptic equation \begin{equation}\label{intro_eq}\left\{\begin{array}{ll} \frac{1}{2}\tr(A(x,\omega)D^2v^\epsilon)+b(x,\omega)\cdot Dv^\epsilon=\epsilon^2 g(\epsilon x) & \textrm{on}\;\; U/\epsilon, \\ v^\epsilon=f(\epsilon x) & \textrm{on}\;\;\partial U/\epsilon,\end{array}\right.\end{equation}  which, writing $E_{x,\omega}$ for the expectation associated to the diffusion in environment $\omega$ beginning from $x$, and writing $\tau^\epsilon$ for the exit time from $U/\epsilon$, admit the representation \begin{equation}\label{intro_rep}v^\epsilon(x)=E_{x,\omega}(f(\epsilon X_{\tau^\epsilon})-\epsilon^2\int_0^{\tau^\epsilon}g(\epsilon X_s)\;ds)\;\;\textrm{on}\;\;\overline{U}/\epsilon.\end{equation}  Observe that the rescaled $u^\epsilon(x)=v^\epsilon(\frac{x}{\epsilon})$ satisfies \begin{equation}\label{intro_eq_res}\left\{\begin{array}{ll} \frac{1}{2}\tr(A(\frac{x}{\epsilon},\omega)D^2u^\epsilon)+\frac{1}{\epsilon}b(\frac{x}{\epsilon},\omega)\cdot Du^\epsilon=g(x) & \textrm{on}\;\; U, \\ u^\epsilon=f(x) & \textrm{on}\;\;\partial U, \end{array}\right.\end{equation} which in turn, following a change of variables in the final integral, admits the representation \begin{equation}\label{intro_rep_res} u^\epsilon(x)=v^\epsilon(\frac{x}{\epsilon})=E_{\frac{x}{\epsilon},\omega}(f(\epsilon X_{\frac{\epsilon^2 \tau^\epsilon}{\epsilon^2}})-\int_0^{\epsilon^2\tau^\epsilon}g(\epsilon X_{s/\epsilon^2})\;ds)\;\;\textrm{on}\;\;\overline{U},\end{equation}  where the stopping time $$\epsilon^2 \tau^\epsilon\;\;\textrm{quantifies the exit of the rescaled process}\;\;\epsilon X_{\frac{s}{\epsilon^2}}\;\;\textrm{from}\;\;U.$$

The limiting behavior of this rescaling was characterized in \cite{SZ}, where it was shown that, provided the perturbation $\eta$ in (\ref{intro_perturbation}) is sufficiently small, there exists a deterministic $\overline{\alpha}>0$ for which, on a subset of full probability, as $\epsilon\rightarrow 0$, \begin{equation}\label{intro_sz}\epsilon X_{\frac{s}{\epsilon^2}}\;\textrm{converges in law on}\;\mathbb{R}^d\;\textrm{to a Brownian motion with variance}\;\;\overline{\alpha}.\end{equation}  The goal here is to obtain the analogous result for the lifetime of such processes in large domains, and the result is stated in terms of the stochastic homogenization of (\ref{intro_eq_res}) for continuous data on the boundary and interior.

\begin{thm}\label{intro_main}  There exists a subset of full probability on which, for every bounded domain $U\subset\mathbb{R}^d$ satisfying an exterior ball condition, the solutions of (\ref{intro_eq}) converge uniformly on $\overline{U}$, as $\epsilon\rightarrow 0$, to the solution \begin{equation}\label{intro_main_1}\left\{\begin{array}{ll} \frac{\overline{\alpha}}{2}\Delta \overline{u}=g(x) & \textrm{on}\;\; U, \\ \overline{u}=f(x) & \textrm{on}\;\;\partial U.\end{array}\right.\end{equation}\end{thm}

Furthermore, the convergence is shown to occur with an algebraic rate.  The rate is first established for boundary data which is the restriction of a bounded, uniformly continuous function and interior data which is the restriction of a bounded, Lipschitz function.  \begin{equation}\label{intro_rate_assumption}\textrm{Assume}\;\;f\in\BUC(\mathbb{R}^d)\;\;\textrm{and}\;\;g\in \Lip(\mathbb{R}^d).\end{equation}  Writing $\sigma_f$ and $Dg$ for the respective moduli of continuity, as defined, for each $x,y\in\mathbb{R}^d$, by $$\abs{f(x)-f(y)}\leq \sigma_f(\abs{x-y})\;\;\textrm{and}\;\;\abs{g(x)-g(y)}\leq\norm{Dg}_{L^\infty(\mathbb{R}^d)}\abs{x-y},$$ the result is the following.

\begin{thm}\label{intro_rate}  Assume (\ref{intro_rate_assumption}).  There exists a subset of full probability and $c_1, c_2, c_3, c_4>0$ such that, for all $\epsilon>0$ sufficiently small depending upon $\omega$, the respective solutions $u^\epsilon$ and $\overline{u}$ of (\ref{intro_eq_res}) and (\ref{intro_main_1}) satisfy, for $C>0$ independent of $\omega$ and $\epsilon$, $$\norm{u^\epsilon-\overline{u}}_{L^\infty(\overline{U})}\leq C(\norm{f}_{L^\infty(\mathbb{R}^d)}\epsilon^{c_1}+\sigma_f(\epsilon^{c_2})+\norm{g}_{L^\infty(\mathbb{R}^d)} \epsilon^{c_3}+\norm{Dg}_{L^\infty(\mathbb{R}^d)}\epsilon^{c_4}).$$\end{thm}

Condition (\ref{intro_rate_assumption}) can be relaxed in the case that the domain is smooth via a standard extension argument or, in the case that $U=B_1$ is the ball, by an explicit radial construction. \begin{equation} \label{intro_rate_u}\textrm{Assume}\;f\in\C(\partial U), g\in\Lip(\overline{U})\;\textrm{and that the domain}\;U\;\textrm{is smooth}.\end{equation}  Then, the rate obtained in Theorem \ref{intro_rate} is preserved up to a domain dependent factor.

\begin{thm}\label{intro_rate_smooth}  Assume (\ref{intro_rate_u}).  There exists a subset of full probability, $c_1, c_2, c_3, c_4>0$ and $C_1=C_1(U)>0$ such that, for all $\epsilon>0$ sufficiently small depending upon $\omega$, the respective solutions $u^\epsilon$ and $\overline{u}$ of (\ref{intro_eq_res}) and (\ref{intro_main_1}) satisfy, for $C>0$ independent of $\omega$ and $\epsilon$, $$\norm{u^\epsilon-\overline{u}}_{L^\infty(\overline{U})}\leq C(\norm{f}_{L^\infty(\partial U)}\epsilon^{c_1}+\sigma_f(C_1\epsilon^{c_2})+\norm{g}_{L^\infty(\overline{U})} \epsilon^{c_3}+C_1\norm{Dg}_{L^\infty(\overline{U})}\epsilon^{c_4}).$$\end{thm}

The methods of this paper also apply to the analogous parabolic equation \begin{equation}\label{intro_parabolic}\left\{\begin{array}{ll} u_t^\epsilon=\frac{1}{2}\tr(A(\frac{x}{\epsilon},\omega)D^2u^\epsilon)+\frac{1}{\epsilon}b(\frac{x}{\epsilon},\omega)\cdot Du^\epsilon+g(x) & \textrm{on}\;\;U\times(0,\infty), \\ u^\epsilon=f(x) & \textrm{on}\;\;\overline{U}\times\left\{0\right\}\cup\partial U\times(0,\infty),\end{array}\right.\end{equation} whose solutions admit the representation $$u^\epsilon(x,t)=E_{\frac{x}{\epsilon},\omega}(f(\epsilon X_{(\epsilon^2\tau^\epsilon\wedge t)/\epsilon^2})+\int_0^{(\epsilon^2\tau^\epsilon\wedge t)}g(\epsilon X_\frac{s}{\epsilon^2})\;ds)\;\;\textrm{on}\;\;\overline{U}\times[0,\infty).$$  In this case, on a subset of full probability, the solutions of (\ref{intro_parabolic}) are shown to convergence, as $\epsilon\rightarrow 0$ and uniformly on $\overline{U}\times[0,\infty)$, to the solution \begin{equation}\label{intro_parabolic_hom}\left\{\begin{array}{ll} \overline{u}_t=\frac{\overline{\alpha}}{2}\Delta \overline{u}+g(x) & \textrm{on}\;\;U\times(0,\infty), \\ \overline{u}=f(x) & \textrm{on}\;\;\overline{U}\times\left\{0\right\}\cup\partial U\times(0,\infty).\end{array}\right.\end{equation}  Since the proof follows by combining the techniques used in this paper and the author's work \cite{F3}, the details are omitted.

\begin{thm}\label{intro_main_parabolic}  There exists a subset of full probability on which, for every bounded domain $U\subset\mathbb{R}^d$ satisfying an exterior ball condition, the respective solutions $u^\epsilon$ and $\overline{u}$ of (\ref{intro_parabolic}) and (\ref{intro_parabolic_hom}) satisfy $$\lim_{\epsilon\rightarrow 0}\norm{u^\epsilon-\overline{u}}_{L^\infty(\overline{U}\times[0,\infty))}=0.$$  Furthermore, the convergence occurs with an algebraic rate in exact analogy with Theorems \ref{intro_rate} and \ref{intro_rate_smooth}.\end{thm}

The essential novelty of this paper is to handle the case $g\neq 0$, since when $g=0$ the results of \cite{F3} proved, on a subset of full probability, as $\epsilon\rightarrow 0$, the solutions of (\ref{intro_eq_res}) converge uniformly on $\overline{U}$ to the solution $$\left\{\begin{array}{ll} \Delta\overline{u}=0 & \textrm{on}\;\; U, \\ \overline{u}=f & \textrm{on}\;\;\partial U.\end{array}\right.$$  The simplification is that, when dealing with merely the exit distribution, events of vanishing probability necessarily pose a vanishing threat.  Or, in terms of the analysis, solutions of (\ref{intro_eq_res}) are uniformly bounded in $\epsilon>0$, and satisfy the estimate $$\norm{u^\epsilon}_{L^\infty(\overline{U})}\leq \norm{f}_{L^\infty(\partial U)}\;\;\textrm{whenever}\;\;g=0.$$

In the case $g\neq 0$, it is not a priori obvious that even such $L^\infty$-estimates are obtainable, since the statistical isotropy (\ref{intro_isotropy}) imposes no symmetry, in general, on the quenched environments.  More precisely, in Section \ref{preliminaries} the diffusion beginning from $x$ in environment $\omega$ will be described in the space of continuous paths by a measure and expectation denoted respectively $$P_{x,\omega}\;\;\textrm{and}\;\;E_{x,\omega}.$$ It is manifestly not the case that these objects are, in any sense, translationally or rotationally invariant in space or that they are in any way symmetric.

The invariance implied by the stationarity (\ref{intro_stationary}) and isotropy (\ref{intro_isotropy}) is seen only after averaging with respect to the entire collection of environments.  That is, the annealed measures and expectations, which are defined as the semi-direct products $$\mathbb{P}_x=\mathbb{P}\ltimes P_{x,\omega}\;\;\textrm{and}\;\;\mathbb{E}_x=\mathbb{E}\ltimes E_{x,\omega},$$ do satisfy a translational and rotational invariance in the sense that, for all $x,y\in\mathbb{R}^d$, \begin{equation}\label{annealed} \mathbb{E}_{x+y}(X_t)=\mathbb{E}_y(x+X_t)=x+\mathbb{E}_y(X_t),\end{equation} and, for all orthogonal transformations $r$ preserving the coordinate axis, for every $x\in\mathbb{R}^d$, \begin{equation}\label{annealed1} \mathbb{E}_{x}(rX_t)=\mathbb{E}_{rx}(X_t).\end{equation}  While this fact plays an important role in \cite{SZ} to preclude, with probability one, the emergence of ballistic behavior of the rescaled process in the asymptotic limit, it does not yield an immediate control, with respect to the quenched expectations, for the exit time of the process from large domains.  And, therefore, does not readily imply that the solutions of (\ref{intro_eq_res}) are uniformly bounded as $\epsilon$ approaches zero.

The proof of Theorem \ref{intro_main} is founded strongly in the results of \cite{SZ}, which in particular establish, on scales of order $\frac{1}{\epsilon}$ in space and $\frac{1}{\epsilon^2}$ in time and with high probability, a comparison between solutions \begin{equation}\label{intro_outline_1}\left\{\begin{array}{ll} v^\epsilon_t=\tr(A(x,\omega)D^2v^\epsilon)+b(x,\omega)\cdot Dv^\epsilon & \textrm{on}\;\;\mathbb{R}^d\times(0,\infty), \\ v^\epsilon=f(\epsilon x) & \textrm{on}\;\;\mathbb{R}^d\times\left\{0\right\},\end{array}\right.\end{equation} and the solution of the homogenized problem \begin{equation}\label{intro_outline_2}\left\{\begin{array}{ll} \overline{v}^\epsilon_t=\frac{\overline{\alpha}}{2}\Delta\overline{v}^\epsilon & \textrm{on}\;\;\mathbb{R}^d\times(0,\infty), \\ \overline{v}^\epsilon=f(\epsilon x) & \textrm{on}\;\;\mathbb{R}^d\times\left\{0\right\},\end{array}\right.\end{equation} with respect to rescaled H\"older-norms defined in (\ref{prob_Holder}).  This comparison is used in Section \ref{section_coupling}, similar to its use in \cite[Proposition~3.1]{SZ} and later in \cite[Proposition~5.1]{F3}, to establish a global coupling, on larges scales in space and time and with high probability, between the diffusion in random environment associated to the generator \begin{equation}\label{intro_outline_3}\frac{1}{2}\sum_{i,j=1}^da_{ij}(x,\omega)\frac{\partial^2}{\partial x_i \partial x_j}+\sum_{i=1}^d b_i(x,\omega)\frac{\partial}{\partial x_i}\end{equation} and a Brownian motion with variance approximately $\overline{\alpha}$.  See Proposition \ref{couple_main} and, in particular, Corollary \ref{couple_cor}.

This coupling will be achieved along a discrete sequence of time steps which, while small with respect to the scale $\frac{1}{\epsilon^2}$, are typically insufficient to characterize the asymptotic behavior of solutions of (\ref{intro_eq}) due to the emergence of the singular in $\frac{1}{\epsilon}$ drift.  The difficulties are twofold.

First, the drift can trap the particle in the domain to create, in expectation, an exponentially in $\frac{1}{\epsilon}$ increasing exit time.  To counteract this, the probability that the exit time is large is first controlled by Proposition \ref{exit_main} in Section \ref{section_exit_time}, where the comparison between solutions of (\ref{intro_outline_1}) and (\ref{intro_outline_2}) is again used to obtain a preliminary tail estimate.  Essentially, it is shown that there exists a small $a_1>0$ and a constant $a_2>0$ such that, with high probability, for $\tau^\epsilon$ the exit time from $U/\epsilon$ and for $C>0$ independent of $\epsilon$, \begin{equation}\label{intro_outline_4}\sup_{x\in \overline{U}/\epsilon}P_{x,\omega}(\tau^\epsilon>\frac{1}{\epsilon^{2+a_1}})\leq C\epsilon^{a_2}.\end{equation}  Note that although this estimate is an improvement upon the generic behavior of processes associated to generators like (\ref{intro_outline_3}), it remains far from implying a uniform in $\epsilon$ control for the expectation of the rescaled exit times $\epsilon^2 \tau^\epsilon$ associated to the rescaled process in the original domain.

Second, the drift can repel the process from the boundary, and thereby make impossible the existence of barriers which are effective except at scales much smaller than $\epsilon$.  To overcome this, a proxy for a barrier is essentially obtained through the arguments of Section \ref{section_main}, see Propositions \ref{main_Brownian} and \ref{main_time}, by combining the coupling established in Corollary \ref{couple_cor} with estimates for the exit time of Brownian motion from Section \ref{Brownian_exit}.  It is here that the exterior ball condition is used most essentially, and the results follow from standard comparison techniques and an explicit formula for the exit time of Brownian motion in annular domains.  See Propositions \ref{disc_annulus} and \ref{disc_u}.

The primary argument of the paper comes in Theorem \ref{main_main} of Section \ref{section_main}, and a precise outline is presented between lines (\ref{main_outline_1}) and (\ref{main_outline_14}).  The idea is to introduce a discretely stopped version of the process, and to consider the corresponding discrete version of the representation (\ref{intro_rep}).  The efficacy of this approximation follows from localization estimates obtained in \cite{SZ}, see Control \ref{localization}, and the substitute for boundary barriers implied by Propositions \ref{main_Brownian} and \ref{main_time}.  The discrete proxy is then compared with the analogous approximation defined by a Brownian motion of variance $\overline{\alpha}$ using the coupling from Corollay \ref{couple_cor}.  Finally, the results from Section \ref{Brownian_exit} together with standard exponential estimates for Brownian motion allow for the recovery of the homogenized solution (\ref{intro_main_1}) from its discrete representation and thereby complete the proof.  The rate is presented in Section \ref{section_rate}, and the proof is a straightforward consequence of the methods used to prove Theorem \ref{main_main}.

Diffusion processes in the stationary ergodic setting were first considered in the case $b(x,\omega)=0$ by Papanicolaou and Varadhan \cite{PV1}.  Furthermore, in the case that (\ref{intro_eq}) can be rewritten in divergence form or in the case that $b(x,\omega)$ is divergence free or the gradient of a stationary field, such processes and various boundary value problems have been studied by Papanicolaou and Varadhan \cite{PV}, De Masi, Ferrari, Goldstein and Wick \cite{MFGW}, Kozlov \cite{Kozlov}, Olla \cite{Olla} and Osada \cite{Osada}.  However, outside of this framework, much less is understood.

In the continuous setting, the results of \cite{SZ}, which apply to the isotropic, perturbative regime described above, are the only available.  And, these have been more recently extended by the author in \cite{F2, F1, F3}.  In particular, the results of \cite{F3} prove that the exit distributions of such processes from large domains converge to that of a Brownian motion, a result which is the continuous analogue of work in the discrete setting by Bolthausen and Zeitouni \cite{BZ}, who characterized the exit distributions from large balls (so, taking $U=B_1$) of random walks in random environment which are small, isotropic perturbations of a simple random walk.  Their work was later refined by Baur and Bolthausen \cite{BB} under a somewhat less stringent isotropy assumption.

The almost-sure characterization of the exit time and the general homogenization statement contained in Theorem \ref{intro_main} remain open in the discrete case.  However, under the assumptions of \cite{BB}, and by using an additional quenched symmetry assumption along a single coordinate direction, Baur \cite{Baur} has obtained a quenched invariance principle analogous to (\ref{intro_sz}) and a characterization of the exit times from large balls (so, taking $U=B_1$).  The symmetry with respect to the quenched measures $P_{x,\omega}$ allows for the exit of the one-dimensional projection $X_t\cdot e_1$ to be estimated by standard Martingale methods, and yields an effective a priori control of the rescaled exit times $\epsilon^2\tau^\epsilon$.   Therefore, when dealing with the continuous analogue of such environments, many of the arguments in this paper can be simplified.

It should be noted that the techniques presented here differ substantially from \cite{BB,Baur,BZ}, which employ renormalization schemes to propagate estimates controlling the convergence of the exit law of the diffusion in random environment to the uniform measure on the boundary of the ball.  The arguments of this paper begin instead from the parabolic results of \cite{SZ}, and apply immediately to general domains.

The organization of the paper is as follows.  Section \ref{preliminaries} contains the notation and assumptions.  Section \ref{inductive} reviews those aspects of \cite{SZ} most relevant to this work and presents the primary probabilistic statement concerning the random environment.  The global coupling is presented in Section \ref{section_coupling} and a tail estimate for the exit time associated to the process in random environment is obtained in Section \ref{section_exit_time}.  Section \ref{Brownian_exit} controls the expectation of the exit time of Brownian motion near the boundary.   The proof of homogenization is presented in Section \ref{section_main} and the rate of convergence is established in Section \ref{section_rate}.

\section{Preliminaries}\label{preliminaries}

\subsection{Notation}

The elements of $\mathbb{R}^d$ and $[0,\infty)$ are written $x$ or $y$ and $t$ respectively and $(x,y)$ denotes the standard inner product.  The spacial gradient and derivative in time of a scalar function $v$ are written $Dv$ and $v_t$, while $D^2v$ denotes the the Hessian matrix.  The spaces of $k\times l$ and $k\times k$ symmetric matrices with real entries are written $\mathcal{M}^{k\times l}$ and $\mathcal{S}(k)$ respectively.  If $M\in\mathcal{M}^{k\times l}$, then $M^t$ is its transpose and $\abs{M}$ is the norm defined by $\abs{M}=\tr(MM^t)^{1/2}.$  The trace of a square matrix $M$ is written $\tr(M)$.  The distance between subsets $A,B\subset\mathbb{R}^d$ is $$d(A,B)=\inf\left\{\;\abs{a-b}\;|\;a\in A, b\in B\;\right\}$$ and, for an index $\mathcal{A}$ and a family of measurable functions $\left\{f_\alpha:\mathbb{R}^d\times\Omega\rightarrow\mathbb{R}^{n_\alpha}\right\}_{\alpha\in\mathcal{A}}$, the sigma algebra generated by the random variables $f_\alpha(x,\omega)$, for $x\in A$ and $\alpha\in\mathcal{A}$, is denoted $$\sigma(f_\alpha(x,\omega)\;|\;x\in A, \alpha\in\mathcal{A}).$$  For domains $U\subset\mathbb{R}^d$, $\BUC(U;\mathbb{R}^d)$, $\C(U;\mathbb{R}^d)$, $\Lip(U;\mathbb{R}^d)$, $\C^{0,\beta}(U;\mathbb{R}^d)$ and $\C^k(U;\mathbb{R}^d)$ are the spaces of bounded continuous, continuous, Lipschitz continuous, $\beta$-H\"{o}lder continuous and $k$-continuously differentiable functions on $U$ with values in $\mathbb{R}^d$.  Furthermore, $C^\infty_c(\mathbb{R}^d)$ denotes the space of smooth, compactly supported functions on $\mathbb{R}^d$.  The closure and boundary of $U\subset\mathbb{R}^d$ are denoted $\overline{U}$ and $\partial U$.  The support of a function $f:\mathbb{R}^d\rightarrow\mathbb{R}$ is written $\Supp(f)$.  The open balls of radius $R$ centered at zero and $x\in\mathbb{R}^d$ are respectively written $B_R$ and $B_R(x)$.  For a real number $r\in\mathbb{R}$, the notation $\left[r\right]$ denotes the largest integer less than or equal to $r$.  Finally, throughout the paper $C$ represents a constant which may change within a line and from line to line but is independent of $\omega\in\Omega$ unless otherwise indicated.

\subsection{The Random Environment}

A probability space $(\Omega,\mathcal{F},\mathbb{P})$ indexes the random environment, and the elements $\omega\in\Omega$ correspond to realizations described by the coefficients $A(\cdot,\omega)$ and $b(\cdot,\omega)$ on $\mathbb{R}^d$.  Their stationarity is quantified by an \begin{equation}\label{transgroup} \textrm{ergodic group of measure-preserving transformations}\; \left\{\tau_x:\Omega\rightarrow\Omega\right\}_{x\in\mathbb{R}^d}\end{equation} such that $A:\mathbb{R}^d\times\Omega\rightarrow\mathcal{S}(d)$ and $b:\mathbb{R}^d\times\Omega\rightarrow\mathbb{R}^d$ are bi-measurable stationary functions satisfying, for each $x,y\in\mathbb{R}^d$ and $\omega\in\Omega$, \begin{equation}\label{stationary} A(x+y,\omega)=A(x,\tau_y\omega)\;\;\textrm{and}\;\;b(x+y,\omega)=b(x,\tau_y\omega).\end{equation}

The diffusion matrix and drift are bounded, Lipschitz functions on $\mathbb{R}^d$ for each $\omega\in\Omega$.  There exists $C>0$ such that, for all $x\in\mathbb{R}^d$ and $\omega\in\Omega$,  \begin{equation}\label{bounded} \abs{b(x,\omega)}\leq C\;\;\;\textrm{and}\;\;\;\abs{A(x,\omega)}\leq C, \end{equation} and, for all $x,y\in\mathbb{R}^d$ and $\omega\in\Omega$, \begin{equation}\label{Lipschitz} \abs{b(x,\omega)-b(y,\omega)}\leq C\abs{x-y}\;\;\;\textrm{and}\;\;\;\abs{A(x,\omega)-A(y,\omega)}\leq C\abs{x-y}.\end{equation}  In addition, the diffusion matrix is uniformly elliptic.  There exists $\nu>1$ such that, for all $x\in\mathbb{R}^d$ and $\omega\in\Omega$, \begin{equation}\label{elliptic} \frac{1}{\nu} I\leq A(x,\omega)\leq \nu I.\end{equation}

The environment is strongly mixing in the sense that the coefficients satisfy a finite range dependence.  There exists $R>0$ such that, for every $A,B\subset\mathbb{R}^d$ satisfying $d(A,B)\geq R$, the sigma algebras \begin{equation}\label{finitedep} \sigma(A(x,\cdot), b(x,\cdot)\;|\;x\in A)\;\;\textrm{and}\;\;\sigma(A(x,\cdot), b(x,\cdot)\;|\;x\in B)\;\;\textrm{are independent.}\end{equation}  The environment is statistically isotropic in the sense that, for every orthogonal transformation $r:\mathbb{R}^d\rightarrow\mathbb{R}^d$ which preserves the coordinate axes, for every $x\in\mathbb{R}^d$, \begin{equation}\label{isotropy} (A(rx,\omega),b(rx,\omega))\;\;\textrm{and}\;\;(rA(x,\omega)r^t,rb(x,\omega))\;\;\textrm{have the same law.}\end{equation}  Finally, the diffusion is a small perturbation of Brownian motion.  There exists $\eta_0>0$, to be fixed small in line (\ref{constants}) of Section \ref{inductive}, such that, for all $x\in\mathbb{R}^d$ and $\omega\in\Omega$, \begin{equation}\label{perturbation} \abs{b(x,\omega)}\leq\eta_0\;\;\textrm{and}\;\;\abs{A(x,\omega)-I}\leq \eta_0.\end{equation}

The remaining two assumptions concern the domain.  First, the domain \begin{equation}\label{domain_bounded} U\subset\mathbb{R}^d\;\;\textrm{is open and bounded.}\end{equation}  And second, $U$ satisfies an exterior ball condition.  There exists $r_0>0$ so that, for each $x\in\partial U$ there exists $x^*\in\mathbb{R}^d$ satisfying \begin{equation}\label{exterior} \overline{B}_{r_0}(x^*)\cap \overline{U}=\left\{x\right\}.\end{equation}

To avoid lengthy statements, a steady assumption is made.  \begin{equation}\label{steady}\textrm{Assume}\;(\ref{transgroup}), (\ref{stationary}), (\ref{bounded}), (\ref{Lipschitz}), (\ref{elliptic}), (\ref{finitedep}), (\ref{isotropy}), (\ref{perturbation}),  (\ref{domain_bounded})\;\textrm{and}\;(\ref{exterior}).\end{equation}

Observe that (\ref{bounded}), (\ref{Lipschitz}) and (\ref{elliptic}) guarantee, for every environment $\omega\in\Omega$ and initial distribution $x\in\mathbb{R}^d$, the well-posedness of the martingale problem associated to the generator $$\frac{1}{2}\sum_{i,j=1}^da_{ij}(x,\omega)\frac{\partial^2}{\partial x_i\partial x_j}+\sum_{i=1}^db_i(x,\omega)\frac{\partial}{\partial x_i},$$ see Strook and Varadhan \cite[Chapter~6,7]{SV}.  The associated probability measure and expectation on the space of continuous paths $\C([0,\infty);\mathbb{R}^d)$ will be respectively denoted $P_{x,\omega}$ and $E_{x,\omega}$ where, almost surely with respect to $P_{x,\omega}$, paths $X_t\in\C([0,\infty);\mathbb{R}^d)$ satisfy the stochastic differential equation \begin{equation}\label{sde}\left\{\begin{array}{l} dX_t=b(X_t,\omega)dt+\sigma(X_t,\omega)dB_t, \\ X_0=x,\end{array}\right.\end{equation} for $A(x,\omega)=\sigma(x,\omega)\sigma(x,\omega)^t$, and for $B_t$ some standard Brownian motion under $P_{x,\omega}$ with respect to the canonical right-continuous filtration on $\C([0,\infty);\mathbb{R}^d)$.

As mentioned in the introduction, the translational and rotational invariance implied in law by (\ref{stationary}) and (\ref{isotropy}) do not imply any invariance properties, in general, for the quenched measures $P_{x,\omega}$.  However, the annealed measures and expectations do inherit these properties.  Precisely, defining the semi-direct product measures $\mathbb{P}_x=\mathbb{P}\ltimes P_{x,\omega}$ and $\mathbb{E}_x=\mathbb{E}\ltimes E_{x,\omega}$ on $\Omega\times\C([0,\infty);\mathbb{R}^d)$, for all $x,y\in\mathbb{R}^d$, \begin{equation}\label{annealed} \mathbb{E}_{x+y}(X_t)=\mathbb{E}_y(x+X_t)=x+\mathbb{E}_y(X_t),\end{equation} and, for every orthogonal transformation $r$ preserving the coordinate axis, for every $x\in\mathbb{R}^d$, \begin{equation}\label{annealed1} \mathbb{E}_{x}(rX_t)=\mathbb{E}_{rx}(X_t).\end{equation}  This fact plays an important role in \cite{SZ} to preclude, with probability one, the emergence of ballistic behavior of the rescaled process in the asymptotic limit.

Define as well, for each $n\geq 0$ and $x\in\mathbb{R}^d$, the Wiener measure $W^n_x$ and expectation $E^{W^n_x}$ on $\C([0,\infty);\mathbb{R}^d)$ corresponding to Brownian motion on $\mathbb{R}^d$ with variance $\alpha_n$ beginning from $x$.  Almost surely with respect to $W^n_x$, paths $X_t\in\C([0,\infty);\mathbb{R}^d)$ satisfy the stochastic differential equation \begin{equation}\label{sde_brownian}\left\{\begin{array}{l} dX_t=\sqrt{\alpha_n}dB_t, \\ X_0=x,\end{array}\right.\end{equation} for $B_t$ some standard Brownian motion under $W^n_x$ with respect to the canonical right-continuous filtration on $\C([0,\infty);\mathbb{R}^d)$.

\subsection{A Remark on Existence and Uniqueness}

The boundedness, Lipschitz continuity and ellipticity of the coefficients, see (\ref{bounded}), (\ref{Lipschitz}) and (\ref{elliptic}), together with the boundedness and regularity of the domain, see (\ref{domain_bounded}) and (\ref{exterior}), guarantee the well-posedness, for every $\omega\in\Omega$, of equations like $$\left\{\begin{array}{ll}\frac{1}{2}\tr(A(x,\omega)D^2w)+b(x,\omega)\cdot Dw=g(x) & \textrm{on}\;\;U, \\ u=f(x) & \textrm{on}\;\;\partial U,\end{array}\right.$$ for every $f\in C(\partial U)$ and $g\in\C(\overline{U})$ in the class of bounded continuous functions.  See, for instance, Friedman \cite[Chapter~3]{Fr}.  Furthermore, if $\tau$ denotes the exit time from $U$, then the solution admits the representation $$u(x)=E_{x,\omega}(f(X_\tau)-\int_0^\tau g(X_s)\;ds)\;\;\textrm{on}\;\;\overline{U},$$ see {\O}ksendal \cite[Exercise~9.12]{Oksendal}.

The same assumptions ensure the well-posedness of parabolic equations like $$\left\{\begin{array}{ll} w_t=\frac{1}{2}\tr(A(x,\omega)D^2w)+b(x,\omega)\cdot Dw & \textrm{on}\;\;\mathbb{R}^d\times(0,\infty), \\ w=f(x) & \textrm{on}\;\;\mathbb{R}^d\times\left\{0\right\},\end{array}\right.$$ for continuous initial data $f(x)$ satisfying, for instance and to the extent that it will be applied in this paper, $\abs{f(x)}\leq C(1+\abs{x}^2)$ on $\mathbb{R}^d$, in the class of continuous functions satisfying a quadratic estimate of the same form locally in time.  See \cite[Chapter~1]{Fr}.  Furthermore, $$w(x,t)=E_{x,\omega}(f(X_t))\;\;\textrm{on}\;\;\mathbb{R}^d\times(0,\infty),$$ see \cite[Exercise~9.12]{Oksendal}.

Analogous formulas hold for the constant coefficient elliptic and parabolic equations associated to Brownian motion and the measures $W^n_x$.  Since these facts are well-known, and since the solution to every equation encountered in this paper admits an explicit probabilistic description, the presentation will not further emphasize these points.

\section{The Inductive Framework and Probabilistic Statement}\label{inductive}

In this section, the aspects of \cite{SZ} most relevant to this work are briefly explained.  A complete description of the inductive framework can be found in \cite{SZ}, and it was later reviewed in the introduction of \cite{F1}.  

Assume the dimension $d$ satisfies \begin{equation}\label{dimension} d\geq 3, \end{equation} and fix a H\"older exponent \begin{equation}\label{Holderexponent} \beta\in\left(0,\frac{1}{2}\right]\;\;\textrm{and a scaling constant}\;\;a\in \left(0,\frac{\beta}{1000d}\right]. \end{equation}

The following constants will come to define the scales in length and time along which the induction scheme is propagated.  Let $L_0$ be integer multiple of five which will later be fixed large in (\ref{constants}).  For each $n\geq 0$, define inductively \begin{equation}\label{L} \ell_n=5\left[\frac{L_n^a}{5}\right]\;\;\textrm{and}\;\;L_{n+1}=\ell_n L_n, \end{equation} where it follows that, for every $L_0$ sufficiently large, $\frac{1}{2}L_n^{1+a}\leq L_{n+1}\leq 2L_n^{1+a}$.  For $c_0>0$ to be fixed small in (\ref{constants}), for each $n\geq 0$, define \begin{equation}\label{kappa} \kappa_n=\exp(c_0(\log\log(L_n))^2)\;\;\textrm{and}\;\;\tilde{\kappa}_n=\exp(2c_0(\log\log(L_n))^2),\end{equation} and observe that, as $n\rightarrow\infty$, the constants $\kappa_n$ are eventually dominated by every positive power of $L_n$.  Furthermore, for each $n\geq 0$, define\begin{equation}\label{D} D_n=L_n\kappa_n\;\;\textrm{and}\;\;\tilde{D}_n=L_n\tilde{\kappa}_n,\end{equation} where, using the preceding remark, the scales $D_n$ and $\tilde{D}_n$ are larger but grow comparably with the previously defined scales $L_n$.

The remaining constants enter into the primary probabilistic statement, see Theorem \ref{induction}, and the H\"older estimates governing the convergence of solutions to the parabolic equation (\ref{prob_eq}), see Theorem \ref{effectivediffusivity} and Control \ref{Holder}.  Fix $m_0\geq 2$ satisfying \begin{equation}\label{m0} (1+a)^{m_0-2}\leq 100<(1+a)^{m_0-1}, \end{equation}  and $\delta>0$ and $M_0>0$ satisfying \begin{equation}\label{delta} \delta=\frac{5}{32}\beta\;\;\textrm{and}\;\;M_0\geq100d(1+a)^{m_0+2}.\end{equation}  In what follows, it is essential that $\delta$ and $M_0$ are sufficiently larger than $a$.

In order to exploit the environment's mixing properties, it will be frequently necessary to introduce a stopped version of the process.  Define for every element $X_t\in \C([0,\infty);\mathbb{R}^d)$ the path \begin{equation}\label{prob_tail}X_t^*=\sup_{0\leq s\leq t}\abs{X_s-X_0},\end{equation} and, for each $n\geq 0$, the stopping time $$T_n=\inf\left\{s\geq 0\;|\;X_s^*\geq \tilde{D}_n\right\}.$$  The effective diffusivity of the ensemble at scale $L_n$ is defined by $$\alpha_n=\frac{1}{2d}\mathbb{E}_0\left(\abs{X_{L_n^2\wedge T_n}}^2\right),$$ where the localization ensures that the $\alpha_n$ are local quantities on scale $\tilde{D}_n$.  The convergence of the $\alpha_n$ to a limiting diffusivity $\overline{\alpha}$ is proven in \cite[Proposition~5.7]{SZ}.

\begin{thm}\label{effectivediffusivity} Assume (\ref{steady}).  There exists $L_0$ and $c_0$ sufficiently large and $\eta_0>0$ sufficiently small such that, for all $n\geq 0$, $$\frac{1}{2\nu}\leq \alpha_n\leq 2\nu\;\;\textrm{and}\;\;\abs{\alpha_{n+1}-\alpha_n}\leq L_n^{-(1+\frac{9}{10})\delta},$$  which implies the existence of $\overline{\alpha}>0$ satisfying $$\frac{1}{2\nu}\leq \overline{\alpha}\leq 2\nu\;\;\textrm{and}\;\;\lim_{n\rightarrow\infty}\alpha_n=\overline{\alpha}.$$\end{thm}

The results of \cite{SZ} obtain an effective comparison on the parabolic scale $(L_n, L_n^2)$ in space and time, with improving probability as $n\rightarrow\infty$, between solutions \begin{equation}\label{prob_eq} \left\{\begin{array}{ll} u_t=\frac{1}{2}\tr(A(x,\omega)D^2u)+b(x,\omega)\cdot Du & \textrm{on}\;\;\mathbb{R}^d\times(0,\infty), \\ u=f(x) & \textrm{on}\;\;\mathbb{R}^d\times\left\{0\right\},\end{array}\right.\end{equation} and solutions of the approximate limiting equation \begin{equation}\label{prob_approx}\left\{\begin{array}{ll} u_{n,t}=\frac{\alpha_n}{2}\Delta u_n & \textrm{on}\;\;\mathbb{R}^d\times(0,\infty), \\ u_n=f(x) & \textrm{on}\;\;\mathbb{R}^d\times\left\{0\right\}.\end{array}\right.\end{equation}  To simplify the notation, for each $n\geq 0$, define the operators \begin{equation}\label{prob_operators} R_nf(x)=u(x,L_n^2)\;\;\textrm{and}\;\;\overline{R}_nf(x)=u_n(x,L_n^2),\end{equation} and the difference \begin{equation}\label{prob_difference} S_nf(x)=R_nf(x)-\overline{R}_nf(x).\end{equation}

Since solutions of (\ref{prob_eq}) are not, in general, effectively comparable with solutions of (\ref{prob_approx}) globally in space, it is necessary to localize using a cutoff function.   For each $v>0$, define \begin{equation}\label{cutoff} \chi(y)=1\wedge(2-\abs{y})_+\;\;\textrm{and}\;\;\chi_{v}(y)=\chi\left(\frac{y}{v}\right), \end{equation}  and, for each $x\in\mathbb{R}^d$ and $n\geq 0$, \begin{equation}\label{cutoff1}  \chi_{n,x}(y)=\chi_{30\sqrt{d}L_n}(y-x).\end{equation}  Furthermore, since the comparison of the solutions must necessarily respect the scaling associated to (\ref{intro_eq}) and (\ref{intro_eq_res}), it is obtained with respect to the rescaled global H\"older-norms, defined for each $n\geq 0$, \begin{equation}\label{prob_Holder} \abs{f}_n=\norm{f}_{L^\infty(\mathbb{R}^d)}+\sup_{x\neq y}L_n^\beta\frac{\abs{f(x)-f(y)}}{\abs{x-y}^\beta}.\end{equation}  See, for instance, the introductions of \cite{F1,SZ} for a more complete discussion concerning the necessity of these norms as opposed, perhaps, to attempting a generically false $L^\infty$-contraction.

The following estimate is the statement propagated by the arguments of \cite{SZ}, and expresses a comparison between solutions of (\ref{prob_eq}) and (\ref{prob_approx}).  Observe that this statement is not true, in general, for all triples $x\in\mathbb{R}^d$, $\omega\in\Omega$ and $n\geq 0$.  However, as described in Theorem \ref{induction} below, it is shown in \cite[Proposition~5.1]{SZ} that such controls are available for large $n$, with high probability and on a large portion of space.

\begin{con}\label{Holder}  Fix $x\in\mathbb{R}^d$, $\omega\in\Omega$ and $n\geq 0$.  Then, for each $f\in C^{0,\beta}(\mathbb{R}^d)$, $$\abs{\chi_{n,x}S_nf}_n\leq L_n^{-\delta}\abs{f}_n.$$\end{con}

In order to account for the error introduced by localization, it is necessary to obtain tail-estimates for the diffusion in random environment.  Recall that $P_{x,\omega}$ is the measure on $\C([0,\infty);\mathbb{R}^d)$ describing the diffusion beginning from $x\in\mathbb{R}^d$ in environment $\omega$ and associated to the generator $$\frac{1}{2}\sum_{i,j=1}^da_{ij}(x,\omega)\frac{\partial^2}{\partial x_i \partial x_j}+\sum_{i=1}^db_i(x,\omega)\frac{\partial}{\partial x_i}.$$  The type of control propagated in \cite{SZ} is an exponential estimate for the probability, under $P_{x,\omega}$, that the maximal excursion $X^*_{L_n^2}$ defined in (\ref{prob_tail}) is large with respect to the time elapsed.

As with Control \ref{Holder}, it is simply false in general that this type of estimate is satisfied for every triple $(x,\omega,n)$.  However, it is shown in \cite[Proposition~2.2]{SZ} that such controls are available for large $n$, with high probability, on a large portion of space.

\begin{con}\label{localization}  Fix $x\in\mathbb{R}^d$, $\omega\in\Omega$ and $n\geq 0$.  For each $v\geq D_n$, for all $\abs{y-x}\leq 30\sqrt{d}L_n$, $$P_{y,\omega}(X^*_{L_n^2}\geq v)\leq \exp(-\frac{v}{D_n}).$$\end{con}

It was shown that, provided the perturbation $\eta_0$ is sufficiently small, Controls \ref{Holder} and \ref{localization} are available with high probability.  Precisely, define for each $n\geq 0$ and $x\in\mathbb{R}^d$, the event \begin{equation}\label{mainevent} B_n(x)=\left\{\;\omega\in\Omega\;|\;\textrm{Controls \ref{Holder} and \ref{localization} hold for the triple}\;(x,\omega,n).\;\right\},\end{equation}  and notice that, in view of (\ref{stationary}), for all $x\in\mathbb{R}^d$ and $n\geq 0$, \begin{equation}\label{mainevent1}\mathbb{P}(B_n(x))=\mathbb{P}(B_n(0)).\end{equation}  Furthermore, observe that $B_n(0)$ does not include the control of traps described in \cite[Proposition~3.3]{SZ}, which play an important role in propagating Control \ref{Holder}, and from which the arguments of this paper have no further need.  The following theorem proves that the compliment of $B_n(0)$ approaches zero as $n$ tends to infinity, see \cite[Theorem~1.1]{SZ}.

\begin{thm}\label{induction}  Assume (\ref{steady}).  There exist $L_0$ and $c_0$ sufficiently large and $\eta_0>0$ sufficiently small such that, for each $n\geq 0$, $$\mathbb{P}\left(\Omega\setminus B_n(0)\right)\leq L_n^{-M_0}.$$\end{thm}

Henceforth, the constants $L_0$, $c_0$ and $\eta_0$ are fixed to satisfy the requirements of Theorems \ref{effectivediffusivity} and \ref{induction}.  \begin{equation}\label{constants} \textrm{Fix constants}\;L_0, c_0\;\textrm{and}\;\eta_0\;\textrm{satisfying the hypothesis of Theorems \ref{effectivediffusivity} and \ref{induction}.}\end{equation}

The events which come to define, following an application of the Borel-Cantelli lemma, the event on which Theorem \ref{intro_main} is obtained are chosen to ensure that Controls \ref{Holder} and \ref{localization} are available at a sufficiently small scale in comparison to $\frac{1}{\epsilon}$.  Fix the smallest integer $\overline{m}>0$ satisfying \begin{equation}\label{prob_m} \overline{m}>1-\frac{\log(1-12a-a^2)}{\log(1+a)},\end{equation} and notice that the definition of $L_n$ in (\ref{L}) implies that, for $C>0$ independent of $n\geq\overline{m}$, $$L_{n+1}\tilde{D}_{n-\overline{m}}\leq CL_{n-1}^{2-10a}.$$  Observe as well that this definition is stronger than was necessary for the arguments of \cite{F3}.

Theorem \ref{induction} is now used to obtain Control \ref{Holder} and Control \ref{localization} at scale $L_{n-\overline{m}}$ on the entirety of the rescaled domain $U/\epsilon$ whenever $L_n\leq \frac{1}{\epsilon}<L_{n+1}$.  It follows from the boundedness of $U$ and the definition of $L_n$ that, for all $n\geq 0$ sufficiently large, whenever $L_n\leq\frac{1}{\epsilon}<L_{n+1}$, the rescaled domain $U/\epsilon$ is contained in what becomes the considerably larger set $[-\frac{1}{2}L_{n+2}^2, \frac{1}{2}L_{n+2}^2]^d$.  Therefore, define, for each $n\geq \overline{m}$, \begin{multline}\label{prob_event_1} A_n=\left\{\;\omega\in\Omega\;|\;\omega\in B_m(x)\;\;\textrm{for all}\;\;x\in L_m\mathbb{Z}^d\cap[-L_{n+2}^2, L_{n+2}^2]^d\;\;\textrm{and}\right. \\ \left.\textrm{for all}\;\;n-\overline{m}\leq m\leq n+2\right\}.\end{multline}  The following proposition proves that, as $n\rightarrow\infty$, the probability of the events $A_n$ rapidly approaches one, since the exponent $$2d(1+a)^2-\frac{M_0}{2}<0$$ is negative owing to (\ref{Holderexponent}) and (\ref{delta}).

\begin{prop}\label{prob_probability}  Assume (\ref{steady}) and (\ref{constants}).  For each $n\geq \overline{m}$, for $C>0$ independent of $n$, $$\mathbb{P}(\Omega\setminus A_n)\leq CL_n^{2d(1+a)^2-\frac{1}{2}M_0}.$$\end{prop}

\begin{proof}  Fix $n\geq\overline{m}$.  Theorem \ref{induction} implies using the definition of $L_n$ in (\ref{L}) that, for $C>0$ independent of $n$, $$\mathbb{P}(\Omega\setminus A_n)\leq \sum_{m=n-\overline{m}}^{n+2}(\frac{L_{n+2}^2}{L_{m}})^dL_m^{-M_0}\leq C\sum_{m=n-\overline{m}}^{n+2}L_n^{2d(1+a)^2-2d(1+a)^{m-n}-M_0(1+a)^{m-n}}.$$  Therefore, $$\mathbb{P}(\Omega\setminus A_n)\leq CL_n^{2d(1+a)^2-M_0(1+a)^{-\overline{m}}},$$ which, since the definition of $\overline{m}$ implies that $$(1+a)^{-\overline{m}}\geq (1+a)(\frac{2-10a}{1+a}-(1+a))\geq\frac{1}{2},\;\;\textrm{yields}\;\;\mathbb{P}(\Omega\setminus A_n)\leq CL_n^{2d(1+a)^2-\frac{M_0}{2}}$$ and completes the proof.\end{proof}

\section{The Global Coupling}\label{section_coupling}

The purpose of this section is to construct with high probability a coupling between the diffusion in random environment and a Brownian motion with variance $\alpha_{n-\overline{m}}$.  This will be achieved along the discrete sequence of time steps $\left\{kL_{n-\overline{m}}^2\right\}$ through the comparison implied by Control \ref{Holder}.   The choice of $\overline{m}$ in (\ref{prob_m}) is made to ensure that, for scales $L_n\leq \frac{1}{\epsilon}<L_{n+1}$, the subsequent discretization on scale $L_{n-\overline{m}}$ provides a sufficiently accurate description of the continuous process.  Notice, however, that it is not obvious a priori that such a discretization exists, since, for generic equations like (\ref{couple_eq}), it is necessary to apply a discretization on vanishing scales in the large domains to accurately represent the exit times and distributions in the asymptotic limit.

The coupling is motivated by the observation that the vector-valued solutions of the parabolic equation \begin{equation}\label{couple_eq}\left\{\begin{array}{ll} u_t=\frac{1}{2}\tr(A(x,\omega)D^2u)+b(x,\omega)\cdot Du & \textrm{on}\;\;\mathbb{R}^d\times(0,\infty), \\ u=\frac{x}{L_{n-\overline{m}}} & \textrm{on}\;\;\mathbb{R}^d\times\left\{0\right\},\end{array}\right.\end{equation} and the approximate homogenized equation \begin{equation}\label{couple_approx} \left\{\begin{array}{ll} u_{n,t}=\frac{\alpha_{n-\overline{m}}}{2}\Delta u_n & \textrm{on}\;\;\mathbb{R}^d\times(0,\infty), \\ u_n=\frac{x}{L_{n-\overline{m}}} & \textrm{on}\;\;\mathbb{R}^d\times\left\{0\right\},\end{array}\right.\end{equation} may be compared using Control \ref{Holder} which yields, following an application of Control \ref{localization} to localize the initial data, and due to the choice of constants in (\ref{L}) and (\ref{kappa}), \begin{equation}\label{couple_approx_100}\abs{u(0,L_{n-\overline{m}}^2)-u_n(0,L_{n-\overline{m}}^2)}=\abs{E_{0,\omega}(\frac{1}{L_{n-\overline{m}}}X_{L_{n-\overline{m}}^2})-E^{W_0^{n-\overline{m}}}(\frac{1}{L_{n-\overline{m}}}X_{L_{n-\overline{m}}^2})}\leq C\tilde{\kappa}_{n-\overline{m}} L_{n-\overline{m}}^{-\delta},\end{equation} where $W^{n-\overline{m}}_x$ is the Wiener measure on $\C([0,\infty);\mathbb{R}^d)$ corresponding to Brownian motion with variance $\alpha_{n-\overline{m}}$ beginning from $x$.

It follows formally that, provided (what will be discrete) copies of the diffusion in random environment $\tilde{X}_t$ and Brownian motion $\tilde{B}_t$ are chosen carefully and are defined with respect to the same measure on an auxiliary probability space $(\tilde{\Omega},\tilde{\mathcal{F}},\tilde{\mathbb{P}})$, a Chebyshev inequality will yield $$(\frac{\gamma}{L_{n-\overline{m}}})^\beta\tilde{\mathbb{P}}(\abs{\tilde{X}_{L_{n-\overline{m}}^2}-\tilde{B}_{L_{n-\overline{m}}^2}}^\beta\geq \gamma^\beta)\leq CL_{n-\overline{m}}^{-\delta}\tilde{\kappa}_{n-\overline{m}},$$ which implies \begin{equation}\label{couple_goal} \tilde{\mathbb{P}}(\abs{\tilde{X}_{L_{n-\overline{m}}^2}-\tilde{B}_{L_{n-\overline{m}}^2}}\geq \gamma)\leq CL_{n-\overline{m}}^{-\delta}\tilde{\kappa}_{n-\overline{m}}(\frac{L_{n-\overline{m}}}{\gamma})^\beta.\end{equation}  An application of the Kantorovich-Rubinstein theorem, see (\ref{couple_kr}), will justify the commutation of absolute value and integration appearing between (\ref{couple_approx_100}) and (\ref{couple_goal}).

Recall that solutions of (\ref{couple_eq}) with initial condition $f(x)$ admit a representation using the Green's function $$p_{t,\omega}(x,y):[0,\infty)\times\mathbb{R}^d\times\mathbb{R}^d\rightarrow\mathbb{R},$$ which represents the density of the diffusion beginning from $x$ in environment $\omega$ at time $t$, taking the form $$u(x,t)=E_{x,\omega}(f(X_t))=\int_{\mathbb{R}^d}p_{t,\omega}(x,y)f(y)\;dy.$$  See \cite[Chapter~1]{Fr} for a detailed discussion of the existence and regularity of these densities, and which follow from assumptions (\ref{bounded}), (\ref{Lipschitz}) and (\ref{elliptic}).  Analogously, solutions of (\ref{couple_approx}) with initial data $f(x)$ admit the heat kernel representation $$\overline{u}_n(x,t)=E^{W_x^{n-\overline{m}}}(f(X_t))=\int_{\mathbb{R}^d}(4\pi\alpha_{n-\overline{m}} t)^{-\frac{d}{2}}\exp(-\frac{\abs{y-x}^2}{4\alpha_{n-\overline{m}} t})f(y)\;dy.$$  The Kantorovich-Rubinstein theorem will be applied to compare the density of the diffusion in random environment against the heat kernel of variance $\alpha_{n-\overline{m}}$.

The Kantorovich-Rubinstein theorem, see Dudley \cite[Theorem~11.8.2]{D}, states that any pair of probability measures $\nu$ and $\nu'$ on $\mathbb{R}^d$ assigning finite mass to a given metric $d$, in the sense that \begin{equation}\label{couple_kr_1} \int_{\mathbb{R}^d}d(x,0)\;\nu(dx)<\infty\;\;\textrm{and}\;\;\int_{\mathbb{R}^d}d(x,0)\;\nu'(dx)<\infty,\end{equation} satisfy \begin{multline}\label{couple_kr} D(\nu,\nu')=\sup\left\{\abs{\int f\;d\nu-\int f\;d\nu'}\;|\;\abs{f(x)-f(y)}\leq d(x,y)\;\;\textrm{on}\;\;\mathbb{R}^d\times\mathbb{R}^d\right\} \\ =\inf\left\{\int_{\mathbb{R}^d\times\mathbb{R}^d}d(x,x')\;\rho(dx,dx')\;|\;\rho\;\textrm{is a probability measure on}\;\mathbb{R}^d\times\mathbb{R}^d\right. \\ \left. \textrm{with first and second marginals}\;\nu\;\textrm{and}\;\nu'\right\}.\end{multline}  The function $D(\cdot,\cdot)$ is referred to as the Kantorovich-Rubinstein or Wasserstein metric associated to $d$.

To ease the notation define, for each $n\geq 0$, $$p_{n-\overline{m},\omega}(x,y)=p_{L_{n-\overline{m}}^2,\omega}(x,y),$$ and the heat kernel $$\overline{p}_{n-\overline{m}}(x,y)=(4\pi\alpha_{n-\overline{m}} L_{n-\overline{m}}^2)^{-\frac{d}{2}}\exp(-\frac{\abs{y-x}^2}{4\alpha_{n-\overline{m}} L_{n-\overline{m}}^2}).$$  The following proposition constructs a Markov process $(X_k,\overline{X}_k)$ on the space $(\mathbb{R}^d\times\mathbb{R}^d)^{\mathbb{N}}$ such that the transition probabilities of first coordinate $X_k$ are determined by $p_{n-\overline{m},\omega}(\cdot,\cdot)$ and, such that those of the second coordinate $\overline{X}_k$ are determined by $\overline{p}_{n-\overline{m}}(\cdot, \cdot)$.  Furthermore, the difference $\abs{X_k-\overline{X_k}}$ satisfies a version of (\ref{couple_goal}) with respect to the underlying measure, where this comparison is obtained using the Kantorovich-Rubinstein Theorem applied to the metrics $$d_{n-\overline{m}}(x,y)=\abs{\frac{x-y}{L_{n-\overline{m}}}}^\beta.$$  The proof is omitted, since it appears in full as \cite[Proposition~5.1]{F3}, and represents only a small reformulation of \cite[Proposition~3.1]{SZ}.

Looking forward, keep in mind that the coupling will be applied to scales $L_{n}\leq \frac{1}{\epsilon}<L_{n+1}$, and it therefore follows from Proposition \ref{exit_main} of Section \ref{section_exit_time} that the coupling estimates do not decay prior to a point before which the diffusion has exited the domain with overwhelming probability.

\begin{prop}\label{couple_main}  Assume (\ref{steady}) and (\ref{constants}).  For every $\omega\in\Omega$, for every $x\in\mathbb{R}^d$, there exists a measure $Q_{n,x}$ on the canonical sigma algebra of the space $(\mathbb{R}^d\times\mathbb{R}^d)^{\mathbb{N}}$ such that, under $Q_{n,x}$, the coordinate processes $X_k$ and $\overline{X}_k$ respectively have the law of a Markov chain on $\mathbb{R}^d$, starting from $x$, with transition kernels $p_{n-\overline{m},\omega}(\cdot,\cdot)$ and $\overline{p}_{n-\overline{m}}(\cdot,\cdot)$.

Furthermore, for every $n\geq \overline{m}$, $\omega\in A_n$ and $x\in[-\frac{1}{2}L_{n+2}^2, \frac{1}{2}L_{n+2}^2]^d$, for $C>0$ independent of $n$, \begin{equation}\label{couple_main_0}Q_{n,x}(\abs{X_k-\overline{X}_k}\geq \gamma\;|\;\textrm{for some}\;0\leq k\leq 2(\frac{L_{n+2}}{L_{n-\overline{m}}})^2)\leq C(\frac{L_{n-\overline{m}}}{\gamma})^\beta(\frac{L_{n+2}}{L_{n-\overline{m}}})^4\tilde{\kappa}_{n-\overline{m}}L_{n-\overline{m}}^{-\delta}.\end{equation}\end{prop}

The following Corollary follows immediately by choosing $\gamma=L_{n-\overline{m}}$ in Proposition \ref{couple_main}.  The corresponding exponent is a consequence of the definition of $\overline{m}$ in (\ref{prob_m}), which implies  $$(1+a)^{\overline{m}+2}-1\leq \frac{(1+a)^3}{2-10a-(1+a)^2}-1\leq \frac{8a}{2}=4a,$$ and therefore, using the definition of $L_n$ in (\ref{L}), for $C>0$ independent of $n$, $$(\frac{L_{n+2}}{L_{n-\overline{m}}})^4\leq C L_{n-\overline{m}}^{16a}.$$  Notice that definitions (\ref{Holderexponent}) and (\ref{delta}) imply the exponent $$16a-\delta<0$$ is negative.

\begin{cor}\label{couple_cor}  Assume (\ref{steady}) and (\ref{constants}).  For every $n\geq \overline{m}$, $\omega\in A_n$ and $x\in[-\frac{1}{2}L_{n+2}^2, \frac{1}{2}L_{n+2}^2]^d$, for $C>0$ independent of $n$, $$Q_{n,x}(\abs{X_k-\overline{X}_k}\geq L_{n-\overline{m}}\;|\;\textrm{for some}\;0\leq k\leq 2(\frac{L_{n+2}}{L_{n-\overline{m}}})^2)\leq C\tilde{\kappa}_{n-\overline{m}}L_{n-\overline{m}}^{16a-\delta}.$$\end{cor}

\section{Tail Estimates and an Upper Bound in Expectation for the Exit Time}\label{section_exit_time}

The purpose of this section is to obtain certain tail estimates for the exit time in probability.  Namely, whenever the scale $\epsilon$ satisfies $L_n\leq \frac{1}{\epsilon}<L_{n+1}$, the diffusion associated to the generator $$\frac{1}{2}\sum_{i,j=1}^d a_{ij}(x,\omega)\frac{\partial^2}{\partial x_i\partial x_j}+\sum_{i=1}^db_i(x,\omega)\frac{\partial}{\partial x_i}$$ is shown to exit the rescaled domain $U/\epsilon$ prior to time $L_{n+2}^2$ in overwhelming fashion.  The corresponding estimate is then propagated inductively forward in time.  Observe, however, that these estimates remain far from the ultimate goal, since the exit time of a Brownian motion from the rescaled domain $U/\epsilon$ is expected to be of order $\frac{1}{\epsilon^2}$ which, as $n$ approaches infinity, is much smaller than $L_{n+2}^2$.  Therefore, Proposition \ref{exit_main} alone does not imply the boundedness to solutions to the rescaled equation (\ref{intro_eq_res}), and this will not be achieved until Theorem \ref{main_main} of Section \ref{section_main}.

The essential elements in the following proof are Control \ref{Holder} and the boundedness of the domain.  The latter allows for the exit time from $U/\epsilon$ to be bounded above by the exit time from $B_{R/\epsilon}$ for a sufficiently large radius.  And, the former ensures that, whenever $L_n\leq \frac{1}{\epsilon}<L_{n+1}$, on the event $A_n$ defined in (\ref{prob_event_1}), the exit of the random diffusion is comparable with that of a Brownian motion.  The following argument is similar to \cite[Proposition~4.1]{F3}, but the estimate is made more precise in $\epsilon$ and is subsequently iterated inductively.  Notice that the dimension $d\geq 3$ appears in the argument and conclusion, a fact that will later be important in the proof of Theorem \ref{main_main}.

\begin{prop}\label{exit_main}  Assume (\ref{steady}) and (\ref{constants}).  For all $n$ sufficiently large, for every $\omega\in A_n$, for all $\epsilon>0$ satisfying $L_n\leq \frac{1}{\epsilon}<L_{n+1}$, for $C>0$ independent of $n$, $$\sup_{x\in \overline{U}}P_{\frac{x}{\epsilon},\omega}(\tau^\epsilon>L_{n+2}^2)\leq C (\epsilon L_{n+2})^{-3}.$$  And, for each $k\geq 0$, $$\sup_{x\in \overline{U}}P_{\frac{x}{\epsilon},\omega}(\tau^\epsilon>kL_{n+2}^2)\leq C (\epsilon L_{n+2})^{-3k}.$$\end{prop}

\begin{proof}  Using the boundedness of the domain, choose $R\geq 1$ such that $\overline{U}\subset B_R$ and choose $n_1\geq 0$ so that, whenever $n\geq n_1$, \begin{equation}\label{exit_main_1}L_{n+1}\overline{U}\subset L_{n+1}B_R\subset[-\frac{1}{2}L_{n+2}^2, \frac{1}{2}L_{n+2}^2]^d.\end{equation} Henceforth, fix $n\geq n_1$, $\omega\in A_n$ and $L_n\leq \frac{1}{\epsilon}<L_{n+1}$.

Fix a smooth cutoff function satisfying $0\leq \chi_{B_R}\leq 1$ with $$\chi_{B_{R}}(x)=\left\{\begin{array}{ll} 1 & \textrm{if}\;\;x\in \overline{B}_R, \\ 0 & \textrm{if}\;\;x\in \mathbb{R}^d\setminus B_{R+1},\end{array}\right.$$ and observe that, for a constant $C>0$ independent of $\epsilon>0$, whenever $L_n\leq \frac{1}{\epsilon}<L_{n+1}$, \begin{equation}\label{exit_main_2} \abs{\chi_{B_R}(\epsilon x)}_{n+2}\leq C(1+\frac{L_{n+2}}{L_n})\leq CL_n^{2a+a^2}.\end{equation}  Because the solutions $$\left\{\begin{array}{ll} v^\epsilon_t=\frac{1}{2}\tr(A(x,\omega)D^2v^\epsilon)+b(x,\omega)\cdot Dv^\epsilon & \textrm{on}\;\;\mathbb{R}^d\times(0,\infty), \\ v^\epsilon=\chi_{B_R}(\epsilon x) & \textrm{on}\;\;\mathbb{R}^d\times\left\{0\right\},\end{array}\right.$$ admit the representation $$v^\epsilon(x,t)=E_{x,\omega}(\chi_{B_R}(\epsilon X_t))\geq P_{x,\omega}(\epsilon X_t\in B_R)\geq P_{x,\omega}(\epsilon X_t \in U),$$ it follows that \begin{equation}\label{exit_main_4}1-v^\epsilon(x,t)\leq P_{x,\omega}(\epsilon X_t\notin U)\leq P_{x,\omega}(\tau^\epsilon\leq t).\end{equation}

The solutions $v^\epsilon$ are now compared using Control \ref{Holder} with the solution $$\left\{\begin{array}{ll} \overline{v}^\epsilon_t=\frac{\alpha_{n+2}}{2}\Delta \overline{v}^\epsilon & \textrm{on}\;\;\mathbb{R}^d\times(0,\infty), \\ \overline{v}^\epsilon=\chi_{B_R}(\epsilon x) & \textrm{on}\;\;\mathbb{R}^d\times\left\{0\right\}.\end{array}\right.$$  Since $\omega\in A_n$ and (\ref{exit_main_1}) guarantee that Control \ref{Holder} is available for every $x\in U/\epsilon$, assumptions (\ref{Holderexponent}) and (\ref{L}) and line (\ref{exit_main_2}) imply that, for $C>0$ independent of $n$, \begin{equation}\label{exit_main_3} \sup_{x\in \overline{U}}\abs{v^\epsilon(x,L_{n+2}^2)-\overline{v}^\epsilon(x,L_{n+2}^2)}\leq CL_{n+2}^{-\delta}L_n^{2a+a^2}\leq CL_n^{2a+a^2-\delta(1+a)^2}\leq CL_n^{3a-\delta}.\end{equation}

To conclude with the first statement, the size of $\overline{v}^\epsilon(x,L_{n+2}^2)$ is bounded using Theorem \ref{effectivediffusivity} and the heat kernel.  For each $x\in \overline{U}/\epsilon$, since $R\geq 1$, for $C>0$ independent of $n$, \begin{equation}\label{exit_main_7}\overline{v}^\epsilon(x,L_{n+2}^2)\leq \int_{B_{\frac{4R}{\epsilon}}(x)}(4\pi \alpha_{n+2}L_{n+2}^2)^{-\frac{d}{2}}\exp(-\frac{\abs{y-x}^2}{4\alpha_{n+2} L_{n+2}^2})\;dy\leq C(\epsilon L_{n+2})^{-d}.\end{equation}  And, in view of (\ref{exit_main_3}), for each $x\in \overline{U}/\epsilon$, for $C>0$ independent of $n$, \begin{equation}\label{exit_main_5} 1-v^\epsilon(x,L_{n+2}^2)\geq 1-\overline{v}^\epsilon(x,L_{n+2}^2)-\abs{v^\epsilon(x,L_{n+2}^2)-\overline{v}^\epsilon(x,L_{n+2}^2)}\geq 1-C(\epsilon L_{n+2})^{-d}-CL_n^{3a-\delta}.\end{equation}  Therefore, since the definitions (\ref{Holderexponent}) and (\ref{delta}) and $L_n\leq\frac{1}{\epsilon}<L_{n+1}$ imply that, for $C>0$ independent of $n$, $$L_n^{3a-\delta}\leq C(\epsilon L_{n+2})^{-3},$$ it follows using inequality (\ref{exit_main_4}) that, since $d\geq 3$ and $x\in \overline{U}/\epsilon$ was arbitrary, for $C>0$ independent of $n$, \begin{equation}\label{exit_main_6} \sup_{x\in \overline{U}}P_{\frac{x}{\epsilon},\omega}(\tau^\epsilon>L_{n+2}^2)\leq C(\epsilon L_{n+2})^{-d}+CL_n^{3a-\delta}\leq C(\epsilon L_{n+2})^{-3},\end{equation} which completes the argument for the first statement.  

The final statement is a consequence of the Markov property and induction.  The case $k=0$ is immediate, and the case $k=1$ is (\ref{exit_main_6}).  For the inductive step, assume that, for $k\geq 1$, $$\sup_{x\in \overline{U}}P_{\frac{x}{\epsilon},\omega}(\tau^\epsilon>kL_{n+2}^2)\leq C(\epsilon L_{n+2})^{-3k}.$$  Then by the Markov property, for each $x\in \overline{U}$, $$P_{\frac{x}{\epsilon},\omega}(\tau^\epsilon>(k+1)L_{n+2}^2)=E_{\frac{x}{\epsilon},\omega}(P_{X_{kL_{n+2}^2},\omega}(\tau^\epsilon>L_{n+2}^2), \tau^\epsilon>kL_{n+2}^2).$$  Therefore, from the inductive hypothesis and (\ref{exit_main_6}), for each $x\in \overline{U}$, $$P_{\frac{x}{\epsilon},\omega}(\tau^\epsilon>(k+1)L_{n+2}^2)\leq (\sup_{x\in \overline{U}}P_{\frac{x}{\epsilon},\omega}(\tau^\epsilon>L_{n+2}^2))P_{x,\omega}(\tau^\epsilon>kL_{n+2}^2)\leq C(\epsilon L_{n+2})^{-3(k+1)},$$ which completes the argument. \end{proof}

\section{Estimates for the Exit Time of Brownian Motion Near the Boundary}\label{Brownian_exit}

In this section estimates are obtained, in expectation and near the boundary of the domain, for the exit time of Brownian motion.  These estimates are shown in Section \ref{section_main} to be inherited with high probability by the diffusion in random environment using the coupling developed in Section \ref{section_coupling}.  The exterior ball condition plays its most essential role in this section, which states that for a now fixed $r_0>0$, for every $x\in\partial U$, there exists $x^*\in\mathbb{R}^d$ satisfying \begin{equation}\label{disc_ball} \overline{B}_{r_0}(x^*)\cap\overline{U}=\left\{x\right\}.\end{equation}  Observe that the following material is very similar to \cite[Proposition~6.1]{F3} and \cite[Corollary~6.1, 6.2]{F3}, though the presentation has been condensed, and is included for the reader's convenience and because it illustrates the primary function of the exterior ball condition.

Define, for each $\delta>0$, the enlargement \begin{equation}\label{disc_inflate} U_{\delta}=\left\{\;x\in\mathbb{R}^d\;|\;d(x,U)< \delta\;\right\},\end{equation} and notice, as a consequence of (\ref{disc_ball}), for every $0<\delta<r_0$, \begin{equation}\label{disc_inflate_ball} U_{\delta}\;\;\textrm{satisfies the exterior ball condition with radius}\;(r_0-\delta).\end{equation}  To begin, Proposition \ref{disc_annulus} and \ref{disc_u} analyze the behavior of Brownian motion in the original domains $U$ and $U_\delta$, and in Corollary \ref{disc_u_scale} the statements are rescaled to obtain estimates for the dilated domains $U/\epsilon$ and $U_\delta/\epsilon$.

The exit time of Brownian motion will first be understood in annular regions about origin defined, for each pair of radii $0<r_1<r_2<\infty$, by $$A_{r_1, r_2}=B_{r_2}\setminus\overline{B}_{r_1}.$$  Let $\tau_{r_1,r_2}$ denote the $\C([0,\infty);\mathbb{R}^d)$ exit time $$\tau_{r_1,r_2}=\inf\left\{\;t\geq 0\;|\;X_t\notin A_{r_1,r_2}\;\right\},$$ and recall that, in expectation and with respect to the Wiener measure $W^{n}_x$, the function $$u^n_{r_1,r_2}(x)=E^{W^n_x}(\tau_{r_1,r_2})\;\;\textrm{on}\;\;\overline{A}_{r_1,r_2}$$ satisfies the equation \begin{equation}\label{disc_annulus_exit}\left\{\begin{array}{ll} 1+\frac{\alpha_n}{2}\Delta u^n_{r_1,r_2}=0 & \textrm{on}\;\; A_{r_1,r_2}, \\ u^n_{r_1,r_2}=0 & \textrm{on}\;\;\partial A_{r_1,r_2}.\end{array}\right.\end{equation}  See, for example, \cite[Exercise~9.12]{Oksendal}.  An upper bound for these solutions is now effectively obtained in a neighborhood of $\partial B_{r_1}$ which necessarily depends upon the pair $(r_1,r_2)$.   However, in the application to follow, the exterior ball conditions (\ref{disc_ball}) and (\ref{disc_inflate_ball}) will allow the radii to be fixed independently of $n\geq 0$.

\begin{prop}\label{disc_annulus}  Assume (\ref{steady}) and (\ref{constants}).  For each pair of radii $0<r_1<r_2<\infty$, for each $n\geq 0$, there exists $C=C(r_1,r_2)>0$ such that $$u^n_{r_1,r_2}(x)\leq C\d(x,\partial B_{r_1})\;\;\textrm{on}\;\;\overline{A}_{r_1,r_2}.$$\end{prop}

\begin{proof}  Fix $0<r_1<r_2<\infty$ and $n\geq 0$.  The solution $u^n_{r_1,r_2}$ of (\ref{disc_annulus_exit}) admits the explicit radial description, owing to $d\geq 3$, and writing $r=\abs{x}$, $$u(x)=u(r)=c_1(r_1,r_2)+c_2(r_1,r_2)r^{2-d}-\frac{r^2}{2d\alpha_n}\;\;\textrm{on}\;\;A_{r_1,r_2},$$ for $$c_1(r_1,r_2)=\frac{1}{2d\alpha_n}\cdot \frac{r_1^2r_2^{2-d}-r_2^2r_1^{2-d}}{r_2^{2-d}-r_1^{2-d}}\;\;\textrm{and}\;\;c_2(r_1,r_2)=\frac{1}{2d\alpha_n}\cdot\frac{r_2^2-r_1^2}{r_2^{2-d}-r_1^{2-d}}.$$  After performing a Taylor expansion in $r$ about $r_1$ and using the fact that $u(r_1)=0$, for each $x\in A_{r_1,r_2}$, $$u^n_{r_1,r_2}(x)=u^n_{r_1,r_2}(r)=c_2(2-d)r_1^{1-d}(r-r_1)+c_2(2-d)(1-d)\int_{r_1}^rs^{-d}(r-s)\;ds-(\frac{2r_1r+r^2}{2d\alpha_n}).$$  Since the integrand is bounded by $r_1^{-d}(r_2-r_1)$, and because the final term is negative, the uniform control of $\alpha_n$ provided by Theorem \ref{effectivediffusivity} guarantees the existence of $C=C(r_1,r_2)>0$ satisfying $$u^n_{r_1,r_2}(x)=u^n_{r_1,r_2}(r)\leq C(r-r_1)=C\d(x,\partial B_{r_1})\;\;\textrm{on}\;\;\overline{A}_{r_1,r_2},$$ and completes the argument.\end{proof}

The comparison principle will now imply that the estimates obtained on $A_{r_1,r_2}$ induce similar estimates near the boundary of the domains $U$ and its inflations $U_\delta$, whenever $\delta>0$ is sufficiently small.  Define the translated annuli, for each $x\in\mathbb{R}^d$ and pair $(r_1,r_2)$, $$A_{r_1,r_2}(x)=x+A_{r_1,r_2}=B_{r_2}(x)\setminus\overline{B}_{r_1}(x).$$  And, for each $\delta>0$, define the $\C([0,\infty);\mathbb{R}^d)$ exit times \begin{equation}\label{disc_stopping} \tau=\inf\left\{\;t\geq 0\;|\;X_t\notin U\;\right\}\;\;\textrm{and}\;\;\tau^\delta=\inf\left\{\; t\geq 0\;|\;X_t\notin U_{\delta}\;\right\}.\end{equation}  The following corollary of Proposition \ref{disc_annulus} controls the expectation of $\tau$ and $\tau^\delta$ in what is essentially the $\delta$-neighborhood of the respective boundaries of $U$ and $U_\delta$.  The radius $r_0$ defined in (\ref{disc_ball}) appears in the argument to quantify the exterior ball condition.

\begin{cor}\label{disc_u} Assume (\ref{steady}) and (\ref{constants}).  For every $0<\delta<\frac{r_0}{2}$, for every $n\geq 0$, for $C>0$ independent of $n$ and $\delta$, $$\sup_{\d(x,\partial U)\leq \delta} E^{W^n_{x}}(\tau)\leq C\delta\;\;\textrm{and}\;\;\sup_{\d(x,\partial U_\delta)\leq 2\delta}E^{W^n_x}(\tau^\delta)<C\delta.$$\end{cor}

\begin{proof}  For each $0<\delta<\frac{r_0}{2}$ observation (\ref{disc_inflate_ball}) implies that $U_{\delta}$ satisfies the exterior ball condition with radius $r_0-\delta$.  Fix $r_2>\frac{r_0}{2}$ such that, whenever $x\in\partial U_\delta$ and $x^*\in\mathbb{R}^d$ satisfy $$\overline{B}_{r_0-\delta}(x^*)\cap\overline{U}_\delta=\left\{x\right\},\;\;\textrm{it follows that}\;\;\overline{U}_\delta\subset B_{r_2}(x^*).$$  The existence of $r_2$ chosen uniformly for $0<\delta<\frac{r_0}{2}$ is guaranteed by the boundedness of $U$.  Since, for each $x\in U$, the stopping time $\tau_{\delta}$ almost-surely bounds $\tau$ with respect to $W^n_x$, the second statement for the inflated domains $U_\delta$ implies the statement for $U$.

Consider $0<\delta<\frac{r_0}{2}$ and $n\geq 0$ and, in Proposition \ref{disc_annulus}, choose $r_1=r_0-\delta$.  The choice of $\delta>0$ guarantees $\frac{r_0}{2}\leq r_1\leq r_0$, and the pair $(r_1,r_2)$ guarantees, for every $x\in\partial U_\delta$ and $x^*\in\mathbb{R}^d$ satisfying \begin{equation}\label{disc_u_1}\overline{B}_{r_0-\delta}(x^*)\cap\overline{U}_\delta=\left\{x\right\},\;\textrm{the containment}\;\;\overline{U}_\delta\subset \overline{A}_{r_1,r_2}(x^*).\end{equation}

Suppose $x\in U_\delta$ satisfies $\d(x, \partial U_\delta)\leq 2\delta$ and choose using compactness $\overline{x}\in\partial U_\delta$ with $\abs{x-\overline{x}}=\d(x, \partial U_\delta)$.  Let $\overline{x}^*$ satisfy (\ref{disc_u_1}) with $\overline{x}$, and let $u^{n,\overline{x}}_{r_1,r_2}$ denote the solution \begin{equation}\label{disc_u_3}\left\{\begin{array}{ll} 1+\frac{\alpha_n}{2}\Delta u^{n,\overline{x}}_{r_1,r_2}=0 & \textrm{on}\;\;A_{r_1,r_2}(\overline{x}^*), \\ u^{n,\overline{x}}_{r_1,r_2}=0 & \textrm{on}\;\;\partial A_{r_1,r_2}(\overline{x}^*).\end{array}\right.\end{equation}  Translational invariance and Proposition \ref{disc_annulus} imply, since $r_1$ and $r_2$ are bounded above and away from zero uniformly in $0<\delta<\frac{r_0}{2}$ and $n\geq0$, for $C>0$ independent of $n$ and $\delta$, \begin{equation}\label{disc_u_2}u^{n,\overline{x}}_{r_1,r_2}(x)\leq C\d(x,\partial A_{r_1,r_2}(\overline{x}^*))\leq C\abs{x-\overline{x}}=C\d(x, \partial U_\delta)\leq C\delta.\end{equation}

Finally, the expectation $E^{W^n_x}(\tau^\delta)$ on $\overline{U}_\delta$ satisfies $$\left\{\begin{array}{ll} 1+\frac{\alpha_n}{2}\Delta E^{W^n_x}(\tau^\delta)=0 & \textrm{on}\;\; U_\delta,\\ E^{W^n_x}(\tau^\delta)=0 & \textrm{on}\;\;\partial U_\delta,\end{array}\right.$$ see \cite[Exercise~9.12]{Oksendal}, and using the containment (\ref{disc_u_1}), the non-negativity of $u^{n,\overline{x}}_{r_1,r_2}$ on $\overline{A}_{r_1,r_2}$ and (\ref{disc_u_3}), $$\left\{\begin{array}{ll} 1+\frac{\alpha_n}{2}\Delta u^{n,\overline{x}}_{r_1,r_2}=0 & \textrm{on}\;\; U_\delta,\\ u^{n,\overline{x}}_{r_1,r_2}\geq0 & \textrm{on}\;\;\partial U_\delta.\end{array}\right.$$  The comparison principle and (\ref{disc_u_2}) therefore imply, for $C>0$ independent of $n$ and $\delta$, $$E^{W^n_x}(\tau^\delta)\leq u^{n,\overline{x}}_{r_1,r_2}(x)\leq C\delta,$$ and complete the proof.\end{proof}

The analogous estimates on the domains $U/\epsilon$ and $U_\delta/\epsilon$ now follow immediately by rescaling.   For each $\epsilon>0$, define the $\C([0,\infty);\mathbb{R}^d)$ exit time \begin{equation}\label{disc_epsilon_exit}\tau^\epsilon=\inf\left\{\;t\geq 0\;|\;X_t\notin U/\epsilon\right\}=\inf\left\{\;t\geq 0\;|\;\epsilon X_t\notin U\right\},\end{equation}  whose expectation $E^{W^n_x}(\tau^\epsilon)$ can be obtained as the rescaling $$E^{W^n_x}(\tau^\epsilon)=\epsilon^{-2}E^{W^n_{\epsilon x}}(\tau)\;\;\textrm{on}\;\;\overline{U}/\epsilon.$$  And, for each $\epsilon>0$ and $\delta>0$, the exit time $$\tau^{\epsilon,\delta}=\inf\left\{\;t\geq 0\;|\;X_t\notin U_\delta/\epsilon\;\right\}=\inf\left\{\;t\geq 0\;|\;\epsilon X_t\notin U_\delta\;\right\},$$ has expectation equal to the rescaling $$E^{W^n_x}(\tau^{\epsilon,\delta})=\epsilon^{-2}E^{W^n_{\epsilon x}}(\tau^\delta)\;\;\textrm{on}\;\;\overline{U}_\delta/\epsilon.$$  These two equalities and Corollary \ref{disc_u} then immediately imply the following.

\begin{cor}\label{disc_u_scale}  Assume (\ref{steady}) and (\ref{constants}).  For every $\epsilon>0$, $0<\delta<\frac{r_0}{2\epsilon}$ and $n\geq 0$, for $C>0$ independent of $\epsilon$, $\delta$ and $n$, $$\sup_{\d(x,\partial U/\epsilon)\leq \delta} E^{W^n_{x}}(\tau^\epsilon)\leq C\epsilon^{-1}\delta\;\;\textrm{and}\;\;\sup_{\d(x,\partial U_{\delta}/\epsilon)\leq 2\delta}E^{W^n_x}(\tau^{\epsilon, \delta})<C\epsilon^{-1}\delta.$$\end{cor}

\section{The Discrete Approximation and Proof of Homogenization}\label{section_main}

In this section, stochastic homogenization is established for solutions \begin{equation}\label{main_eq}\left\{\begin{array}{ll} \frac{1}{2}\tr(A(\frac{x}{\epsilon},\omega)D^2u^\epsilon)+\frac{1}{\epsilon}b(\frac{x}{\epsilon},\omega)\cdot Du^\epsilon=g(x) & \textrm{on}\;\;U, \\ u^\epsilon=f(x) & \textrm{on}\;\;\partial U,\end{array}\right.\end{equation} which are, on a subset of full probability, shown to converge uniformly on $\overline{U}$ as $\epsilon\rightarrow 0$ to the solution \begin{equation}\label{main_eq_hom}\left\{\begin{array}{ll} \frac{\overline{\alpha}}{2}\Delta \overline{u}=g(x) & \textrm{on}\;\;U, \\ \overline{u}=f(x) & \textrm{on}\;\;\partial U.\end{array}\right.\end{equation}  The result will be obtained by analyzing the lifetime of the diffusion process associated to the generator \begin{equation}\label{end_gen}\frac{1}{2}\sum_{i,j=1}^da_{ij}(x,\omega)\frac{\partial^2}{\partial x_i \partial x_j}+\sum_{i=1}^db_i(x,\omega)\frac{\partial}{\partial x_i}\end{equation} in the large domains $U/\epsilon$.

The discrete coupling developed in Section \ref{section_coupling} will play an essential role in the proof, and suggests the introduction of a discretely stopped version of the diffusion.  Namely, whenever the scale satisfies $L_n\leq \frac{1}{\epsilon}<L_{n+1}$, a discrete version of the process with time steps $L_{n-\overline{m}}^2$ will be considered, and stopped as soon as it hits the $\tilde{D}_{n-\overline{m}}$ neighborhood of the compliment of the dilated domain $U/\epsilon$.

Note carefully, however, that this type of discrete approximation does not generally provide an accurate description of processes associated to generators like (\ref{intro_gen}), since a continuous diffusion beginning in the $\tilde{D}_{n-\overline{m}}$ neighborhood of the boundary may be compelled by the drift to exit the domain in a region far removed from the stopping point of its discrete proxy.  A fact which can readily be seen by considering a nonzero, constant drift, and which is a situation that can occur within this framework with a rapidly vanishing but nonzero probability on all scales.  In essence, therefore, Propositions \ref{main_Brownian} and \ref{main_time} of this section effectively establish a boundary barrier for equation (\ref{main_eq}) of a quality which is generically impossible to obtain.

The discrete $\C([0,\infty);\mathbb{R}^d)$ stopping time is defined, for each $\epsilon>0$, and for each $n\geq \overline{m}$, by \begin{equation}\label{end_tau1} \tau^{\epsilon,n}_1=\inf\left\{\;kL_{n-\overline{m}}^2\geq 0\;|\;\d(X_{kL_{n-\overline{m}}^2},(U/\epsilon)^c)\leq \tilde{D}_{n-\overline{m}}\;\right\},\end{equation} and represents the first time $X_{kL_{n-\overline{m}}^2}$ enters the $\tilde{D}_{n-\overline{m}}$ neighborhood of the compliment of $U/\epsilon$.  Since it is not true that $\tau^{\epsilon,n}\leq \tau^\epsilon$ for every path $X_t$, the failure of this inequality will need to be controlled in probability with respect to $P_{x,\omega}$ by the exponential localization estimate implied by Control \ref{localization}.

Similarly, for each $\epsilon>0$, and for each $n\geq \overline{m}$, define the $\C([0,\infty);\mathbb{R}^d)$ stopping times \begin{equation}\label{end_tau2} \tau^{\epsilon,n}_2=\inf\left\{\;kL_{n-\overline{m}}^2\geq 0\;|\;\d(X_{kL_{n-\overline{m}}^2},(U/\epsilon))\geq \tilde{D}_{n-\overline{m}}\;\right\}.\end{equation}  These stopping times quantify the first time that the discrete process $X_{kL_{n-\overline{m}}^2}$ exits the $\tilde{D}_{n-\overline{m}}$ neighborhood of $(U/\epsilon)$.  The definitions imply $\tau^{\epsilon,n}_1\leq \tau^{\epsilon,n}_2$ and, whenever $\tau^{\epsilon,n}_1\leq \tau^\epsilon$, it is immediate that $\tau^{\epsilon,n}_1\leq \tau^\epsilon\leq \tau^{\epsilon,n}_2.$

Proposition \ref{main_Brownian} will use Corollary \ref{disc_u_scale} to obtain an effective tail estimate with respect to the Wiener measure $W^n_x$ for $\tau^{\epsilon,n}_2$ near the boundary of $U/\epsilon$.  This estimate, together with the coupling constructed in Proposition \ref{couple_main}, then yield on the event $A_n$ an upper bound for the probability $$P_{x,\omega}(\tau^\epsilon-\tau_1^{\epsilon,n}\geq L_{n-1}^2)\;\;\textrm{for}\;\;x\in\overline{U}/\epsilon.$$  It is this estimate that effectively acts as a barrier by ensuring that, with high probability and following an application of the exponential estimate implied by Control \ref{localization}, a diffusion beginning in the $\tilde{D}_{n-\overline{m}}$ neighborhood of the compliment $(U/\epsilon)^c$ exits the true domain $U/\epsilon$ in a small neighborhood of its starting position when compared with the scaling in $\epsilon$.

In what follows, recall that $\overline{m}$ is the smallest integer satisfying \begin{equation}\label{main_mbar}\overline{m}>1-\frac{\log(1-12a-a^2)}{\log(1+a)},\end{equation} which ensures that, by the choice of constants $L_n$ in (\ref{L}) and $\tilde{D}_n$ in (\ref{D}), for $C>0$ independent of $n\geq\overline{m}$, \begin{equation}\label{main_zeta_1}L_{n+1}\tilde{D}_{n-\overline{m}}\leq CL_{n-1}^{2-10a}.\end{equation}  Further, observe by using the definitions of $L_n$ in (\ref{L}) and $\tilde{\kappa}_n$ in (\ref{kappa}) that there exists $C>0$ independent of $n\geq\overline{m}$ satisfying \begin{equation}\label{main_zeta_2}\tilde{\kappa}_{n-\overline{m}}L_{n-\overline{m}}^{16a-\delta}\leq CL_{n-1}^{-10a}.\end{equation}  The following proposition is the control of the second discrete exit time, in terms of Brownian motion and near the boundary of $U/\epsilon$.  The proof is a refinement of the estimate obtained in \cite[Proposition~7.1]{F3} and follows from Chebyshev's inequality and standard exponential estimates for Brownian motion.

\begin{prop}\label{main_Brownian}  Assume (\ref{steady}) and (\ref{constants}).  For all $n\geq 0$ sufficiently large, for every $\epsilon>0$ satisfying $L_n\leq\frac{1}{\epsilon}< L_{n+1}$, for $C>0$ independent of $n$, $$\sup_{\d(x,(U/\epsilon)^c)\leq 2\tilde{D}_{n-\overline{m}}}W^{n-\overline{m}}_x(\tau_2^{\epsilon,n}\geq L_{n-1}^2)\leq CL_{n-1}^{-10a}.$$\end{prop}

\begin{proof}  Fix $n_1\geq 0$ such that, whenever $n\geq n_1$, for $r_0$ from the exterior ball condition (\ref{disc_ball}), \begin{equation}\label{main_Brownian_1} 2\tilde{D}_{n-\overline{m}}\leq \frac{r_0L_n}{2}.\end{equation}  And, therefore, whenever $n\geq n_1$ and $\d(x,(U/\epsilon)^c)\leq 2\tilde{D}_{n-\overline{m}}$, the conditions of Proposition \ref{disc_u_scale} are satisfied.

Fix $n\geq n_1$, $\epsilon>0$ satisfying $L_n\leq \frac{1}{\epsilon}<L_{n+1}$ and $x\in \mathbb{R}^d$ such that $d(x,(U/\epsilon)^c)\leq2\tilde{D}_{n-\overline{m}}$.  The stopping time $\tau^{\epsilon,\delta}$ is the exit time from the $\delta$-neighborhood of $(U/\epsilon)$, and for $\delta=2\tilde{D}_{n-\overline{m}}$ Proposition \ref{disc_u_scale} states, for $C>0$ independent of $n$, $$E^{W^n_x}(\tau^{\epsilon,2\tilde{D}_{n-\overline{m}}})\leq C(2\tilde{D}_{n-\overline{m}})\epsilon^{-1}< C\tilde{D}_{n-\overline{m}}L_{n+1}.$$  So, using observation (\ref{main_zeta_1}), for $C>0$ independent of $n$, $$E^{W^n_x}(\tau^{\epsilon,2\tilde{D}_{n-\overline{m}}})\leq CL_{n-1}^{2-10a}.$$  And therefore, by Chebyshev's inequality, for $C>0$ independent of $n$, \begin{equation}\label{main_Brownian_2} W^n_x(\tau^{\epsilon,2\tilde{D}_{n-\overline{m}}}\geq \frac{1}{2}L_{n-1}^2)\leq CL_{n-1}^{-10a}.\end{equation}

The argument if finished by applying the translational invariance of the heat kernel, the Markov property, and standard exponential tail estimates for Brownian motion on scale $\tilde{D}_{n-\overline{m}}$, see Revuz and Yor \cite[Chapter~2, Proposition~1.8]{RY}.  It follows from exponential estimates and the definitions (\ref{L}), (\ref{kappa}) and (\ref{D}) that, for $C>0$ and $c>0$ independent of $n$, \begin{equation}\label{main_Brownian_3}W^n_x(\tau^{\epsilon,2\tilde{D}_{n-\overline{m}}}+L_{n-\overline{m}}^2\leq \tau^{\epsilon,n}_2)\leq W_0^n(X^*_{L_{n-\overline{m}}^2}\geq \tilde{D}_{n-\overline{m}})\leq C\exp(-c\tilde{\kappa}_{n-\overline{m}}^2).\end{equation}  And, since (\ref{L}), (\ref{kappa}) and (\ref{D}) guarantee the existence of $C>0$ independent of $n\geq\overline{m}$ satisfying $$\exp(-c\tilde{\kappa}_{n-\overline{m}}^2)\leq CL_{n-1}^{-10a},$$ and since $L_{n-\overline{m}}^2<\frac{1}{2}L_{n-1}^2$ for all $n$ sufficiently large, the combination (\ref{main_Brownian_2}) and (\ref{main_Brownian_3}) imply $$W^n_x(\tau^{\epsilon,n}_2\geq L_{n-1}^2)\leq CL_{n-1}^{-10a},$$ which, because $d(x,(U/\epsilon)^c)\leq2\tilde{D}_{n-\overline{m}}$, $L_n\leq \frac{1}{\epsilon}<L_{n+1}$ and $n\geq n_1$ were arbitrary, completes the argument.  \end{proof}

Before proceeding, recall the events $A_n$ defined, for each $n\geq 0$, as \begin{multline}\label{main_event}A_n=\left\{\;\omega\in\Omega\;|\;\omega\in B_m(x)\;\;\textrm{for all}\;\;x\in L_m\mathbb{Z}^d\cap[-L_{n+2}^2, L_{n+2}^2]^d\;\;\textrm{and}\right. \\ \left. \textrm{for all}\;\;n-\overline{m}\leq m\leq n+2.\right\},\end{multline} and which guarantee in particular the localization estimate implied by Control \ref{main_localization} for every environment $\omega\in A_n$, point $x\in[-L_{n+2}^2, L_{n+2}^2]^d$ and scale $L_{n-\overline{m}}$ to $L_{n+2}$.

\begin{con}\label{main_localization}  Fix $x\in\mathbb{R}^d$, $\omega\in\Omega$ and $n\geq 0$.  For each $v\geq D_n$, for all $\abs{y-x}\leq 30\sqrt{d}L_n$, $$P_{y,\omega}(X^*_{L_n^2}\geq v)\leq \exp(-\frac{v}{D_n}).$$\end{con}

The following establishes, on the event $A_n$ with respect to $P_{x,\omega}$, a comparison between the continuous exit time $\tau^\epsilon$ and discrete stopping time $\tau^{\epsilon,n}_1$.   The estimate will be achieved on scales $\epsilon$ satisfying $L_n\leq \frac{1}{\epsilon}<L_{n+1}$ for all $n$ sufficiently large.  The presentation is an improvement of \cite[Proposition~7.3]{F3}, and the proof is a consequence of the global coupling from Corollary \ref{couple_cor} and the estimates for Brownian motion from Proposition \ref{main_Brownian}.  As before, notice that the dimension $d\geq 3$ appears in the conclusion due to its reliance upon Proposition \ref{exit_main}.

\begin{prop}\label{main_time}  Assume (\ref{steady}) and (\ref{constants}).  For each $n\geq \overline{m}$ sufficiently large, for every $\epsilon>0$ satisfying $L_n\leq\frac{1}{\epsilon}<L_{n+1}$ and for every $\omega\in A_n$, for $C>0$ independent of $n$, $$\sup_{x\in \overline{U}/\epsilon}P_{x,\omega}(\tau^\epsilon-\tau^{\epsilon,n}_1\geq L_{n-1}^2)\leq CL_{n-1}^{-10a}+C(\epsilon L_{n+2})^{-3}.$$\end{prop}

\begin{proof}  Let $n_1\geq 0$ be as in Proposition \ref{main_Brownian}.  Namely, for each $n\geq n_1$, $$2\tilde{D}_{n-\overline{m}}\leq \frac{r_0L_n}{2},$$ which ensures that the assumptions of Proposition \ref{main_Brownian} are satisfied for every $n\geq n_1$.  Furthermore, let $n_2\geq 0$ be such that, whenever $n\geq n_2$, $$L_{n+1}\overline{U}\subset[-\frac{1}{2}L_{n+2}^2, \frac{1}{2}L_{n+2}^2]^d,$$ which ensures, for every $n\geq n_2$ and $L_n\leq \frac{1}{\epsilon}<L_{n+2}$, the containment $\overline{U}/\epsilon\subset[-\frac{1}{2}L_{n+2}^2, \frac{1}{2}L_{n+2}^2]^d$ and therefore, for every $x\in\overline{U}/\epsilon$, the statement of Corollary \ref{couple_cor}.

Fix $n\geq \max(n_1, n_2, \overline{m})$, $\epsilon>0$ satisfying $L_n\leq\frac{1}{\epsilon}<L_{n+1}$, $\omega\in A_n$ and $x\in \overline{U}/\epsilon$.  Recall the measure $Q_{n,x}$ from Proposition \ref{couple_main} which defines the Markov chain $(X_k,\overline{X}_k)$ on $(\mathbb{R}^d\times\mathbb{R}^d)^{\mathbb{N}}$, and which is, in its respective coordinates, a discrete version of the process in random environment and a Brownian motion with variance $\alpha_{n-\overline{m}}$ with time steps $L_{n-\overline{m}}^2$.  Define $C_n$ to be the event $$C_n=(\abs{X_k-\overline{X}_k}\geq L_{n-\overline{m}}\;|\;\textrm{for some}\;0\leq k\leq 2(\frac{L_{n+2}}{L_{n-\overline{m}}})^2\;),$$ where Corollary \ref{couple_cor} and (\ref{main_zeta_2}) imply, for $C>0$ independent of $n$, \begin{equation}\label{main_time_1}Q_{n,x}(C_n)\leq C\tilde{\kappa}_{n-\overline{m}}L_{n-\overline{m}}^{16a-\delta}\leq CL_{n-1}^{-10a}.\end{equation}

Let $\tilde{\tau}^\epsilon$ denote the discrete $\C([0,\infty);\mathbb{R}^d)$ stopping time \begin{equation}\label{main_time_2} \tilde{\tau}^\epsilon=\inf\left\{\;kL_{n-\overline{m}}^2\geq0\;|\;X_{kL_{n-\overline{m}}^2}\notin U/\epsilon\;\right\},\end{equation} and define the analogous stopping times $$T^{\epsilon,n}_1=\inf\left\{\;k\geq 0\;|\;(X_k,\overline{X}_k)\;\textrm{satisfies}\;d(X_k,(U/\epsilon)^c)\leq \tilde{D}_{n-\overline{m}}\;\right\},$$ which is merely $\tau^{\epsilon,n}_1$ defined for the first coordinate of $(X_k, \overline{X}_k)$, and $$\tilde{T}^\epsilon=\inf\left\{\;k\geq 0\;|\;(X_k,\overline{X}_k)\;\textrm{satisfies}\;X_k\notin U\;\right\},$$ which acts as $\tilde{\tau}^\epsilon$ for the first coordinate of $(X_k,\overline{X}_k)$.  It follows from the definitions that $\tau^\epsilon\leq \tilde{\tau}^\epsilon$ and $T^{\epsilon,n}_1\leq \tilde{T}^\epsilon$.

The definition of $Q_{n,x}$ and the Markov property imply that \begin{multline}\label{main_time_6}P_{x,\omega}(\tilde{\tau}^\epsilon-\tau^{\epsilon,n}_1\geq L_{n-1}^2)=Q_{n,x}(\tilde{T}^\epsilon-T^{\epsilon,n}_1\geq (\frac{L_{n-1}}{L_{n-\overline{m}}})^2) \\ =Q_{n,x}(\tilde{T}^\epsilon-T^{\epsilon,n}_1\geq (\frac{L_{n-1}}{L_{n-\overline{m}}})^2,C_n)+Q_{n,x}(\tilde{T}^\epsilon-T^{\epsilon,n}_1\geq (\frac{L_{n-1}}{L_{n-\overline{m}}})^2,C_n^c),\end{multline}  where (\ref{main_time_1}) states that the first term of (\ref{main_time_6}) is bounded, for $C>0$ independent of $n$, by \begin{equation}\label{main_time_3} Q_{n,x}(\tilde{T}^\epsilon-T^{\epsilon,n}_1\geq (\frac{L_{n-1}}{L_{n-\overline{m}}})^2, C_n)\leq Q_{n,x}(C_n)\leq CL_{n-1}^{-10a}.\end{equation}  The second term is further decomposed like \begin{multline}\label{main_time_4} Q_{n,x}(\tilde{T}^\epsilon-T^{\epsilon,n}_1\geq (\frac{L_{n-1}}{L_{n-\overline{m}}})^2, C_n^c)=Q_{n,x}(\tilde{T}^\epsilon-T^{\epsilon,n}_1\geq (\frac{L_{n-1}}{L_{n-\overline{m}}})^2, C_n^c,T^{\epsilon,n}_1> (\frac{L_{n+2}}{L_{n-\overline{m}}})^2) \\ +Q_{n,x}(\tilde{T}^\epsilon-T^{\epsilon,n}_1\geq (\frac{L_{n-1}}{L_{n-\overline{m}}})^2, C_n^c,T^{\epsilon,n}_1\leq (\frac{L_{n+2}}{L_{n-\overline{m}}})^2).\end{multline}

The first term of (\ref{main_time_4}) is bounded using Proposition \ref{exit_main}, and particularly (\ref{exit_main_5}) which applies equally to the discrete sequence since $L_{n-\overline{m}}^2$ divides $L_{n+2}^2$ according to the choice (\ref{L}), to yield, for $C>0$ independent of $n$, \begin{equation}\label{main_time_7}Q_{n,x}(\tilde{T}^\epsilon-T^{\epsilon,n}_1\geq (\frac{L_{n-1}}{L_{n-\overline{m}}})^2,  C_n^c, T^{\epsilon,n}_1> (\frac{L_{n+2}}{L_{n-\overline{m}}})^2)\leq Q_{n,x}(T^{\epsilon,n}_1>(\frac{L_{n+2}}{L_{n-\overline{m}}})^2)\leq C(\epsilon L_{n+2})^{-3}.\end{equation}

To bound the second term of (\ref{main_time_4}), define the discrete stopping time $$\overline{T}^{\epsilon,n}_2=\inf\left\{\;k\geq 0\;|\;(X_k,\overline{X}_k)\;\;\textrm{satisfies}\;\;\d(\overline{X}_k,(U/\epsilon))\geq \tilde{D}_{n-\overline{m}}\right\},$$ which acts as $\tau^{\epsilon,n}_2$ defined for the second coordinate of the process $(X_k,\overline{X}_k)$.  First, on the event $C_n^c$, for every $0\leq k\leq 2(\frac{L_{n+2}}{L_{n-\overline{m}}})^2$, $$\d(X_k,(U/\epsilon)^c)\leq\tilde{D}_{n-\overline{m}}\;\;\textrm{implies}\;\;\d(\overline{X}_k,(U/\epsilon)^c)\leq \tilde{D}_{n-\overline{m}}+L_{n-\overline{m}}\leq 2\tilde{D}_{n-\overline{m}},$$ and $$\d(\overline{X}_k,(U/\epsilon))\geq \tilde{D}_{n-\overline{m}}\;\;\textrm{implies}\;\;\d(X_k,(U/\epsilon))\geq \tilde{D}_{n-\overline{m}}-L_{n-\overline{m}}> 0.$$  Next, on the event $(C_n^c, T^{\epsilon,n}_1\leq (\frac{L_{n-1}}{L_{n-\overline{m}}})^2)$, it follows by definition that $$\d(X_{T^{\epsilon,n}_1},(U/\epsilon)^c)\leq\tilde{D}_{n-\overline{m}}\;\;\textrm{and}\;\;\d(\overline{X}_{\overline{T}^{\epsilon,n}_2},(U/\epsilon))\geq\tilde{D}_{n-\overline{m}}.$$  Therefore, from the Markov property, Proposition \ref{main_Brownian} and the definition of $Q_{n,x}$, for $C>0$ independent of $n$, \begin{multline}\label{main_time_5} Q_{n,x}(\tilde{T}^\epsilon-T^{\epsilon,n}_1\geq (\frac{L_{n-1}}{L_{n-\overline{m}}})^2, C_n^c, T^{\epsilon,n}_1\leq (\frac{L_{n+2}}{L_{n-\overline{m}}})^2)\leq \sup_{\d(x,(U/\epsilon)^c)\leq 2\tilde{D}_{n-\overline{m}}}Q_{n,x}(\overline{T}^{\epsilon,n}_2\geq (\frac{L_{n-1}}{L_{n-\overline{m}}})^2) \\ =\sup_{\d(x,(U/\epsilon)^c)\leq 2\tilde{D}_{n-\overline{m}}}W^{n-\overline{m}}_x(\tau_2^{\epsilon,n}\geq L_{n-1}^2)\leq CL_{n-1}^{-10a}.\end{multline}

Since $\tau^\epsilon\leq \tilde{\tau}^\epsilon$ by definition (\ref{main_time_2}), the collection (\ref{main_time_6}), (\ref{main_time_3}), (\ref{main_time_4}), (\ref{main_time_7}) and (\ref{main_time_5}) assert that, for $C>0$ independent of $n$, $$P_{x,\omega}(\tau^\epsilon-\tau^{\epsilon,n}_1\geq L_{n-1}^2)\leq P_{x,\omega}(\tilde{\tau}^\epsilon-\tau^{\epsilon,n}_1\geq L_{n-1}^2)\leq CL_{n-1}^{-10a}+C(\epsilon L_{n+2})^{-3},$$ which, since $n$ sufficiently large, $\omega\in A_n$, $L_n\leq\frac{1}{\epsilon}<L_{n+1}$ and $x\in \overline{U}/\epsilon$ were arbitrary, completes the argument.\end{proof}

Stochastic homogenization for solutions of (\ref{main_eq}) is now established.  Because the case of zero righthand side and nonzero boundary data was considered in \cite{F3}, by linearity it remains only to prove homogenization for solutions of (\ref{main_eq}) with nonzero righthand side which vanish along the boundary.  Precisely, it will first be shown that solutions \begin{equation}\label{main_eq_1} \left\{\begin{array}{ll} \frac{1}{2}\tr(A(\frac{x}{\epsilon},\omega)D^2u^\epsilon)+\frac{1}{\epsilon}b(\frac{x}{\epsilon},\omega)\cdot Du^\epsilon=g(x) & \textrm{on}\;\;U, \\ u^\epsilon=0 & \textrm{on}\;\;\partial U,\end{array}\right.\end{equation} converge, as $\epsilon\rightarrow 0$, on a subset of full probability and uniformly on $\overline{U}$, to the solution \begin{equation}\label{main_hom_1}\left\{\begin{array}{ll} \frac{\overline{\alpha}}{2}\Delta\overline{u}=g(x) & \textrm{on}\;\;U, \\ \overline{u}=0 & \textrm{on}\;\;\partial U.\end{array}\right.\end{equation}  And, the proof will essentially analyze solutions to the rescaled equation \begin{equation}\label{main_eq_2} \left\{\begin{array}{ll} \frac{1}{2}\tr(A(x,\omega)D^2v^\epsilon)+b(x,\omega)\cdot Dv^\epsilon=\epsilon^2g(\epsilon x) & \textrm{on}\;\;U/\epsilon, \\ v^\epsilon=0 & \textrm{on}\;\;\partial U/\epsilon,\end{array}\right.\end{equation} which admit the representation $$u^\epsilon(x)=v^\epsilon(\frac{x}{\epsilon})=E_{\frac{x}{\epsilon},\omega}(-\epsilon^2\int_0^{\tau^\epsilon}g(\epsilon X_s)\;ds)\;\;\textrm{on}\;\;\overline{U},$$ for $\tau^\epsilon$ the exit time from $U/\epsilon$.

The first step will be to apply a sub-optimal bound for the exit time through the use of Proposition \ref{exit_main}.  Precisely, for scales $L_n\leq \frac{1}{\epsilon}<L_{n+1}$, it will be shown on the event $A_n$ that, up to an error vanishing with $\epsilon$, the solution $v^\epsilon$ is well-approximated by the quantity \begin{equation}\label{main_outline_1}v^\epsilon(\frac{x}{\epsilon})\simeq E_{\frac{x}{\epsilon},\omega}(-\epsilon^2\int_0^{\tau^\epsilon}g(\epsilon X_s)\;ds, \tau^\epsilon\leq L_{n+2}^2).\end{equation}  Since the exit time of a corresponding Brownian motion is expected to be of order $\frac{1}{\epsilon^2}$ which, as $n\rightarrow\infty$, is significantly smaller than $L_{n+2}^2$, this estimate does not imply an effective upper bound for the exit time of the diffusion in random environment.  However, it does allow for the application of the global coupling established in Corollary \ref{couple_cor}.

The second step replaces the continuous exit time $\tau^\epsilon$ with its discrete proxy $\tau^{\epsilon,n}_1$, where the exponential estimates guaranteed on the event $A_n$ by Control \ref{main_localization} and Proposition \ref{main_time} will be used to show that the discretely stopped version of (\ref{main_outline_1}) is a good approximation for the solution $v^\epsilon$ in the sense that \begin{equation}\label{main_outline_2}E_{\frac{x}{\epsilon},\omega}(-\epsilon^2\int_0^{\tau^\epsilon}g(\epsilon X_s)\;ds, \tau^\epsilon\leq L_{n+2}^2)\simeq E_{\frac{x}{\epsilon},\omega}(-\epsilon^2\int_0^{\tau^{\epsilon,n}_1}g(\epsilon X_s)\;ds, \tau^{\epsilon,n}_1\leq L_{n+2}^2).\end{equation}  And, again using the localization estimates from Control \ref{main_localization}, the integral will be shown to be accurately represented by its discrete approximation on scale $L^2_{n-\overline{m}}$ in the sense that, up to an error vanishing with $\epsilon$, \begin{multline}\label{main_outline_3}E_{\frac{x}{\epsilon},\omega}(-\epsilon^2\int_0^{\tau^{\epsilon,n}_1}g(\epsilon X_s)\;ds, \tau^{\epsilon,n}_1\leq L_{n+2}^2)\simeq \\ E_{\frac{x}{\epsilon},\omega}(-\epsilon^2\sum_{k=0}^{(\tau^{\epsilon,n}_1-L_{n-\overline{m}}^2)/L_{n-\overline{m}}^2}L_{n-\overline{m}}^2g(\epsilon X_{kL_{n-\overline{m}}^2}), \tau^{\epsilon,n}_1\leq L_{n+2}^2).\end{multline}

The global coupling established in Section \ref{section_coupling} now plays its role.  It follows from the definition of the measure $Q_{n,\frac{x}{\epsilon}}$ and process $(X_k,\overline{X}_k)$ on $(\mathbb{R}^d\times\mathbb{R}^d)^{\mathbb{N}}$ that, writing $E^{Q_{n,\frac{x}{\epsilon}}}$ for the expectation with respect to $Q_{n,\frac{x}{\epsilon}}$, \begin{multline}\label{main_outline_4} E_{\frac{x}{\epsilon},\omega}(-\epsilon^2\sum_{k=0}^{(\tau^{\epsilon,n}_1-L_{n-\overline{m}}^2)/L_{n-\overline{m}}^2}L_{n-\overline{m}}^2g(\epsilon X_{kL_{n-\overline{m}}^2}), \tau^{\epsilon,n}_1\leq L_{n+2}^2)= \\ E^{Q_{n,\frac{x}{\epsilon}}}(-\epsilon^2\sum_{k=0}^{T^{\epsilon,n}_1-1}L_{n-\overline{m}}^2g(\epsilon X_k), T^{\epsilon,n}_1\leq(\frac{L_{n+2}}{L_{n-\overline{m}}})^2),\end{multline} for $T^{\epsilon,n}_1$ the analogue of $\tau^{\epsilon,n}_1$ for the first coordinate of $(X_k, \overline{X}_k)$. The coupling estimates stated in Corollary \ref{couple_cor} are then used to obtain a comparison with Brownian motion of variance $\alpha_{n-\overline{m}}$ and to prove, up to an error vanishing with $\epsilon$, \begin{multline}\label{main_outline_5} E^{Q_{n,\frac{x}{\epsilon}}}(-\epsilon^2\sum_{k=0}^{T^{\epsilon,n}_1-1}L_{n-\overline{m}}^2g(\epsilon X_k), T^{\epsilon,n}_1\leq(\frac{L_{n+2}}{L_{n-\overline{m}}})^2)\simeq \\ E^{Q_{n,\frac{x}{\epsilon}}}(-\epsilon^2\sum_{k=0}^{T^{\epsilon,n}_1-1}L_{n-\overline{m}}^2g(\epsilon \overline{X}_k), T^{\epsilon,n}_1\leq(\frac{L_{n+2}}{L_{n-\overline{m}}})^2).\end{multline}

The remainder of the proof is then essentially an unwinding of the above outline in terms of Brownian motion.  For $\overline{T}^{\epsilon,n}_2$ the analogue of $\tau^{\epsilon,n}_2$ defined for the second coordinate of the process $(X_k, \overline{X}_k)$, it is first shown that, up to an error vanishing with $\epsilon$, \begin{multline}\label{main_outline_6} E^{Q_{n,\frac{x}{\epsilon}}}(-\epsilon^2\sum_{k=0}^{T^{\epsilon,n}_1-1}L_{n-\overline{m}}^2g(\epsilon \overline{X}_k), T^{\epsilon,n}_1\leq(\frac{L_{n+2}}{L_{n-\overline{m}}})^2)\simeq \\ E^{Q_{n,\frac{x}{\epsilon}}}(-\epsilon^2\sum_{k=0}^{\overline{T}^{\epsilon,n}_2-1}L_{n-\overline{m}}^2g(\epsilon \overline{X}_k), \overline{T}^{\epsilon,n}_2\leq (\frac{L_{n+2}}{L_{n-\overline{m}}})^2),\end{multline} where, by the definition of $Q_{n,\frac{x}{\epsilon}}$, \begin{multline}\label{main_outline_7} E^{Q_{n,\frac{x}{\epsilon}}}(-\epsilon^2\sum_{k=0}^{\overline{T}^{\epsilon,n}_2-1}L_{n-\overline{m}}^2g(\epsilon \overline{X}_k), \overline{T}^{\epsilon,n}_2\leq (\frac{L_{n+2}}{L_{n-\overline{m}}})^2)= \\ E^{W^{n-\overline{m}}_{\frac{x}{\epsilon}}}(-\epsilon^2\sum_{k=0}^{(\tau^{\epsilon,n}_2-L_{n-\overline{m}}^2)/L_{n-\overline{m}}^2}L_{n-\overline{m}}^2g(\epsilon X_{kL_{n-\overline{m}}^2}), \tau^{\epsilon,n}_2\leq L_{n+2}^2).\end{multline}  Following standard exponential estimates for Brownian motion, using the control of the $\alpha_n$ implied by Theorem \ref{effectivediffusivity} and the upper bound for the exit time in probability obtained in Proposition \ref{main_Brownian} repeatedly, it will be shown that, up to an error vanishing with $\epsilon$, \begin{equation}\label{main_outline_8} E^{W^{n-\overline{m}}_{\frac{x}{\epsilon}}}(-\epsilon^2\sum_{k=0}^{\tau^{\epsilon,n}_2-1}L_{n-\overline{m}}^2g(\epsilon \overline{X}_{kL_{n-\overline{m}}}^2), \tau^{\epsilon,n}_2\leq L_{n+2}^2)\simeq E^{W^{n-\overline{m}}_{\frac{x}{\epsilon}}}(-\epsilon^2\int_0^{\tau^{\epsilon,n}_2}g(\epsilon X_s)\;ds, \tau^{\epsilon,n}_2\leq L_{n+2}^2).\end{equation} The same estimates then replace $\tau^{\epsilon,n}_2$ with $\tau^\epsilon$ and remove the cutoff to provide \begin{equation}\label{main_outline_10} E^{W^{n-\overline{m}}_{\frac{x}{\epsilon}}}(-\epsilon^2\int_0^{\tau^{\epsilon,n}_2}g(\epsilon X_s)\;ds, \tau^{\epsilon,n}_2\leq L_{n+2}^2)\simeq E^{W^{n-\overline{m}}_{\frac{x}{\epsilon}}}(-\epsilon^2\int_0^{\tau^\epsilon}g(\epsilon X_s)\;ds).\end{equation}

The final step comes in approximating the solution of the homogenized equation (\ref{main_hom_1}) by solutions of the approximate equations \begin{equation}\label{main_outline_11}\left\{\begin{array}{ll} \frac{\alpha_{n-\overline{m}}}{2}\Delta \overline{u}_{n-\overline{m}}=g(x) & \textrm{on}\;\;U, \\ \overline{u}_{n-\overline{m}}=0 & \textrm{on}\;\;\partial U,\end{array}\right.\end{equation} which admit the representation \begin{equation}\label{main_outline_12}\overline{u}_{n-\overline{m}}(x)=E^{W^{n-\overline{m}}_\frac{x}{\epsilon}}(-\epsilon^2\int_0^{\tau^\epsilon}g(\epsilon X_s)\;ds)=E^{W^{n-\overline{m}}_\frac{x}{\epsilon}}(-\int_0^{\epsilon^2\tau^\epsilon}g(\epsilon X_\frac{s}{\epsilon^2})\;ds)\;\;\textrm{on}\;\;\overline{U},\end{equation} and coincide with the righthand side of (\ref{main_outline_10}).

\begin{prop}\label{main_alpha}  For each $n\geq \overline{m}$, for $C=C(U,\overline{\alpha})>0$ independent of $n$, the solutions of (\ref{main_hom_1}) and (\ref{main_outline_11}) satisfy $$\norm{\overline{u}-\overline{u}_{n-\overline{m}}}_{L^\infty(\overline{U})}\leq C\norm{g}_{L^\infty(\overline{U})}L_{n-\overline{m}}^{-(1+\frac{9}{10})\delta}.$$ \end{prop}

\begin{proof}  Fix $n\geq \overline{m}$.  For the respective solutions $\overline{u}$ and $\overline{u}_{n-\overline{m}}$ of (\ref{main_hom_1}) and (\ref{main_outline_11}), the difference $$w_{n-\overline{m}}=\overline{u}-\overline{u}_{n-\overline{m}}$$ solves the equation $$\left\{\begin{array}{ll} \frac{\overline{\alpha}}{2}\Delta w_{n-\overline{m}}=(1-\frac{\overline{\alpha}}{\alpha_{n-\overline{m}}})g(x) & \textrm{on}\;\;U, \\ w_{n-\overline{m}}=0 & \textrm{on}\;\;\partial U.\end{array}\right.$$  Therefore, using Theorem \ref{effectivediffusivity}, for $C>0$ independent of $n$, writing $W^\infty_x$ for the Wiener measure defining Brownian motion beginning from $x$ with variance $\overline{\alpha}$, and writing $\tau_U$ for the exit time from $U$, $$\norm{w_{n-\overline{m}}}_{L^\infty(\overline{U})}\leq C\norm{g}_{L^\infty(\overline{U})}\abs{\alpha_{n-\overline{m}}-\overline{\alpha}}\sup_{x\in\overline{U}}E^{W^\infty_x}(\tau_U)\leq C\norm{g}_{L^\infty(\overline{U})}L_{n-\overline{m}}^{-(1+\frac{9}{10})\delta},$$ which completes the argument. \end{proof}

And, therefore, in view of the representation (\ref{main_outline_12}) and Proposition \ref{main_alpha}, up to an error vanishing with $\epsilon$, \begin{equation}\label{main_outline_14}E^{W^{n-\overline{m}}_{\frac{x}{\epsilon}}}(-\epsilon^2\int_0^{\tau^\epsilon}g(\epsilon X_s)\;ds)\simeq E^{W^\infty_{\frac{x}{\epsilon}}}(-\epsilon^2\int_0^{\tau^\epsilon}g(\epsilon X_s)\;ds),\end{equation} which completes the proof.  The full homogenization statement including a nonzero boundary condition then follows immediately after recalling the results from \cite{F3}.

The stochastic homogenization will be obtained on a subset of full probability defined by the events $A_n$ and an application of the Borel-Cantelli lemma.  Since Proposition \ref{prob_probability} implies that, for each $n\geq \overline{m}$, for $C\geq 0$ independent of $n$, $$\mathbb{P}(\Omega\setminus A_n)\leq CL_n^{2d(1+a)^2-\frac{1}{2}M_0},$$  the definition of $L_n$ in (\ref{L}) and the negative exponent $2d(1+a)^2-\frac{1}{2}M_0<0$ guarantee the sum $$\sum_{n=\overline{m}}^\infty\mathbb{P}(\Omega\setminus A_n)\leq C\sum_{n=\overline{m}}^\infty L_n^{2d(1+a)^2-\frac{1}{2}M_0}<\infty.$$  The Borel-Cantelli lemma therefore implies the event \begin{equation}\label{main_mainevent} \Omega_0=\left\{\;\omega\in\Omega\;|\;\textrm{There exists}\;\overline{n}=\overline{n}(\omega)\;\textrm{such that}\;\omega\in A_n\;\textrm{for all}\;n\geq \overline{n}.\;\right\}\;\;\textrm{satisfies}\;\;\mathbb{P}(\Omega_0)=1.\end{equation}  Note particularly that the subset of full probability $\Omega_0$ is independent of the domain and the righthand side.  It is on this event that homogenization is achieved following the outline presented between lines (\ref{main_outline_1}) to (\ref{main_outline_14}).

The result is first established for functions which are the restriction of a smooth, compactly supported function on $\mathbb{R}^d$.  \begin{equation}\label{main_smooth}\textrm{Assume}\;\;g(x)\in C^\infty_c(\mathbb{R}^d).\end{equation}  This assumption is removed by a standard approximation argument in Theorem \ref{main_arbitrary}.

\begin{thm}\label{main_main} Assume (\ref{steady}), (\ref{constants}) and (\ref{main_smooth}).  For every $\omega\in \Omega_0$, the respective solutions $u^\epsilon$ and $\overline{u}$ of (\ref{main_eq_1}) and (\ref{main_hom_1}) satisfy $$\lim_{\epsilon\rightarrow 0}\norm{u^\epsilon-\overline{u}}_{L^\infty(\overline{U})}=0.$$\end{thm}

\begin{proof}  Fix $\omega\in\Omega_0$ and $n_0\geq \overline{m}$ such that, for all $n\geq n_0$, for $r_0$ the constant quantifying the exterior ball condition, $$2\tilde{D}_{n-\overline{m}}<\frac{r_0L_n}{2}\;\;\textrm{and}\;\;\omega\in A_n.$$  Then, fix $\epsilon_0>0$ sufficiently small so that, whenever $0<\epsilon<\epsilon_0$ satisfies $L_n\leq \frac{1}{\epsilon}<L_{n+1}$ it follows that $n\geq n_0$.  Furthermore, using the boundedness of the domain $U$, choose $0<\epsilon_1<\epsilon_0$ such that, whenever $0<\epsilon<\epsilon_1$ and $L_n\leq \frac{1}{\epsilon}<L_{n+1}$, $$L_{n+1}U\subset [-\frac{1}{2}L_{n+2}^2, \frac{1}{2}L_{n+2}^2]^d.$$  These conditions guarantee that whenever $0<\epsilon<\epsilon_1$ the conclusions of Proposition \ref{main_Brownian} and \ref{main_time} are satisfied, and that Controls \ref{Holder} and \ref{main_localization} are available, on scales $L_{n-\overline{m}}$ to $L_{n+2}$, for the entirety of the domain $U/\epsilon$.  

Henceforth, fix $x\in\overline{U}$ and $0<\epsilon<\epsilon_1$.  Write $u^\epsilon$ for the solution of (\ref{main_eq_1}) and $v^\epsilon$ for the solution of the rescaled (\ref{main_eq_2}), and recall the representation \begin{equation}\label{main_main_1}u^\epsilon(x)=v^\epsilon(\frac{x}{\epsilon})=E_{\frac{x}{\epsilon},\omega}(-\epsilon^2\int_0^{\tau^\epsilon}g(\epsilon X_s)\;ds).\end{equation}  In order to apply the coupling estimates obtained in Section \ref{section_coupling}, it is necessary to restrict the above integral to the event $\left\{\tau^\epsilon\leq L_{n+2}^2\right\}$.

\emph{The proof of (\ref{main_outline_1}).}  First, observe that \begin{multline}\label{main_main_2}\abs{E_{\frac{x}{\epsilon},\omega}(-\epsilon^2\int_0^{\tau^\epsilon}g(\epsilon X_s)\;ds)-E_{\frac{x}{\epsilon},\omega}(-\epsilon^2\int_0^{\tau^\epsilon}g(\epsilon X_s)\;ds, \tau^\epsilon\leq L_{n+2}^2)}\leq \\ \epsilon^2\norm{g}_{L^\infty(\mathbb{R}^d)}\sum_{k=1}^\infty(k+1)L_{n+2}^2P_{\frac{x}{\epsilon},\omega}(kL_{n+2}^2< \tau^\epsilon\leq (k+1)L_{n+2}^2)\leq \\ \epsilon^2\norm{g}_{L^\infty(\mathbb{R}^d)}\sum_{k=1}^\infty(k+1)L_{n+2}^2P_{\frac{x}{\epsilon},\omega}(\tau^\epsilon> kL_{n+2}^2).\end{multline}  Therefore, since $L_n\leq\frac{1}{\epsilon}<L_{n+1}$, and since Proposition \ref{exit_main} proved that, on the event $A_n$, for each $k\geq 0$, for $C>0$ independent of $n$ and $k$, $$P_{\frac{x}{\epsilon},\omega}(\tau^\epsilon> kL_{n+2}^2)\leq C(\epsilon L_{n+2})^{-3k},$$ it follows from the definition of $L_n$ in (\ref{L}) and properties of the geometric series that, for $C>0$ independent of $n$, \begin{multline}\label{main_main_3} \epsilon^2\norm{g}_{L^\infty(\overline{U})}\sum_{k=1}^\infty(k+1)L_{n+2}^2P_{\frac{x}{\epsilon},\omega}(\tau^\epsilon > kL_{n+2}^2) \leq C\norm{g}_{L^\infty(\mathbb{R}^d)}(\epsilon L_{n+2})^2\sum_{k=1}^\infty (k+1)(\epsilon L_{n+2})^{-3k} \\ \leq C\norm{g}_{L^\infty(\mathbb{R}^d)}(\epsilon L_{n+2})^{-1}.\end{multline}  Notice that the effectiveness of this estimates relies upon the assumption $d\geq 3$ through the application of Proposition \ref{exit_main}.  And, since $L_n\leq\frac{1}{\epsilon}< L_{n+1}$, for $C>0$ independent of $n$, \begin{equation}\label{main_main_4}\abs{E_{\frac{x}{\epsilon},\omega}(-\epsilon^2\int_0^{\tau^\epsilon}g(\epsilon X_s)\;ds)-E_{\frac{x}{\epsilon},\omega}(-\epsilon^2\int_0^{\tau^\epsilon}g(\epsilon X_s)\;ds, \tau^\epsilon<L_{n+2}^2)}\leq C\norm{g}_{L^\infty(\mathbb{R}^d)}\frac{L_{n+1}}{L_{n+2}},\end{equation} which completes the proof of (\ref{main_outline_1}).

\emph{The proof of (\ref{main_outline_2}).}  Recall the discrete $\C([0,\infty);\mathbb{R}^d)$ stopping time  $$\tau^{\epsilon,n}_1=\inf\left\{\;kL_{n-\overline{m}}^2\geq 0\;|\;\d(X_{kL_{n-\overline{m}}^2},(U/\epsilon)^c)\leq \tilde{D}_{n-\overline{m}}\;\right\}.$$  First, decompose the second term of (\ref{main_main_4}) as \begin{multline}\label{main_main_5} E_{\frac{x}{\epsilon},\omega}(-\epsilon^2\int_0^{\tau^\epsilon}g(\epsilon X_s)\;ds, \tau^\epsilon<L_{n+2}^2)=\\ E_{\frac{x}{\epsilon},\omega}(-\epsilon^2\int_0^{\tau^\epsilon}g(\epsilon X_s)\;ds, \tau^\epsilon<L_{n+2}^2, \tau^\epsilon +L_{n-\overline{m}}^2 \leq \tau^{\epsilon,n}_1)+\\ E_{\frac{x}{\epsilon},\omega}(-\epsilon^2\int_0^{\tau^\epsilon}g(\epsilon X_s)\;ds, \tau^\epsilon\leq L_{n+2}^2, \tau^\epsilon +L_{n-\overline{m}}^2 >\tau^{\epsilon,n}_1).\end{multline}  Since $\omega\in A_n$ and because the definitions imply that on the event $\left\{\tau^\epsilon +L_{n-\overline{m}}^2 \leq \tau^{\epsilon,n}_1\right\}$ the diffusion undergoes an excursion of size at least $\tilde{D}_{n-\overline{m}}$ in time $L_{n-\overline{m}}^2$, the exponential estimates guaranteed by Control \ref{main_localization} act to bound the first term of this equality, and yield \begin{multline}\label{main_main_6} \abs{E_{\frac{x}{\epsilon},\omega}(-\epsilon^2\int_0^{\tau^\epsilon}g(\epsilon X_s)\;ds, \tau^\epsilon\leq L_{n+2}^2, \tau^\epsilon+L_{n-\overline{m}}^2\leq \tau^{\epsilon,n}_1)}\leq \\ \epsilon^2L_{n+2}^2\norm{g}_{L^\infty(\overline{U})}E_{\frac{x}{\epsilon},\omega}(P_{X_{\tau^\epsilon},\omega}(X^*_{L_{n-\overline{m}}^2}\geq \tilde{D}_{n-\overline{m}}))\leq (\frac{L_{n+2}}{L_n})^2\norm{g}_{L^\infty(\mathbb{R}^d)}\exp(-\tilde{\kappa}_{n-\overline{m}}).\end{multline}

The second term of (\ref{main_main_5}) is further decomposed according to \begin{multline}\label{main_main_7}E_{\frac{x}{\epsilon},\omega}(-\epsilon^2\int_0^{\tau^\epsilon}g(\epsilon X_s)\;ds, \tau^\epsilon\leq L_{n+2}^2, \tau^\epsilon +L_{n-\overline{m}}^2 >\tau^{\epsilon,n}_1)= \\ E_{\frac{x}{\epsilon},\omega}(-\epsilon^2\int_0^{\tau^\epsilon}g(\epsilon X_s)\;ds, \tau^\epsilon\leq L_{n+2}^2, 0 \leq \tau^{\epsilon,n}_1-\tau^\epsilon<L_{n-\overline{m}}^2)+ \\ E_{\frac{x}{\epsilon},\omega}(-\epsilon^2\int_0^{\tau^\epsilon}g(\epsilon X_s)\;ds, \tau^\epsilon\leq L_{n+2}^2, \tau^\epsilon-\tau^{\epsilon,n}_1>0).\end{multline}  In comparing the lefthand side of (\ref{main_main_7}) with the discretely stopped version \begin{equation}\label{main_main_8}E_{\frac{x}{\epsilon},\omega}(-\epsilon^2\int_0^{\tau^{\epsilon,n}_1}g(\epsilon X_s)\;ds, \tau^\epsilon\leq L_{n+2}^2, \tau^\epsilon +L_{n-\overline{m}}^2 >\tau^{\epsilon,n}_1),\end{equation} the decomposition (\ref{main_main_7}) implies that the difference is bounded by \begin{multline}\label{main_main_9} |E_{\frac{x}{\epsilon},\omega}(-\epsilon^2\int_0^{\tau^\epsilon}g(\epsilon X_s)\;ds, \tau^\epsilon\leq L_{n+2}^2, \tau^\epsilon +L_{n-\overline{m}}^2 >\tau^{\epsilon,n}_1)- \\ E_{\frac{x}{\epsilon},\omega}(-\epsilon^2\int_0^{\tau^{\epsilon,n}_1}g(\epsilon X_s)\;ds, \tau^\epsilon\leq L_{n+2}^2, \tau^\epsilon +L_{n-\overline{m}}^2>\tau^{\epsilon,n}_1)|\leq \\ \epsilon^2\norm{g}_{L^\infty(\mathbb{R}^d)}E_{\frac{x}{\epsilon},\omega}(\tau^{\epsilon,n}_1-\tau^\epsilon, \tau^\epsilon\leq L_{n+2}^2, 0 \leq \tau^{\epsilon,n}_1-\tau^\epsilon<L_{n-\overline{m}}^2) + \\\epsilon^2\norm{g}_{L^\infty(\mathbb{R}^d)}E_{\frac{x}{\epsilon},\omega}(\tau^\epsilon-\tau^{\epsilon,n}_1, \tau^\epsilon\leq L_{n+2}^2, \tau^\epsilon-\tau^{\epsilon,n}_1>0).\end{multline}  The event describing the first term of the righthand side of (\ref{main_main_9}) allows for the immediate $L^\infty$-estimate of the integrand \begin{multline}\label{main_main_10}\epsilon^2\norm{g}_{L^\infty(\mathbb{R}^d)}E_{\frac{x}{\epsilon},\omega}(\tau^{\epsilon,n}_1-\tau^\epsilon, \tau^\epsilon\leq L_{n+2}^2, 0 \leq \tau^{\epsilon,n}_1-\tau^\epsilon<L_{n-\overline{m}}^2)\leq \\ \epsilon^2\norm{g}_{L^\infty(\mathbb{R}^d)}L_{n-\overline{m}}^2\leq (\frac{L_{n-\overline{m}}}{L_n})^2\norm{g}_{L^\infty(\mathbb{R}^d)}.\end{multline}  The second term of the righthand side of (\ref{main_main_9}) is bounded using Proposition \ref{main_time}.  Form the decomposition \begin{multline}\label{main_main_11} \epsilon^2\norm{g}_{L^\infty(\mathbb{R}^d)}E_{\frac{x}{\epsilon},\omega}(\tau^\epsilon-\tau^{\epsilon,n}_1, \tau^\epsilon\leq L_{n+2}^2, \tau^\epsilon-\tau^{\epsilon,n}_1>0)=\\ \epsilon^2\norm{g}_{L^\infty(\mathbb{R}^d)}E_{\frac{x}{\epsilon},\omega}(\tau^\epsilon-\tau^{\epsilon,n}_1, \tau^\epsilon\leq L_{n+2}^2, 0<\tau^\epsilon-\tau^{\epsilon,n}_1\leq L_{n-1}^2)+\\ \epsilon^2\norm{g}_{L^\infty(\mathbb{R}^d)}E_{\frac{x}{\epsilon},\omega}(\tau^\epsilon-\tau^{\epsilon,n}_1, \tau^\epsilon\leq L_{n+2}^2, \tau^\epsilon-\tau^{\epsilon,n}_1>L_{n-1}^2).\end{multline}  The event defining the first term of the righthand side of (\ref{main_main_11}) admits the immediate $L^\infty$-estimate for the integrand \begin{multline}\label{main_main_12}\epsilon^2\norm{g}_{L^\infty(\mathbb{R}^d)}E_{\frac{x}{\epsilon},\omega}(\tau^\epsilon-\tau^{\epsilon,n}_1, \tau^\epsilon\leq L_{n+2}^2, 0<\tau^\epsilon-\tau^{\epsilon,n}_1\leq L_{n-1}^2)\leq \\ \epsilon^2L_{n-1}^2\norm{g}_{L^\infty(\mathbb{R}^d)}\leq (\frac{L_{n-1}}{L_n})^2\norm{g}_{L^\infty(\mathbb{R}^d)}.\end{multline}  Then, Proposition \ref{main_time} is applied to the second term of (\ref{main_main_11}), and yields, for $C>0$ independent of $n$, \begin{multline}\label{main_main_14} \epsilon^2\norm{g}_{L^\infty(\mathbb{R}^d)}E_{\frac{x}{\epsilon},\omega}(\tau^\epsilon-\tau^{\epsilon,n}_1, \tau^\epsilon\leq L_{n+2}^2, \tau^\epsilon-\tau^{\epsilon,n}_1>L_{n-1}^2)\leq \\ \epsilon^2\norm{g}_{L^\infty(\mathbb{R}^d)}E_{\frac{x}{\epsilon},\omega}(\tau^\epsilon, \tau^\epsilon\leq L_{n+2}^2, \tau^\epsilon-\tau^{\epsilon,n}_1>L_{n-1}^2)\leq \\ \epsilon^2\norm{g}_{L^\infty(\mathbb{R}^d)}L_{n+2}^2P_{\frac{x}{\epsilon},\omega}(\tau^\epsilon-\tau^{\epsilon,n}_1\geq L_{n-1}^2)\leq \\ C\norm{g}_{L^\infty(\mathbb{R}^d)}(\epsilon L_{n+2})^2(L_{n-1}^{-10a}+(\epsilon L_{n+2})^{-3}).\end{multline}  Therefore, owing to the definition of $L_n$ in (\ref{L}), since $L_n\leq \frac{1}{\epsilon}<L_{n+1}$, for $C>0$ independenet of $n$, \begin{multline}\label{main_main_15} \epsilon^2\norm{g}_{L^\infty(\mathbb{R}^d)}E_{\frac{x}{\epsilon},\omega}(\tau^\epsilon-\tau^{\epsilon,n}_1, \tau^\epsilon\leq L_{n+2}^2, \tau^\epsilon-\tau^{\epsilon,n}_1>L_{n-1}^2)\leq \\ C\norm{g}_{L^\infty(\mathbb{R}^d)}(L_n^{4a+2a^2-\frac{10a}{1+a}}+\frac{L_{n+1}}{L_{n+2}}),\end{multline} where definition (\ref{Holderexponent}) implies the exponent $$4a+2a^2-\frac{10a}{1+a}<0$$ is negative.

In combination, (\ref{main_main_10}), (\ref{main_main_12}) and (\ref{main_main_15}) imply using (\ref{main_main_9}) the bound, for $C>0$ independent of $n$, \begin{multline}\label{main_main_16} |E_{\frac{x}{\epsilon},\omega}(-\epsilon^2\int_0^{\tau^\epsilon}g(\epsilon X_s)\;ds, \tau^\epsilon\leq L_{n+2}^2, \tau^\epsilon +L_{n-\overline{m}}^2 >\tau^{\epsilon,n}_1)- \\ E_{\frac{x}{\epsilon},\omega}(-\epsilon^2\int_0^{\tau^{\epsilon,n}_1}g(\epsilon X_s)\;ds, \tau^\epsilon\leq L_{n+2}^2, \tau^\epsilon +L_{n-\overline{m}}^2 >\tau^{\epsilon,n}_1)| \leq \\ C\norm{g}_{L^\infty(\mathbb{R}^d)}((\frac{L_{n-1}}{L_n})^2+L_n^{4a+2a^2-\frac{10a}{1+a}}+\frac{L_{n+1}}{L_{n+2}}). \end{multline}  And, since the definitions (\ref{L}) and (\ref{kappa}) imply that, for $C>0$ independent of $n$, $$(\frac{L_{n+2}}{L_n})^2\exp(-\tilde{\kappa}_{n-\overline{m}})\leq C((\frac{L_{n-1}}{L_n})^2+L_n^{4a+2a^2-\frac{10a}{1+a}}+\frac{L_{n+1}}{L_{n+2}}),$$ equation (\ref{main_main_5}) and estimates (\ref{main_main_6}) and (\ref{main_main_16}) combine for the estimate, for $C>0$ independent of $n$, \begin{multline}\label{main_main_17} |E_{\frac{x}{\epsilon},\omega}(-\epsilon^2\int_0^{\tau^\epsilon}g(\epsilon X_s)\;ds, \tau^\epsilon\leq L_{n+2}^2)- \\ E_{\frac{x}{\epsilon},\omega}(-\epsilon^2\int_0^{\tau^{\epsilon,n}_1}g(\epsilon X_s)\;ds, \tau^\epsilon\leq L_{n+2}^2, \tau^\epsilon +L_{n-\overline{m}}^2 >\tau^{\epsilon,n}_1)| \leq \\ C\norm{g}_{L^\infty(\mathbb{R}^d)}((\frac{L_{n-1}}{L_n})^2+L_n^{4a+2a^2-\frac{10a}{1+a}}+\frac{L_{n+1}}{L_{n+2}}).\end{multline}

To obtain (\ref{main_outline_2}), it remains only to estimate the difference between the discretely stopped quantity within the absolute value of (\ref{main_main_17}) and \begin{equation}\label{main_main_18} E_{\frac{x}{\epsilon},\omega}(-\epsilon^2\int_0^{\tau^{\epsilon,n}_1}g(\epsilon X_s)\;ds, \tau^{\epsilon,n}_1\leq L_{n+2}^2).\end{equation}  First, notice with the aid of Proposition \ref{exit_main} that, for $C>0$ independent of $n$, \begin{multline}\label{main_main_19} |E_{\frac{x}{\epsilon},\omega}(-\epsilon^2\int_0^{\tau^{\epsilon,n}_1}g(\epsilon X_s)\;ds, \tau^\epsilon\leq L_{n+2}^2, \tau^\epsilon +L_{n-\overline{m}}^2 >\tau^{\epsilon,n}_1)- \\ E_{\frac{x}{\epsilon},\omega}(-\epsilon^2\int_0^{\tau^{\epsilon,n}_1}g(\epsilon X_s)\;ds, \tau^{\epsilon,n}_1<L_{n+2}^2+L_{n-\overline{m}}^2)|\leq \\ C\epsilon^2\norm{g}_{L^\infty(\mathbb{R}^d)}L_{n+2}^2P_{\frac{x}{\epsilon},\omega}(\tau^\epsilon>L_{n+2}^2)\leq \\ C\norm{g}_{L^\infty(\mathbb{R}^d)}(\epsilon L_{n+2})^{-1}\leq C\norm{g}_{L^\infty(\mathbb{R}^d)}\frac{L_{n+1}}{L_{n+2}}.\end{multline}  And then, again using Proposition \ref{exit_main} and in particular line (\ref{exit_main_5}) which applies equally to the discrete sequence since $L_{n-\overline{m}}^2$ divides $L_{n+2}^2$, for $C>0$ independent of $n$, \begin{multline}\label{main_main_20} |E_{\frac{x}{\epsilon},\omega}(-\epsilon^2\int_0^{\tau^{\epsilon,n}_1}g(\epsilon X_s)\;ds, \tau^{\epsilon,n}_1<L_{n+2}^2+L_{n-\overline{m}}^2)-E_{\frac{x}{\epsilon},\omega}(-\epsilon^2\int_0^{\tau^{\epsilon,n}_1}g(\epsilon X_s)\;ds, \tau^{\epsilon,n}_1 \leq L_{n+2}^2)| \leq \\ C\epsilon^2\norm{g}_{L^\infty(\mathbb{R}^d)}L_{n+2}^2 P_{\frac{x}{\epsilon},\omega}(\tau^{\epsilon,n}_1>L_{n+2}^2)\leq C\norm{g}_{L^\infty(\mathbb{R}^d)}(\epsilon L_{n+2})^{-1}\leq C\norm{g}_{L^\infty(\mathbb{R}^d)}\frac{L_{n+1}}{L_{n+2}}.\end{multline}  Therefore, in view of (\ref{main_main_17}), (\ref{main_main_19}) and (\ref{main_main_20}), for $C>0$ independent of $n$, \begin{multline}\label{main_main_21} \abs{E_{\frac{x}{\epsilon},\omega}(-\epsilon^2\int_0^{\tau^\epsilon}g(\epsilon X_s)\;ds, \tau^\epsilon\leq L_{n+2}^2)-E_{\frac{x}{\epsilon},\omega}(-\epsilon^2\int_0^{\tau^{\epsilon,n}_1}g(\epsilon X_s)\;ds, \tau^{\epsilon,n}_1\leq L_{n+2}^2)}\leq \\ C\norm{g}_{L^\infty(\mathbb{R}^d)}((\frac{L_{n-1}}{L_n})^2+L_n^{4a+2a^2-\frac{10a}{1+a}}+\frac{L_{n+1}}{L_{n+2}}),\end{multline} which completes the proof of (\ref{main_outline_2}).

\emph{The Proof of (\ref{main_outline_3}).}  The discrete approximation of the integral is a result of the Lipschitz continuity of $g$ and the exponential estimates implied by Control \ref{main_localization}.  Observe that, for $C>0$ independent of $n$, \begin{multline}\label{main_main_22} |E_{\frac{x}{\epsilon},\omega}(-\epsilon^2\int_0^{\tau^{\epsilon,n}_1}g(\epsilon X_s)\;ds, \tau^{\epsilon,n}_1\leq L_{n+2}^2)- \\ E_{\frac{x}{\epsilon},\omega}(-\epsilon^2\sum_{k=0}^{(\tau^{\epsilon,n}_1-L_{n-\overline{m}}^2)/L_{n-\overline{m}}^2}L_{n-\overline{m}}^2g(\epsilon X_{kL_{n-\overline{m}}^2}), \tau^{\epsilon,n}_1\leq L_{n+2}^2)|\leq \\ C\norm{g}_{L^\infty(\mathbb{R}^d)}(\epsilon L_{n-\overline{m}})^2E_{\frac{x}{\epsilon},\omega}(\sum_{k=0}^{(\tau^{\epsilon,n}_1-L_{n-\overline{m}}^2)/L_{n-\overline{m}}^2}P_{X_{kL_{n-\overline{m}}^2},\omega}(X^*_{L_{n-\overline{m}}^2}> \tilde{D}_{n-\overline{m}}),  \tau^{\epsilon,n}_1\leq L_{n+2}^2)+\\ C(\epsilon L_{n-\overline{m}})^2\norm{Dg}_{L^\infty(\mathbb{R}^d)}\epsilon \tilde{D}_{n-\overline{m}}E_{\frac{x}{\epsilon},\omega}(\sum_{k=0}^{(\tau^{\epsilon,n}_1-L_{n-\overline{m}}^2)/L_{n-\overline{m}}^2}P_{X_{kL_{n-\overline{m}}^2},\omega}(X^*_{L_{n-\overline{m}}^2}\leq \tilde{D}_{n-\overline{m}}),  \tau^{\epsilon,n}_1\leq L_{n+2}^2).\end{multline}  And, therefore, Control \ref{main_localization}, the event $\left\{\tau^{\epsilon,n}_1\leq L_{n+2}^2\right\}$ and $L_n\leq \frac{1}{\epsilon}<L_{n+1}$ imply that, for $C>0$ independent of $n$, \begin{multline}\label{main_main_23} |E_{\frac{x}{\epsilon},\omega}(-\epsilon^2\int_0^{\tau^{\epsilon,n}_1}g(\epsilon X_s)\;ds, \tau^{\epsilon,n}_1\leq L_{n+2}^2)- \\ E_{\frac{x}{\epsilon},\omega}(-\epsilon^2\sum_{k=0}^{(\tau^{\epsilon,n}_1-L_{n-\overline{m}}^2)/L_{n-\overline{m}}^2}L_{n-\overline{m}}^2g(\epsilon X_{kL_{n-\overline{m}}^2}), \tau^{\epsilon,n}_1\leq L_{n+2}^2)|\leq \\ C\norm{g}_{L^\infty(\mathbb{R}^d)} (\frac{L_{n+2}}{L_n})^2\exp(-\tilde{\kappa}_{n-\overline{m}})+C\norm{Dg}_{L^\infty(\mathbb{R}^d)}\epsilon^3L_{n+2}^2\tilde{D}_{n-\overline{m}}.\end{multline}  Since the definitions (\ref{L}), (\ref{kappa}) and (\ref{D}), the choice of $\overline{m}$ in (\ref{main_mbar}) and $L_n\leq \frac{1}{\epsilon}<L_{n+1}$ ensure that, for $C>0$ independent of $n$, $$(\frac{L_{n+2}}{L_n})^2\exp(-\tilde{\kappa}_{n-\overline{m}})\leq C\epsilon^3L_{n+2}^2\tilde{D}_{n-\overline{m}}\leq CL_n^{-3+2(1+a)^2+(1+a)^{-\overline{m}}}\leq CL_n^{-5a},$$ the lefthand side of (\ref{main_main_23}) is bounded, for $C>0$ independent of $n$, by \begin{equation}\label{main_main_24} C(\norm{g}_{L^\infty(\mathbb{R}^d)}+\norm{Dg}_{L^\infty(\mathbb{R}^d)})L_{n}^{-5a},\end{equation} which completes the proof of (\ref{main_outline_3}).

Recall the stopping time $$T^{\epsilon,n}_1=\inf\left\{\;k\geq 0\;|\;(X_k,\overline{X}_k)\;\textrm{satisfies}\;d(X_k,(U/\epsilon)^c)\leq \tilde{D}_{n-\overline{m}}\;\right\},$$ which is the discrete version of $\tau^{\epsilon,n}_1$ defined for the first coordinate of the process $(X_k, \overline{X}_k)$ described by the measure $Q_{n,\frac{x}{\epsilon}}$ constructed in Section \ref{section_coupling}.  The definition of $Q_{n,\frac{x}{\epsilon}}$ and the Markov property imply that \begin{multline}\label{main_main_25} E_{\frac{x}{\epsilon},\omega}(-\epsilon^2\sum_{k=0}^{(\tau^{\epsilon,n}_1-L_{n-\overline{m}}^2)/L_{n-\overline{m}}^2}L_{n-\overline{m}}^2g(\epsilon X_{kL_{n-\overline{m}}^2}), \tau^{\epsilon,n}_1\leq L_{n+2}^2)= \\ E^{Q_{n,\frac{x}{\epsilon}}}(-\epsilon^2\sum_{k=0}^{T^{\epsilon,n}_1-1}L_{n-\overline{m}}^2g(\epsilon X_k), T^{\epsilon,n}_1\leq(\frac{L_{n+2}}{L_{n-\overline{m}}})^2),\end{multline} and therefore, to recap the progress to this point, in combination (\ref{main_main_1}), (\ref{main_main_4}), (\ref{main_main_17}), (\ref{main_main_21}) (\ref{main_main_24}) imply that, for $C>0$ independent of $n$, \begin{multline}\label{main_main_26} \abs{u^\epsilon(x)-E^{Q_{n,\frac{x}{\epsilon}}}(-\epsilon^2\sum_{k=0}^{T^{\epsilon,n}_1-1}L_{n-\overline{m}}^2g(\epsilon X_k), T^{\epsilon,n}_1\leq(\frac{L_{n+2}}{L_{n-\overline{m}}})^2)}\leq \\ C\norm{g}_{L^\infty(\overline{U})}(\frac{L_{n+1}}{L_{n+2}}+(\frac{L_{n-1}}{L_n})^2+L_n^{4a+2a^2-\frac{10a}{1+a}}+L_n^{-5a})+C\norm{Dg}_{L^\infty(\mathbb{R}^d)}L_n^{-5a}.\end{multline}  This estimate effectively proves the efficacy of the discrete approximation scheme.  The next step in the proof will follow from the global coupling estimates established by Corollary \ref{couple_cor} and standard estimates for Brownian motion.

\emph{The Proof of (\ref{main_outline_5}).}  Let $Q_{n,\frac{x}{\epsilon}}$ denote the measure describing the Markov process $(X_k, \overline{X}_k)$ on $(\mathbb{R}^d\times\mathbb{R}^d)^\mathbb{N}$ constructed in Proposition \ref{couple_main}, and define $C_n$ to be the event $$C_n=(\abs{X_k-\overline{X}_k}\geq L_{n-\overline{m}}\;|\;\textrm{for some}\;0\leq k\leq 2(\frac{L_{n+2}}{L_{n-\overline{m}}})^2),$$ where Corollary \ref{couple_cor} asserts that, for $C>0$ independent of $n$, \begin{equation}\label{main_main_27} Q_{n,\frac{x}{\epsilon}}(C_n)\leq C\tilde{\kappa}_{n-\overline{m}}L_{n-\overline{m}}^{16a-\delta}.\end{equation}  The goal now is to estimate the expectation of the difference \begin{equation}\label{main_main_28} \abs{E^{Q_{n,\frac{x}{\epsilon}}}(-\epsilon^2\sum_{k=0}^{T^{\epsilon,n}_1-1}L_{n-\overline{m}}^2g(\epsilon X_k)-(-\epsilon^2\sum_{k=0}^{T^{\epsilon,n}_1-1}L_{n-\overline{m}}^2g(\epsilon \overline{X}_k)), T^{\epsilon,n}_1\leq(\frac{L_{n+2}}{L_{n-\overline{m}}})^2)}.\end{equation}  Form the decomposition with respect to the event $C_n$ and use the triangle inequality to obtain \begin{multline}\label{main_main_29} \abs{E^{Q_{n,\frac{x}{\epsilon}}}(-\epsilon^2\sum_{k=0}^{T^{\epsilon,n}_1-1}L_{n-\overline{m}}^2g(\epsilon X_k)-(-\epsilon^2\sum_{k=0}^{T^{\epsilon,n}_1-1}L_{n-\overline{m}}^2g(\epsilon \overline{X}_k)), T^{\epsilon,n}_1\leq(\frac{L_{n+2}}{L_{n-\overline{m}}})^2)}\leq \\ \abs{E^{Q_{n,\frac{x}{\epsilon}}}(-\epsilon^2\sum_{k=0}^{T^{\epsilon,n}_1-1}L_{n-\overline{m}}^2g(\epsilon X_k)-(-\epsilon^2\sum_{k=0}^{T^{\epsilon,n}_1-1}L_{n-\overline{m}}^2g(\epsilon \overline{X}_k)), T^{\epsilon,n}_1\leq(\frac{L_{n+2}}{L_{n-\overline{m}}})^2, C_n)}+ \\ \abs{E^{Q_{n,\frac{x}{\epsilon}}}(-\epsilon^2\sum_{k=0}^{T^{\epsilon,n}_1-1}L_{n-\overline{m}}^2g(\epsilon X_k)-(-\epsilon^2\sum_{k=0}^{T^{\epsilon,n}_1-1}L_{n-\overline{m}}^2g(\epsilon \overline{X}_k)), T^{\epsilon,n}_1\leq(\frac{L_{n+2}}{L_{n-\overline{m}}})^2, C_n^c)}.\end{multline}

The first term of (\ref{main_main_29}) is bounded using (\ref{main_main_27}) and the event $\left\{T^{\epsilon,n}_1\leq(\frac{L_{n+2}}{L_{n-\overline{m}}})^2)\right\}$, which imply, for $C>0$ independent of $n$, using the definitions of $L_n$ in (\ref{L}) and $\tilde{\kappa}_{n-\overline{m}}$ in (\ref{kappa}), \begin{multline}\label{main_main_30} \abs{E^{Q_{n,\frac{x}{\epsilon}}}(-\epsilon^2\sum_{k=0}^{T^{\epsilon,n}_1-1}L_{n-\overline{m}}^2g(\epsilon X_k)-(-\epsilon^2\sum_{k=0}^{T^{\epsilon,n}_1-1}L_{n-\overline{m}}^2g(\epsilon \overline{X}_k)), T^{\epsilon,n}_1\leq(\frac{L_{n+2}}{L_{n-\overline{m}}})^2, C_n)}\leq \\ C\norm{g}_{L^\infty(\mathbb{R}^d)}(\epsilon L_{n+2})^2\tilde{\kappa}_{n-\overline{m}}L_{n-\overline{m}}^{16a-\delta}\leq C\norm{g}_{L^\infty(\mathbb{R}^d)}L_n^{21a-\frac{\delta}{2}}, \end{multline} where the definitions (\ref{Holderexponent}) and (\ref{delta}) imply that the exponent $$21a-\frac{\delta}{2}<0$$ is negative.

The second term of (\ref{main_main_29}) is bounded using the Lipschitz continuity of $g$, the event bounding $T^{\epsilon,n}_1$ and the definition of $C_n^c$.  Namely, for $C>0$ independent of $n$, \begin{multline}\label{main_main_31}\abs{E^{Q_{n,\frac{x}{\epsilon}}}(-\epsilon^2\sum_{k=0}^{T^{\epsilon,n}_1-1}L_{n-\overline{m}}^2g(\epsilon X_k)-(-\epsilon^2\sum_{k=0}^{T^{\epsilon,n}_1-1}L_{n-\overline{m}}^2g(\epsilon \overline{X}_k)), T^{\epsilon,n}_1\leq(\frac{L_{n+2}}{L_{n-\overline{m}}})^2, C_n^c)}\leq \\ C\norm{Dg}_{L^\infty(\mathbb{R}^d)}(\epsilon L_{n+2})^2\epsilon L_{n-\overline{m}}\leq C\norm{Dg}_{L^\infty(\mathbb{R}^d)}L_n^{-5a},\end{multline} where the final inequality is obtained as in the arguments leading from (\ref{main_main_23}) to (\ref{main_main_24}).  Therefore, in view of (\ref{main_main_29}), estimates (\ref{main_main_30}) and (\ref{main_main_31}) combine to form the estimate, for $C>0$ independent of $n$, \begin{multline}\label{main_main_32} \abs{E^{Q_{n,\frac{x}{\epsilon}}}(-\epsilon^2\sum_{k=0}^{T^{\epsilon,n}_1-1}L_{n-\overline{m}}^2g(\epsilon X_k)-(-\epsilon^2\sum_{k=0}^{T^{\epsilon,n}_1-1}L_{n-\overline{m}}^2g(\epsilon \overline{X}_k)), T^{\epsilon,n}_1\leq(\frac{L_{n+2}}{L_{n-\overline{m}}})^2)}\leq \\ C\norm{g}_{L^\infty(\mathbb{R}^d)}L_n^{21a-\frac{\delta}{2}}+C\norm{Dg}_{L^\infty(\mathbb{R}^d)}L_n^{-5a},\end{multline} and complete the proof of (\ref{main_outline_5}).

\emph{The proof of (\ref{main_outline_6}).}  Recall the discrete exit time $$\overline{T}^{\epsilon,n}_2=\inf\left\{\;k\geq 0\;|\;(X_k,\overline{X}_k)\;\;\textrm{satisfies}\;\;\d(\overline{X}_k,(U/\epsilon))\geq \tilde{D}_{n-\overline{m}}\right\},$$ which is acts as $\tau^{\epsilon,n}_2$ for the second coordinate of the process $(X_k, \overline{X}_k)$.  The purpose now is to replace $T^{\epsilon,n}_1$ with $\overline{T}^{\epsilon,n}_2$ for the second term of the difference (\ref{main_main_32}).  First, an upper bound is imposed for $\overline{T}^{\epsilon,n}_2$, and the difference is bounded, for $C>0$ independent of $n$, by \begin{multline}\label{main_main_33}|E^{Q_{n,\frac{x}{\epsilon}}}(-\epsilon^2\sum_{k=0}^{T^{\epsilon,n}_1-1}L_{n-\overline{m}}^2g(\epsilon \overline{X}_k), T^{\epsilon,n}_1\leq (\frac{L_{n+2}}{L_{n-\overline{m}}})^2)- \\ E^{Q_{n,\frac{x}{\epsilon}}}(-\epsilon^2\sum_{k=0}^{T^{\epsilon,n}_1-1}L_{n-\overline{m}}^2g(\epsilon \overline{X}_k), T^{\epsilon,n}_1\leq (\frac{L_{n+2}}{L_{n-\overline{m}}})^2, \overline{T}^{\epsilon,n}_2\leq (\frac{L_{n+2}}{L_{n-\overline{m}}})^2)|\leq \\ C\norm{g}_{L^\infty(\mathbb{R}^d)}(\epsilon L_{n+2})^2Q_{n,\frac{x}{\epsilon}}(\overline{T}^{\epsilon,n}_2>(\frac{L_{n+2}}{L_{n-\overline{m}}})^2).\end{multline}  And, since line (\ref{exit_main_7}) of Proposition \ref{exit_main} applies to the discrete sequence and stopping time after increasing $R$ to $2R$ and increasing the constant, together with the definition of $Q_{n,\frac{x}{\epsilon}}$ and the stopping times, for $C>0$ independent of $n$,\begin{equation}\label{main_main_34}Q_{n,\frac{x}{\epsilon}}(\overline{T}^{\epsilon,n}_2>(\frac{L_{n+2}}{L_{n-\overline{m}}})^2)=W^{n-\overline{m}}_{\frac{x}{\epsilon}}(\tau^{\epsilon,n}_2> L_{n+2}^2)\leq C(\epsilon L_{n+2})^{-3}.\end{equation}  Therefore, for $C>0$ independent of $n$, the lefthand side of (\ref{main_main_33}) is bounded by \begin{equation}\label{main_main_35} C\norm{g}_{L^\infty(\mathbb{R}^d)}(\epsilon L_{n+2})^{-1}\leq C\norm{g}_{L^\infty(\mathbb{R}^d)}\frac{L_{n+1}}{L_{n+2}}.\end{equation}  Note that a better estimate can be achieved in (\ref{main_main_34}) for Brownian motion, however any improvement at this stage will not improve the overall rate of homogenization.

The next step replaces $T^{\epsilon,n}_1$ in the sum with $\overline{T}^{\epsilon,n}_2$ following a decomposition in terms of the event $C_n$ and an application of the triangle inequality.  Using (\ref{main_main_27}) and the bounds for the exit times, on the event $C_n$ the expectation of the difference \begin{multline}\label{main_main_36}|E^{Q_{n,\frac{x}{\epsilon}}}(-\epsilon^2\sum_{k=0}^{T^{\epsilon,n}_1-1}L_{n-\overline{m}}^2g(\epsilon \overline{X}_k)-(-\epsilon^2\sum_{k=0}^{\overline{T}^{\epsilon,n}_2-1}L_{n-\overline{m}}^2g(\epsilon \overline{X}_k)), \\ T^{\epsilon,n}_1\leq (\frac{L_{n+2}}{L_{n-\overline{m}}})^2, \overline{T}^{\epsilon,n}_2\leq (\frac{L_{n+2}}{L_{n-\overline{m}}})^2, C_n)|\end{multline} is bounded, for $C>0$ independent of $n$, by \begin{equation}\label{main_main_37} C(\epsilon L_{n+2})^2\norm{g}_{L^\infty(\mathbb{R}^d)}Q_{n,\frac{x}{\epsilon}}(C_n)\leq C\norm{g}_{L^\infty(\mathbb{R}^d)}(\frac{L_{n+2}}{L_n})^2\tilde{\kappa}_{n-\overline{m}}L_{n-\overline{m}}^{16a-\delta}\leq C\norm{g}_{L^\infty(\mathbb{R}^d)}L_n^{21a-\frac{\delta}{2}},\end{equation} where the final inequality is obtained identically to (\ref{main_main_30}).

It follows immediately from the definitions that $T^{\epsilon,n}_1\leq \overline{T}^{\epsilon,n}_2$ on the event \begin{equation}\label{main_main_38}\left\{T^{\epsilon,n}_1\leq (\frac{L_{n+2}}{L_{n-\overline{m}}})^2, \overline{T}^{\epsilon,n}_2\leq (\frac{L_{n+2}}{L_{n-\overline{m}}})^2, C_n^c\right\}.\end{equation}  Therefore, on the event $C_n^c$, the expectation of the difference \begin{multline}\label{main_main_39}|E^{Q_{n,\frac{x}{\epsilon}}}(-\epsilon^2\sum_{k=0}^{T^{\epsilon,n}_1-1}L_{n-\overline{m}}^2g(\epsilon \overline{X}_k)-(-\epsilon^2\sum_{k=0}^{\overline{T}^{\epsilon,n}_2-1}L_{n-\overline{m}}^2g(\epsilon \overline{X}_k)), \\ T^{\epsilon,n}_1\leq (\frac{L_{n+2}}{L_{n-\overline{m}}})^2, \overline{T}^{\epsilon,n}_2\leq (\frac{L_{n+2}}{L_{n-\overline{m}}})^2, C_n^c)|\end{multline} is bounded by \begin{equation}\label{main_main_40} (\epsilon L_{n-\overline{m}})^2\norm{g}_{L^\infty(\mathbb{R}^d)}E^{Q_{n,\frac{x}{\epsilon}}}(\overline{T}^{\epsilon,n}_2-T^{\epsilon,n}_1, T^{\epsilon,n}_1\leq (\frac{L_{n+2}}{L_{n-\overline{m}}})^2, \overline{T}^{\epsilon,n}_2\leq (\frac{L_{n+2}}{L_{n-\overline{m}}})^2, C_n^c).\end{equation}  And, on the event (\ref{main_main_38}), the definitions of $T^{\epsilon,n}_1$ and $C_n^c$ imply $$d(X_{T^{\epsilon,n}_1}, (U/\epsilon)^c)\leq \tilde{D}_{n-\overline{m}}\;\;\textrm{which guarantees}\;\;\d(\overline{X}_{T^{\epsilon,n}_1},(U/\epsilon)^c)\leq 2\tilde{D}_{n-\overline{m}},$$ and owing to the definition of $Q_{n,\frac{x}{\epsilon}}$, the Markov property, standard exponential estimates for Brownian motion \cite[Chapter~2, Proposition~1.8]{RY} and Proposition \ref{disc_u_scale}, using the definitions of (\ref{L}), (\ref{kappa}) and (\ref{D}), observe that for $C>0$ independent of $n$, \begin{multline}\label{main_main_41}E^{Q_{n,\frac{x}{\epsilon}}}(\overline{T}^{\epsilon,n}_2-T^{\epsilon,n}_1, T^{\epsilon,n}_1\leq (\frac{L_{n+2}}{L_{n-\overline{m}}})^2, \overline{T}^{\epsilon,n}_2\leq (\frac{L_{n+2}}{L_{n-\overline{m}}})^2, C_n^c)\leq \\ \sup_{\d(y,(U/\epsilon)^c)\leq 2\tilde{D}_{n-\overline{m}}}E^{W^{n-\overline{m}}_y}(\frac{\tau^{\epsilon,n}_2}{L_{n-\overline{m}}^2})\leq C\frac{\tilde{\kappa}_{n-\overline{m}}}{\epsilon L_{n-\overline{m}}}.\end{multline}  Therefore, combining (\ref{main_main_39}), (\ref{main_main_40}) and (\ref{main_main_41}) and using $L_n\leq \frac{1}{\epsilon}<L_{n+1}$, for $C>0$ independent of $n$, \begin{multline}\label{main_main_42}|E^{Q_{n,\frac{x}{\epsilon}}}(-\epsilon^2\sum_{k=0}^{T^{\epsilon,n}_1-1}L_{n-\overline{m}}^2g(\epsilon \overline{X}_k)-(-\epsilon^2\sum_{k=0}^{\overline{T}^{\epsilon,n}_2-1}L_{n-\overline{m}}^2g(\epsilon \overline{X}_k)), \\ T^{\epsilon,n}_1\leq (\frac{L_{n+2}}{L_{n-\overline{m}}})^2, \overline{T}^{\epsilon,n}_2\leq (\frac{L_{n+2}}{L_{n-\overline{m}}})^2, C_n^c)|\leq C\norm{g}_{L^\infty(\mathbb{R}^d)}\frac{\tilde{D}_{n-\overline{m}}}{L_n}.\end{multline}  And, with (\ref{main_main_36}), (\ref{main_main_37}) and (\ref{main_main_42}), conclude using the triangle inequality that, for $C>0$ independent of $n$, \begin{multline}\label{main_main_43}|E^{Q_{n,\frac{x}{\epsilon}}}(-\epsilon^2\sum_{k=0}^{T^{\epsilon,n}_1-1}L_{n-\overline{m}}^2g(\epsilon \overline{X}_k)-(-\epsilon^2\sum_{k=0}^{\overline{T}^{\epsilon,n}_2-1}L_{n-\overline{m}}^2g(\epsilon \overline{X}_k)),T^{\epsilon,n}_1\leq (\frac{L_{n+2}}{L_{n-\overline{m}}})^2, \overline{T}^{\epsilon,n}_2\leq (\frac{L_{n+2}}{L_{n-\overline{m}}})^2)| \\ \leq C\norm{g}_{L^\infty(\mathbb{R}^d)}(L_n^{21a-\frac{\delta}{2}}+\frac{\tilde{D}_{n-\overline{m}}}{L_n}).\end{multline}

Finally, analogously to the arguments (\ref{main_main_33}) to (\ref{main_main_35}), for $C>0$ independent of $n$, \begin{multline}\label{main_main_44}|E^{Q_{n,\frac{x}{\epsilon}}}(-\epsilon^2\sum_{k=0}^{\overline{T}^{\epsilon,n}_2-1}L_{n-\overline{m}}^2g(\epsilon \overline{X}_k), T^{\epsilon,n}_1\leq (\frac{L_{n+2}}{L_{n-\overline{m}}})^2, \overline{T}^{\epsilon,n}_2\leq (\frac{L_{n+2}}{L_{n-\overline{m}}})^2)- \\ E^{Q_{n,\frac{x}{\epsilon}}}(-\epsilon^2\sum_{k=0}^{\overline{T}^{\epsilon,n}_2-1}L_{n-\overline{m}}^2g(\epsilon \overline{X}_k), \overline{T}^{\epsilon,n}_2\leq (\frac{L_{n+2}}{L_{n-\overline{m}}})^2)|\leq \\ C\norm{g}_{L^\infty(\mathbb{R}^d)}(\epsilon L_{n+2})^2Q_{n,\frac{x}{\epsilon}}(T^{\epsilon,n}_1>(\frac{L_{n+2}}{L_{n-\overline{m}}})^2).\end{multline}  Since the exponential estimates implied by Control \ref{main_localization} and Proposition \ref{exit_main}, and in particular line (\ref{exit_main_5}), yield, for $C>0$ independent of $n$, $$Q_{n,\frac{x}{\epsilon}}(T^{\epsilon,n}_1>(\frac{L_{n+2}}{L_{n-\overline{m}}})^2)\leq C(\epsilon L_{n+2})^{-3},$$ the lefthand side of (\ref{main_main_44}) is bounded, for $C>0$ independent of $n$, by \begin{equation}\label{main_main_45} C\norm{g}_{L^\infty(\mathbb{R}^d)}(\epsilon L_{n+2})^{-1}\leq \norm{g}_{L^\infty(\mathbb{R}^d)}\frac{L_{n+1}}{L_{n+2}}.\end{equation}  In total then, the collection (\ref{main_main_35}), (\ref{main_main_43}) and (\ref{main_main_45}) yield, for $C>0$ independent of $n$, \begin{multline}\label{main_main_46} |E^{Q_{n,\frac{x}{\epsilon}}}(-\epsilon^2\sum_{k=0}^{T^{\epsilon,n}_1-1}L_{n-\overline{m}}^2g(\epsilon \overline{X}_k), T^{\epsilon,n}_1\leq (\frac{L_{n+2}}{L_{n-\overline{m}}})^2)- \\ E^{Q_{n,\frac{x}{\epsilon}}}(-\epsilon^2\sum_{k=0}^{\overline{T}^{\epsilon,n}_2-1}L_{n-\overline{m}}^2g(\epsilon \overline{X}_k), \overline{T}^{\epsilon,n}_2\leq (\frac{L_{n+2}}{L_{n-\overline{m}}})^2)|\leq \\ C\norm{g}_{L^\infty(\mathbb{R}^d)}(L_n^{21a-\frac{\delta}{2}}+\frac{\tilde{D}_{n-\overline{m}}}{L_n}+\frac{L_{n+1}}{L_{n+2}}),\end{multline} and complete the proof of (\ref{main_outline_6}).

The definition of $Q_{n,\frac{x}{\epsilon}}$ and the Markov property imply \begin{multline}\label{main_main_47} E^{Q_{n,\frac{x}{\epsilon}}}(-\epsilon^2\sum_{k=0}^{\overline{T}^{\epsilon,n}_2-1}L_{n-\overline{m}}^2g(\epsilon \overline{X}_k), \overline{T}^{\epsilon,n}_2\leq (\frac{L_{n+2}}{L_{n-\overline{m}}})^2)= \\ E^{W^{n-\overline{m}}_{\frac{x}{\epsilon}}}(-\epsilon^2\sum_{k=0}^{(\tau^{\epsilon,n}_2-L_{n-\overline{m}}^2)/L_{n-\overline{m}}^2}L_{n-\overline{m}}^2g(\epsilon X_{kL_{n-\overline{m}}^2}), \tau^{\epsilon,n}_2\leq L_{n+2}^2).\end{multline}  Therefore, to recap the progress, the collection of estimates (\ref{main_main_26}), (\ref{main_main_32}), (\ref{main_main_46}) and (\ref{main_main_47}) produce the bound, for $C>0$ independent of $n$, \begin{multline}\label{main_main_48} \abs{u^\epsilon(x)-E^{W^{n-\overline{m}}_{\frac{x}{\epsilon}}}(-\epsilon^2\sum_{k=0}^{(\tau^{\epsilon,n}_2-L_{n-\overline{m}}^2)/L_{n-\overline{m}}^2}L_{n-\overline{m}}^2g(\epsilon X_{kL_{n-\overline{m}}^2}), \tau^{\epsilon,n}_2\leq L_{n+2}^2)}\leq \\ C\norm{g}_{L^\infty(\mathbb{R}^d)}(\frac{L_{n+1}}{L_{n+2}}+(\frac{L_{n-1}}{L_n})^2+L_n^{4a+2a^2-\frac{10a}{1+a}}+L_n^{21a-\frac{\delta}{2}}+\frac{\tilde{D}_{n-\overline{m}}}{L_n}+L_n^{-5a}) \\+ C\norm{Dg}_{L^\infty(\mathbb{R}^d)}L_n^{-5a}.\end{multline}  It remains to recover the integral with respect to Brownian motion from its discrete approximation.

\emph{The proof of (\ref{main_outline_8}).}  The proof follows, in reverse order, the arguments leading to the proof of (\ref{main_outline_5}) from (\ref{main_outline_1}).  Observe that, for $C>0$ independent of $n$, \begin{multline}\label{main_main_49}\abs{E^{W^{n-\overline{m}}_{\frac{x}{\epsilon}}}(-\epsilon^2\sum_{k=0}^{(\tau^{\epsilon,n}_2-L_{n-\overline{m}}^2)/L_{n-\overline{m}}^2}L_{n-\overline{m}}^2g(\epsilon X_{kL_{n-\overline{m}}^2})-(-\epsilon^2\int_0^{\tau^{\epsilon,n}_2}g(\epsilon X_s)\;ds), \tau^{\epsilon,n}_2\leq L_{n+2}^2)} \\ \leq C(\epsilon L_{n-\overline{m}})^2\norm{g}_{L^\infty(\mathbb{R}^d)}E^{W^{n-\overline{m}}_{\frac{x}{\epsilon}}}(\sum_{k=0}^{(\tau^{\epsilon,n}_2-L_{n-\overline{m}}^2)/L_{n-\overline{m}}^2}W^{n-\overline{m}}_{X_{kL_{n-\overline{m}}^2}}(X^*_{L_{n-\overline{m}}^2}>\tilde{D}_{n-\overline{m}}), \tau^{\epsilon,n}_2\leq L_{n+2}^2)+\\ (\epsilon L_{n-\overline{m}})^2\norm{Dg}_{L^\infty(\mathbb{R}^d)}\epsilon\tilde{D}_{n-\overline{m}}E^{W^{n-\overline{m}}_{\frac{x}{\epsilon}}}(\sum_{k=0}^{(\tau^{\epsilon,n}_2-L_{n-\overline{m}}^2)/L_{n-\overline{m}}^2}W^{n-\overline{m}}_{X_{kL_{n-\overline{m}}^2}}(X^*_{L_{n-\overline{m}}^2}\leq \tilde{D}_{n-\overline{m}}), \tau^{\epsilon,n}_2\leq L_{n+2}^2).\end{multline}  Standard exponential estimates for Brownian motion, see \cite[Chapter~2, Proposition~1.8]{RY}, imply (again, a better estimate is possible, but to no improvement of the rate) that the first term of (\ref{main_main_49}) is bounded, using $L_n\leq\frac{1}{\epsilon}<L_{n+1}$, for $C>0$ independent of $n$, by \begin{equation}\label{main_main_50} C\norm{g}_{L^\infty(\mathbb{R}^d)}(\epsilon L_{n+2})^2\exp(-\tilde{\kappa}_{n-\overline{m}})\leq C\norm{g}_{L^\infty(\mathbb{R}^d)}(\frac{L_{n+2}}{L_n})^2\exp(-\tilde{\kappa}_{n-\overline{m}}).\end{equation}  The $L^\infty$-estimate implied by the upper bound on $\tau^{\epsilon,n}_2$ ensures that the second term of (\ref{main_main_49}) is bounded, for $C>0$ independent of $n$, by \begin{equation}\label{main_main_51} C\norm{Dg}_{L^\infty(\mathbb{R}^d)}\epsilon^3L_{n+2}^2\tilde{D}_{n-\overline{m}}\leq C\norm{Dg}_{L^\infty(\mathbb{R}^d)}L_n^{-5a}, \end{equation} where the final inequality is obtained identically as in the arguments leading from (\ref{main_main_23}) to (\ref{main_main_24}).

In combination, lines (\ref{main_main_50}) and (\ref{main_main_51}) bound the lefthand side of (\ref{main_main_49}), for $C>0$ independent of $n$, by \begin{multline}\label{main_main_52}\abs{E^{W^{n-\overline{m}}_{\frac{x}{\epsilon}}}(-\epsilon^2\sum_{k=0}^{(\tau^{\epsilon,n}_2-L_{n-\overline{m}}^2)/L_{n-\overline{m}}^2}L_{n-\overline{m}}^2g(\epsilon X_{kL_{n-\overline{m}}^2})-(-\epsilon^2\int_0^{\tau^{\epsilon,n}_2}g(\epsilon X_s)\;ds), \tau^{\epsilon,n}_2\leq L_{n+2}^2)} \\ \leq C\norm{g}_{L^\infty(\mathbb{R}^d)}(\frac{L_{n+2}}{L_n})^2\exp(-\tilde{\kappa}_{n-\overline{m}})+C\norm{Dg}_{L^\infty(\mathbb{R}^d)}L_n^{-5a},\end{multline} and complete the proof of (\ref{main_outline_8}).

\emph{The proof of (\ref{main_outline_10}).}  Recall that $\tau^\epsilon$ denotes the exit time from $U/\epsilon$.  Observe by using Theorem \ref{effectivediffusivity}, line (\ref{exit_main_7}) from Proposition \ref{exit_main} and $L_n\leq\frac{1}{\epsilon}<L_{n+1}$ that, for $C>0$ independent of $n$, \begin{multline}\label{main_main_53} |E^{W^{n-\overline{m}}_{\frac{x}{\epsilon}}}(-\epsilon^2\int_0^{\tau^{\epsilon,n}_2}g(\epsilon X_s)\;ds, \tau^{\epsilon,n}_2\leq L_{n+2}^2)- \\ E^{W^{n-\overline{m}}_{\frac{x}{\epsilon}}}(-\epsilon^2\int_0^{\tau^{\epsilon,n}_2}g(\epsilon X_s)\;ds, \tau^{\epsilon,n}_2\leq L_{n+2}^2, \tau^\epsilon\leq L_{n+2}^2)|\leq \\ C\norm{g}_{L^\infty(\mathbb{R}^d)}(\epsilon L_{n+2})^2W^{n-\overline{m}}_{\frac{x}{\epsilon}}(\tau^\epsilon>L_{n+2}^2)\leq \\ C\norm{g}_{L^\infty(\mathbb{R}^d)}(\epsilon L_{n+2})^{-1}\leq C\norm{g}_{L^\infty(\mathbb{R}^d)}\frac{L_{n+1}}{L_{n+2}}.\end{multline}  Of course, estimate (\ref{main_main_53}) can be improved for Brownian motion, but what is written is sufficient and does not negatively effect the rate to be obtained in Section \ref{section_rate}.

Since the definitions imply $\tau^\epsilon\leq \tau^{\epsilon,n}_2$, the Markov property, Corollary \ref{disc_u_scale}, standard exponential estimates for Brownian motion, see  \cite[Chapter~2, Proposition~1.8]{RY}, and $L_n\leq \frac{1}{\epsilon}<L_{n+1}$ then bound the expectation of the difference, for $C>0$ independent of $n$, by \begin{multline}\label{main_main_55} \abs{E^{W^{n-\overline{m}}_{\frac{x}{\epsilon}}}(-\epsilon^2\int_{\tau^\epsilon}^{\tau^{\epsilon,n}_2}g(\epsilon X_s)\;ds, \tau^{\epsilon,n}_2\leq L_{n+2}^2, \tau^\epsilon\leq L_{n+2}^2)}\leq \\ \epsilon^2\norm{g}_{L^\infty(\mathbb{R}^d)}\sup_{y\in\partial U}E^{W^{n-\overline{m}}_y}(\tau^{\epsilon,n}_2) \leq C\norm{g}_{L^\infty(\mathbb{R}^d)}\epsilon\tilde{D}_{n-\overline{m}}\leq C\norm{g}_{L^\infty(\mathbb{R}^d)}\frac{\tilde{D}_{n-\overline{m}}}{L_n}.\end{multline}  Then, again using line (\ref{exit_main_7}) of Proposition \ref{exit_main} (after replacing $R$ with $2R$ and increasing the constant) and standard exponential estimates for Brownian motion, see \cite[Chapter~2, Proposition~1.8]{RY}, \begin{multline}\label{main_main_56} |E^{W^{n-\overline{m}}_{\frac{x}{\epsilon}}}(-\epsilon^2\int_0^{\tau^\epsilon}g(\epsilon X_s)\;ds, \tau^{\epsilon,n}_2\leq L_{n+2}^2, \tau^\epsilon\leq L_{n+2}^2)- \\ E^{W^{n-\overline{m}}_{\frac{x}{\epsilon}}}(-\epsilon^2\int_0^{\tau^\epsilon}g(\epsilon X_s)\;ds, \tau^\epsilon\leq L_{n+2}^2)|\leq \\ C\norm{g}_{L^\infty(\mathbb{R}^d)}(\epsilon L_{n+2})^2W^{n-\overline{m}}_{\frac{x}{\epsilon}}(\tau^{\epsilon,n}_2>L_{n+2}^2)\leq \\ C\norm{g}_{L^\infty(\mathbb{R}^d)}(\epsilon L_{n+2})^{-1}\leq C\norm{g}_{L^\infty(\mathbb{R}^d)}\frac{L_{n+1}}{L_{n+2}}.\end{multline}  As before, estimate (\ref{main_main_55}) is not optimal for Brownian motion, but is sufficient and does not negatively impact the rate.

It remains only to estimate the difference \begin{equation}\label{main_main_57} \abs{E^{W^{n-\overline{m}}_{\frac{x}{\epsilon}}}(-\epsilon^2\int_0^{\tau^\epsilon}g(\epsilon X_s)\;ds, \tau^\epsilon\leq L_{n+2}^2)-E^{W^{n-\overline{m}}_{\frac{x}{\epsilon}}}(-\epsilon^2\int_0^{\tau^\epsilon}g(\epsilon X_s)\;ds)}.\end{equation}  Since the control of the $\alpha_n$ implied by Theorem \ref{effectivediffusivity} implies Proposition \ref{exit_main} applies equally to Brownian motion (though, as before, a better estimate can be obtained, but to no effect on the rate), it follows as in (\ref{main_main_2}), for $C>0$ independent of $n$, \begin{multline}\label{main_main_58} \abs{E^{W^{n-\overline{m}}_{\frac{x}{\epsilon}}}(-\epsilon^2\int_0^{\tau^\epsilon}g(\epsilon X_s)\;ds, \tau^\epsilon\leq L_{n+2}^2)-E^{W^{n-\overline{m}}_{\frac{x}{\epsilon}}}(-\epsilon^2\int_0^{\tau^\epsilon}g(\epsilon X_s)\;ds)}\leq \\ C\epsilon^2\norm{g}_{L^\infty(\mathbb{R}^d)}\sum_{k=1}^\infty(k+1)L_{n+2}^2W^{n-\overline{m}}_{\frac{x}{\epsilon}}(\tau^\epsilon> kL_{n+2}^2)\leq \\ C\norm{g}_{L^\infty(\mathbb{R}^d)}(\epsilon L_{n+2})^{-1}\leq C\norm{g}_{L^\infty(\mathbb{R}^d)}\frac{L_{n+1}}{L_{n+2}}.\end{multline}  Therefore, in view of (\ref{main_main_55}), (\ref{main_main_56}) and (\ref{main_main_58}), for $C>0$ independent of $n$, \begin{multline}\label{main_main_59} \abs{E^{W^{n-\overline{m}}_{\frac{x}{\epsilon}}}(-\epsilon^2\int_0^{\tau^{\epsilon,n}_2}g(\epsilon X_s)\;ds, \tau^{\epsilon,n}_2\leq L_{n+2}^2)-E^{W^{n-\overline{m}}_{\frac{x}{\epsilon}}}(-\epsilon^2\int_0^{\tau^\epsilon}g(\epsilon X_s)\;ds)}\leq \\ C\norm{g}_{L^\infty(\mathbb{R}^d)}(\frac{\tilde{D}_{n-\overline{m}}}{L_{n}}+\frac{L_{n+1}}{L_{n+2}}),\end{multline} which completes the proof of (\ref{main_outline_10}).

\emph{Conclusion.}  Finally, writing $\overline{u}$ and $\overline{u}_{n-\overline{m}}$ for the respective solutions of (\ref{main_hom_1}) and (\ref{main_outline_11}), Proposition \ref{main_alpha} implies that, for $C>0$ independent of $n$, \begin{multline}\label{main_main_60}\abs{E^{W^{n-\overline{m}}_{\frac{x}{\epsilon}}}(-\epsilon^2\int_0^{\tau^\epsilon}g(\epsilon X_s)\;ds)-E^{W^\infty_{\frac{x}{\epsilon}}}(-\epsilon^2\int_0^{\tau^\epsilon}g(\epsilon X_s)\;ds)}=\abs{\overline{u}_{n-\overline{m}}(x)-\overline{u}(x)}\leq \\ C\norm{g}_{L^\infty(\overline{U})}L_{n-\overline{m}}^{-(1+\frac{9}{10})\delta}.\end{multline}  And, since there exists $C>0$ independent of $n\geq\overline{m}$ such that $$(\frac{L_{n+2}}{L_n})^2\exp(-\tilde{\kappa}_{n-\overline{m}})\leq CL_{n-\overline{m}}^{-(1+\frac{9}{10})\delta},$$ the combination of (\ref{main_main_48}), (\ref{main_main_52}), (\ref{main_main_59}) and (\ref{main_main_60}) results in the estimate, for $C>0$ independent of $n$, \begin{multline}\label{main_main_61}  \abs{u^\epsilon(x)-\overline{u}(x)}\leq \\ C\norm{g}_{L^\infty(\mathbb{R}^d)}(L_{n-\overline{m}}^{-(1+\frac{9}{10})\delta}+\frac{L_{n+1}}{L_{n+2}}+(\frac{L_{n-1}}{L_n})^2+L_n^{4a+2a^2-\frac{10a}{1+a}}+L_n^{21a-\frac{\delta}{2}}+\frac{\tilde{D}_{n-\overline{m}}}{L_n}+L_n^{-5a}) \\+ C\norm{Dg}_{L^\infty(\mathbb{R}^d)}L_n^{-5a}.\end{multline}  Since definitions (\ref{Holderexponent}), (\ref{L}), (\ref{D}) and (\ref{delta}) imply the righthand side of (\ref{main_main_61}) vanishes as $n$ approaches infinity, since $n$ approaches infinity and $\epsilon$ approaches zero, and since $\omega\in\Omega_0$ and $x\in \overline{U}$ were arbitrary, this completes the argument.  \end{proof}

It remains to extend Theorem \ref{main_main} to a general continuous righthand side, which follows from a standard extension argument.  \begin{equation}\label{main_continuous}\textrm{Assume}\;\;g\in\C(\overline{U}).\end{equation}  Notice that the approximation argument relies upon the result of Theorem \ref{main_main} for $g=-1$.  That is, it relies upon the fact that Theorem \ref{main_main} already contains an almost sure control of the exit time in expectation.

\begin{thm}\label{main_arbitrary} Assume (\ref{steady}), (\ref{constants}) and (\ref{main_continuous}).  For every $\omega\in \Omega_0$, the respective solutions $u^\epsilon$ and $\overline{u}$ of (\ref{main_eq_1}) and (\ref{main_hom_1}) satisfy $$\lim_{\epsilon\rightarrow 0}\norm{u^\epsilon-\overline{u}}_{L^\infty(\overline{U})}=0.$$\end{thm}

\begin{proof} Use the Tietze Extension Theorem, see for example Armstrong \cite[Theorem~2.15]{Armstrong}, to construct an extension with compact support $$\tilde{g}\in\BUC(\mathbb{R}^d)\;\;\textrm{satisfying}\;\;\tilde{g}|_{\overline{U}}=g.$$  By convolution construct, for each $\delta>0$, a $\tilde{g}^\delta\in\C^\infty_c(\mathbb{R}^d)$ such that $$\norm{\tilde{g}^\delta-\tilde{g}}_{L^\infty(\mathbb{R}^d)}\leq \delta,$$ and write $u^{\epsilon,\delta}$ for the solution $$\left\{\begin{array}{ll} \frac{1}{2}\tr(A(\frac{x}{\epsilon},\omega)D^2 u^{\epsilon,\delta})+\frac{1}{\epsilon}b(\frac{x}{\epsilon},\omega)\cdot Du^{\epsilon,\delta}=\tilde{g}^\delta(x) & \textrm{on}\;\;U, \\ u^{\epsilon,\delta}=0 & \textrm{on}\;\;\partial U.\end{array}\right.$$  Similarly, write $\overline{u}^\delta$ for the solution $$\left\{\begin{array}{ll} \frac{\overline{\alpha}}{2}\Delta\overline{u}^{\delta}=\tilde{g}^\delta(x) & \textrm{on}\;\;U, \\ \overline{u}^\delta=0 & \textrm{on}\;\;\partial U.\end{array}\right.$$

The representation formula for the solutions, the comparison principle and the triangle inequality imply that, writing $\tau^\epsilon$ for the exit time time from $U/\epsilon$ and $\tau_U$ for the exit time from $U$, for each $\omega\in \Omega$, $\delta>0$ and $\epsilon>0$, \begin{multline*}\norm{u^\epsilon-\overline{u}}_{L^\infty(\overline{U})}\leq \norm{u^\epsilon-u^{\epsilon,\delta}}_{L^\infty(\overline{U})}+\norm{u^{\epsilon,\delta}-\overline{u}^\delta}_{L^\infty(\overline{U})}+\norm{\overline{u}^\delta-\overline{u}}_{L^\infty(\overline{U})}\leq \\ \delta\sup_{x\in \overline{U}}E_{\frac{x}{\epsilon},\omega}(\epsilon^2\tau^\epsilon)+\delta\sup_{x\in \overline{U}}E^{W^\infty_x}(\tau_U)+\norm{u^{\epsilon,\delta}-\overline{u}^\delta}_{L^\infty(\overline{U})}.\end{multline*}  Therefore, since Theorem \ref{main_main} implies that $$\lim_{\epsilon\rightarrow 0}(\sup_{x\in\overline{U}}E_{\frac{x}{\epsilon},\omega}(\epsilon^2\tau^\epsilon))=\sup_{x\in \overline{U}}E^{W^\infty_x}(\tau_U),$$ and because $\tilde{g}^\delta$ satisfies the conditions of Theorem \ref{main_main}, for every $\omega\in\Omega_0$, for every $\delta>0$, $$\limsup_{\epsilon\rightarrow 0}\norm{u^\epsilon-\overline{u}}_{L^\infty(\overline{U})}\leq 2\delta \sup_{x\in \overline{U}}E^{W^\infty_x}(\tau_U)+\limsup_{\epsilon\rightarrow 0}\norm{u^{\epsilon,\delta}-\overline{u}^\delta}_{L^\infty(\overline{U})}=2\delta \sup_{x\in \overline{U}}E^{W^\infty_x}(\tau_U),$$ and this, since $\delta>0$ is arbitrary, completes the argument.  \end{proof}

The general homogenization statement for nonzero boundary data is now presented, after recalling the result of \cite{F3}.  The purpose will be to show that, on the event $\Omega_0$, solutions \begin{equation}\label{main_final_eq} \left\{\begin{array}{ll} \frac{1}{2}\tr(A(\frac{x}{\epsilon},\omega)D^2u^\epsilon)+\frac{1}{\epsilon}b(\frac{x}{\epsilon},\omega)\cdot Du^\epsilon=g(x) & \textrm{on}\;\;U, \\ u^\epsilon=f(x) & \textrm{on}\;\;\partial U,\end{array}\right.\end{equation} converge, as $\epsilon\rightarrow 0$ and uniformly on $\overline{U}$, to the solution \begin{equation}\label{main_final_hom} \left\{\begin{array}{ll} \frac{\overline{\alpha}}{2}\Delta \overline{u}=g(x) & \textrm{on}\;\;U, \\ \overline{u}=f(x) & \textrm{on}\;\;\partial U,\end{array}\right.\end{equation} whenever the righthand side and boundary data are continuous.  \begin{equation}\label{main_final_continuous} \textrm{Assume}\;g\in\C(\overline{U})\;\textrm{and}\;f\in \C(\partial U).\end{equation}  Notice that, in the case $g=0$, the variance $\overline{\alpha}$ does not effect the exit distribution because it reflects only a time change of the underlying Brownian motion.  Or, in terms of the equation, for each $n\geq 0$, the solution to the approximate homogenized problem \begin{equation}\label{main_final_approx} \left\{\begin{array}{ll} \frac{\alpha_n}{2}\Delta \overline{u}_n=g(x) & \textrm{on}\;\;U, \\ \overline{u}_n=f(x) & \textrm{on}\;\;\partial U,\end{array}\right.\end{equation} satisfies (\ref{main_final_hom}) whenever $g=0$.

\begin{prop}\label{main_bull}  Assume (\ref{steady}), (\ref{constants}) and $g=0$.  For each $n\geq 0$, for $\overline{u}$ and $\overline{u}_n$ the respective solutions of (\ref{main_final_eq}) and (\ref{main_final_approx}), $$\overline{u}=\overline{u}_n\;\;\textrm{on}\;\;\overline{U}.$$\end{prop}

The following Theorem is then an immediate consequence of \cite[Theorem~7.5]{F3}, since the event on which the statement was obtained in \cite{F3} contains the $\Omega_0$ defined in (\ref{main_mainevent}) as a subset.

\begin{thm}\label{main_close}  Assume (\ref{steady}), (\ref{constants}), (\ref{main_final_continuous}) and $g=0$.  For every $\omega\in\Omega_0$, the respective solutions $u^\epsilon$ and $\overline{u}$ of (\ref{main_final_eq}) and (\ref{main_final_hom}) satisfy $$\lim_{\epsilon\rightarrow 0}\norm{u^\epsilon-\overline{u}}_{L^\infty(\overline{U})}=0.$$\end{thm}

The final theorem is then an immediate consequence of Theorem \ref{main_main}, Theorem \ref{main_close} and linearity.

\begin{thm}\label{main_end}  Assume (\ref{steady}), (\ref{constants}), (\ref{main_final_continuous}).  For every $\omega\in\Omega_0$, the respective solutions $u^\epsilon$ and $\overline{u}$ of (\ref{main_final_eq}) and (\ref{main_final_hom}) satisfy $$\lim_{\epsilon\rightarrow 0}\norm{u^\epsilon-\overline{u}}_{L^\infty(\overline{U})}=0.$$\end{thm}

\section{The Rate of Homogenization}\label{section_rate}

An algebraic rate for the convergence established in Theorem \ref{main_end} is now obtained.  The result will be shown first for boundary data which is the restriction of a bounded, uniformly continuous function and interior data which is the restriction of a bounded, Lipschitz function.  \begin{equation}\label{rate_restrict} \textrm{Assume}\;\;f\in\BUC(\mathbb{R}^d)\;\;\textrm{and}\;\;g\in\Lip(\mathbb{R}^d).\end{equation}  The moduli of continuity will be denoted $\sigma_f$ and $Dg$.  Namely, for each $x,y\in\mathbb{R}^d$, \begin{equation}\label{rate_moduli} \abs{f(x)-f(y)}\leq\sigma_f(\abs{x-y})\;\;\textrm{and}\;\;\abs{g(x)-g(y)}\leq \norm{Dg}_{L^\infty(\mathbb{R}^d)}\abs{x-y}.\end{equation}  A rate for the convergence in the case $g=0$ was established in \cite[Theorem~8.1]{F3}.

\begin{thm}\label{rate_old}  Assume (\ref{steady}), (\ref{constants}), (\ref{rate_restrict}) and $g=0$.  There exists $C>0$ and $c_1,c_2>0$ such that, for every $\omega\in\Omega_0$, for all $\epsilon>0$ sufficiently small depending on $\omega$, the respective solutions $u^\epsilon$ and $\overline{u}$ of (\ref{main_final_eq}) and (\ref{main_final_hom}) satisfy $$\norm{u^\epsilon-\overline{u}}_{L^\infty(\overline{U})}\leq C\norm{f}_{L^\infty(\mathbb{R}^d)} \epsilon^{c_1}+C\sigma_f(\epsilon^{c_2}).$$\end{thm}

The following establishes a similar result in the case $f=0$, and follows quickly from the analysis carried out in Theorem \ref{main_main}.

\begin{thm}\label{rate_new}  Assume (\ref{steady}), (\ref{constants}), (\ref{rate_restrict}) and $f=0$.  There exists $C>0$ and $c_3,c_4>0$ such that, for every $\omega\in\Omega_0$, for all $\epsilon>0$ sufficiently small depending on $\omega$, the respective solutions $u^\epsilon$ and $\overline{u}$ of (\ref{main_final_eq}) and (\ref{main_final_hom}) satisfy $$\norm{u^\epsilon-\overline{u}}_{L^\infty(\overline{U})}\leq C\norm{g}_{L^\infty(\mathbb{R}^d)} \epsilon^{c_3}+C\norm{Dg}_{L^\infty(\mathbb{R}^d)}\epsilon^{c_4}.$$\end{thm}

\begin{proof}  Fix $\omega\in\Omega_0$ and $n_1\geq \overline{m}$ such that, for all $n\geq n_1$, for $r_0$ the constant quantifying the exterior ball condition, $$2\tilde{D}_{n-\overline{m}}<\frac{r_0L_n}{2}\;\;\textrm{and}\;\;\omega\in A_n.$$  Then, fix $\epsilon_1>0$ sufficiently small so that, whenever $0<\epsilon<\epsilon_1$ satisfies $L_n\leq \frac{1}{\epsilon}<L_{n+1}$ it follows that $n\geq n_1$.  Furthermore, using the boundedness of the domain $U$, choose $0<\epsilon_2<\epsilon_1$ such that, whenever $0<\epsilon<\epsilon_2$ and $L_n\leq \frac{1}{\epsilon}<L_{n+1}$, $$L_{n+1}U\subset [-\frac{1}{2}L_{n+2}^2, \frac{1}{2}L_{n+2}^2]^d.$$  These conditions guarantee, whenever $0<\epsilon<\epsilon_2$, the conclusion of line (\ref{main_main_61}) of Theorem \ref{main_main}.

Precisely, for every $0<\epsilon<\epsilon_2$ satisfying $L_n\leq \frac{1}{\epsilon}<L_{n+1}$, for $C>0$ independent of $n$, \begin{multline}\label{rate_new_1}  \norm{u^\epsilon(x)-\overline{u}(x)}_{L^\infty(\overline{U})}\leq \\ C\norm{g}_{L^\infty(\mathbb{R}^d)}(L_{n-\overline{m}}^{-(1+\frac{9}{10})\delta}+\frac{L_{n+1}}{L_{n+2}}+(\frac{L_{n-1}}{L_n})^2+L_n^{4a+2a^2-\frac{10a}{1+a}}+L_n^{21a-\frac{\delta}{2}}+\frac{\tilde{D}_{n-\overline{m}}}{L_n}+L_n^{-5a}) \\+ C\norm{Dg}_{L^\infty(\mathbb{R}^d)}L_n^{-5a}.\end{multline}  The definitions (\ref{Holderexponent}), (\ref{L}), (\ref{D}) imply that, since $L_n\leq \frac{1}{\epsilon}<L_{n+1}$, there exists $c_3>0$ satisfying, for $C>0$ independent of $n$, \begin{equation}\label{rate_new_2}L_{n-\overline{m}}^{-(1+\frac{9}{10})\delta}+\frac{L_{n+1}}{L_{n+2}}+(\frac{L_{n-1}}{L_n})^2+L_n^{4a+2a^2-\frac{10a}{1+a}}+L_n^{21a-\frac{\delta}{2}}+\frac{\tilde{D}_{n-\overline{m}}}{L_n}+L_n^{-5a}\leq C\epsilon^{c_3}.\end{equation}  And, definitions (\ref{Holderexponent}) and (\ref{L}) imply, using $L_n\leq \frac{1}{\epsilon}<L_{n+1}$, the existence of $c_4>0$ satisfying, for $C>0$ independent of $n$,  \begin{equation}\label{rate_new_3} L_n^{-5a}\leq C\epsilon^{c_4}.\end{equation}

Therefore, in combination (\ref{rate_new_1}), (\ref{rate_new_2}) and (\ref{rate_new_3}) yield, for all $0<\epsilon<\epsilon_2$, for $C>0$ independent of $0<\epsilon<\epsilon_2$,
$$\norm{u^\epsilon(x)-\overline{u}(x)}_{L^\infty(\overline{U})}\leq C\norm{g}_{L^\infty(\mathbb{R}^d)} \epsilon^{c_3}+C\norm{Dg}_{L^\infty(\mathbb{R}^d)}\epsilon^{c_4},$$ which completes the argument.  \end{proof}

The following statement establishes an algebraic rate of convergence for boundary and interior data which are respectively the restrictions of a bounded, uniformly continuous function and a bounded, Lipschitz function.  This requirement is removed for smooth domains in Theorem \ref{rate_smooth_domain}.  The proof follows immediately from Theorem \ref{rate_old}, Theorem \ref{rate_new} and linearity.

\begin{thm}\label{rate_main}  Assume (\ref{steady}), (\ref{constants}) and (\ref{rate_restrict}).  There exists $C>0$ and $c_1, c_2, c_3,c_4>0$ such that, for every $\omega\in\Omega_0$, for all $\epsilon>0$ sufficiently small depending on $\omega$, the respective solutions $u^\epsilon$ and $\overline{u}$ of (\ref{main_final_eq}) and (\ref{main_final_hom}) satisfy $$\norm{u^\epsilon-\overline{u}}_{L^\infty(\overline{U})}\leq C(\norm{f}_{L^\infty(\mathbb{R}^d)}\epsilon^{c_1}+\sigma_f(\epsilon^{c_2})+\norm{g}_{L^\infty(\mathbb{R}^d)} \epsilon^{c_3}+\norm{Dg}_{L^\infty(\mathbb{R}^d)}\epsilon^{c_4}).$$\end{thm}

Theorem \ref{rate_main} is now extended to general smooth domains up to a domain dependent factor.  Observe that, in the case of the ball $U=B_r$, it follows by an explicit radial extension or, in the case that the domain $U$ is smooth, it follows from the Product Neighborhood Theorem, see Milnor \cite[Page~46]{Milnor}, that every continuous function $f\in\C(\partial U)$ and Lipschitz function $g\in\Lip(\overline{U})$ admit extensions $$\tilde{f}\in\BUC(\mathbb{R}^d)\;\;\textrm{and}\;\;\tilde{g}\in\Lip(\mathbb{R}^d)$$ satisfying, for a constant $C=C(U)$ depending only upon the domain, $$\sigma_{\tilde{f}}(s)\leq \sigma_f(Cs)\;\;\textrm{for all}\;\;s\geq 0\;\;\textrm{sufficiently small,}$$ and $$\abs{D\tilde{g}}\leq C\abs{Dg}\;\;\textrm{in a neighborhood of}\;\;U.$$  That is, for smooth domains, assumption (\ref{rate_restrict}) can always be achieved up to a domain dependent factor.  \begin{equation}\label{rate_u}\textrm{Assume}\;f\in\C(\partial U), g\in\Lip(\overline{U})\;\textrm{and that the domain}\;U\;\textrm{is smooth}.\end{equation}  The following statement is then an immediate consequence of Theorem \ref{rate_main} and the preceding remarks.

\begin{thm}\label{rate_smooth_domain}  Assume (\ref{steady}), (\ref{constants}) and (\ref{rate_u}).  There exists $C>0$, $c_1, c_2, c_3, c_4>0$ and $C_1=C_1(U)>0$ such that, for every $\omega\in\Omega_0$, for all $\epsilon>0$ sufficiently small depending on $\omega$, the respective solutions $u^\epsilon$ and $\overline{u}$ of (\ref{main_final_eq}) and (\ref{main_final_hom}) satisfy $$\norm{u^\epsilon-\overline{u}}_{L^\infty(\overline{U})}\leq C(\norm{f}_{L^\infty(\overline{U})}\epsilon^{c_1}+\sigma_f(C_1\epsilon^{c_2})+\norm{g}_{L^\infty(\overline{U})} \epsilon^{c_3}+C_1\norm{Dg}_{L^\infty(\overline{U})}\epsilon^{c_4}).$$\end{thm}

\subsection*{Acknowledgments}
I would like to thank Professors Panagiotis Souganidis and Ofer Zeitouni for many useful conversations.

\bibliography{Exit}
\bibliographystyle{plain}

\end{document}